\documentclass[10pt]{amsart}

\date{May 16, 2024}

\usepackage{color}
\usepackage{enumerate}

\newcommand{\Bc}{\mathcal{B}}

\newtheorem{thm}{Theorem}[section]
\newtheorem{prop}[thm]{Proposition}
\newtheorem{lem}[thm]{Lemma}
\newtheorem{rem}[thm]{Remark}

\newtheorem{fact}[thm]{Fact}
\newtheorem{defi}[thm]{Definition}

\begin{document}

\title[Endpoint Schatten class properties of commutators]{Endpoint Schatten class properties\\ of commutators}

\author{Rupert L. Frank}
\address[Rupert L. Frank]{Mathe\-matisches Institut, Ludwig-Maximilans Universit\"at M\"unchen, The\-re\-sien\-str.~39, 80333 M\"unchen, Germany, and Munich Center for Quantum Science and Technology, Schel\-ling\-str.~4, 80799 M\"unchen, Germany, and Mathematics 253-37, Caltech, Pasa\-de\-na, CA 91125, USA}
\email{r.frank@lmu.de}

\author{Fedor Sukochev}
\address[Fedor Sukochev]{School of Mathematics and Statistics, University of New South Wales, Kensington, 2052, NSW, Australia}
\email{f.sukochev@unsw.edu.au}

\author{Dmitriy Zanin}
\address[Dmitriy Zanin]{School of Mathematics and Statistics, University of New South Wales, Kensington, 2052, NSW, Australia}
\email{d.zanin@unsw.edu.au}

\begin{abstract}
	We study trace ideal properties of the commutators $[(-\Delta)^{\frac{\epsilon}{2}},M_f]$ of a power of the Laplacian with the multiplication operator by a function $f$ on $\mathbb R^d$. For a certain range of $\epsilon\in\mathbb R$, we show that this commutator belongs to the weak Schatten class $\mathcal L_{\frac d{1-\epsilon},\infty}$ if and only if the distributional gradient of $f$ belongs to $L_{\frac d{1-\epsilon}}$. Moreover, in this case we determine the asymptotics of the singular values. Our proofs use, among other things, the tool of Double Operator Integrals.
\end{abstract}

\maketitle

\section{Introduction and main results}

\subsection{Background and motivations}

In this paper we are interested in spectral properties of commutators of the form
$$
[(-\Delta)^{\frac{\epsilon}{2}},M_f]
\qquad\text{in}\ L_2(\mathbb R^d) \,.
$$
Here $\epsilon$ is a parameter that we typically fix in one of the intervals $(0,1]$ and $(-\tfrac d2,0)$, and $M_f$ is the operator of multiplication by a function $f\in L_{1,{\rm loc}}(\mathbb R^d)$. In Section~\ref{sec:preliminary} we will recall that the operator $[(-\Delta)^{\frac{\epsilon}{2}},M_f]$ is always well defined from $C_c^\infty(\mathbb R^d)$ to the the space of distributions $(C_c^\infty(\mathbb R^d))'$.

The spectral properties that we are interested in are boundedness, compactness and membership to a trace ideal, as well as the computation of the asymptotics of its singular values. The trace ideals in question are Schatten classes $\mathcal L_p$ and weak Schatten classes $\mathcal L_{p,\infty}$, whose definition we will recall in Section \ref{prelim section}. 

Our goal is to investigate these spectral properties under minimal assumptions on the function $f$. In particular, we will derive the asymptotics of the singular values under the sole assumption that the asymptotic coefficient is finite. We will also be interested in converse theorems, where spectral properties of the commutators imply certain properties of the function $f$.

To set the stage, let us recall a prototypical result in this area, due to Coifman, Rochberg and Weiss \cite{CoRoWe} with later contributions by Janson \cite{Ja} and Uchiyama \cite{Uc} and many others. In a certain sense this result corresponds to the case $\epsilon=0$ of our problem. Let $d\geq 2$ and let $R_j = -i\partial_j (-\Delta)^{-\frac12}$, $j=1,\ldots,d$, be a Riesz transform. Then
$$
[R_j,M_f]
\qquad\text{in}\ L_2(\mathbb R^d)
$$
is bounded if and only if $f\in BMO(\mathbb R^d)$, the class of functions of bounded mean oscillations. In addition, there is a two-sided bound between the operator norm of $[R_j,M_f]$ and the seminorm of $f$ in $BMO(\mathbb R^d)$. Furthermore, the operator is compact if and only if $f\in VMO(\mathbb R^d)$, the class of functions of vanishing mean oscillations. Concerning trace ideals, Janson and Wolff \cite{JaWo} characterized membership of $[R_j,M_f]$ to Schatten classes $\mathcal L_p$ with $d<p<\infty$ in terms of membership of $f$ to a homogeneous fractional Sobolev space, and they showed that the operator does not belong to $\mathcal L_d$ unless $f$ is constant; for alternative proofs see \cite{RoSe2,JaPe}. Connes, Sullivan and Teleman, together with Semmes, \cite{CoSuTe} added a characterization in the endpoint case of the weak Schatten class $\mathcal L_{d,\infty}$, namely membership to a certain homogeneous first order Sobolev space. In the recent papers \cite{LMSZ,FrSuZa} we have revisited and extended the latter result.

Thus, there is a scale of nested function spaces, parametrized by $p$, namely the homogeneous fractional Sobolev spaces, such that if the regularity of $f$ improves (in the sense of it belonging to a smaller one of these spaces), then the trace ideal properties of $[R_j,M_f]$ improve (in the sense of it belonging to a smaller trace ideal space). At the same time, there is a saturation effect in the sense that if the regularity properties of $f$ improve beyond membership to the corresponding first order Sobolev space, then the trace ideal properties no longer improve. The existence of a critical exponent $p_*$ with this property, here $p_*=d$, is sometimes referred to as the `Janson--Wolff phenomenon' or `cut-off phenomenon'. It has been observed in many other problems of this kind, involving commutator or Hankel operators, see for instance, \cite{ArFiPe,FeRo,ArFiJaPe,PeRoWu,JaRo,EnRo} and references therein.

The present paper is a continuation of \cite{FrSuZa}, but can be read independently. Our goal is to prove the analogue of the before-mentioned results for the operators $[(-\Delta)^{\frac{\epsilon}{2}},M_f]$ for $\epsilon\neq 0$. Several results about these operators were obtained by Murray \cite{Mu} (concerning boundedness) and by Janson and Peetre \cite{JaPe} (concerning boundedness and trace ideal properties) and we review them in detail later on in this introduction. Here we just mention that for every $\epsilon\in (-\frac d2,0)\cup(0,1)$ there is a Janson--Wolff phenomenon with critical exponent $p_*=\frac{d}{1-\epsilon}$. The existing results characterize membership to $\mathcal L_p$ for $p>p_*$ in terms of $f$ belonging to some homogeneous fractional Sobolev space. It is also known that $[(-\Delta)^{\frac{\epsilon}{2}},M_f]$ does not belong to $\mathcal L_{p_*}$ unless $f$ is constant. Our new results characterize membership to the endpoint space, namely the weak Schatten ideal $\mathcal L_{p_*,\infty}$, in terms of $f$ belonging to some homogeneous first order Sobolev space. Moreover, we will compute the asymptotics of the singular values under the sole assumption of membership to this Sobolev space. 

We will also obtain some results in the case $\epsilon=1$, which is somewhat different. Note that the exponent $p_*=\frac{d}{1-\epsilon}$ tends to $+\infty$ as $\epsilon\to 1^-$. Among other things, we show that $[(-\Delta)^{\frac{1}{2}},M_f]$ is never compact, unless $f$ is constant. Moreover, we provide a short proof that a well-known criterion of Calder\'on \cite{Ca} for boundedness is not only sufficient, but also necessary.


\subsection{Main results}

We now come to a precise formulation of our main results in the following three theorems. The necessary notation related to function spaces and trace ideals, as well as the precise meaning of the operator $[(-\Delta)^{\frac{\epsilon}{2}},M_f]$ can be found in the following two sections.

Our first result concerns the case $0<\epsilon<1$.

\begin{thm}\label{main theorem positive} Let $d\geq 1$ and let $0<\epsilon<1.$
\begin{enumerate}[{\rm (i)}]
\item\label{mtpa} If $f\in\dot{W}^{1}_{\frac{d}{1-\epsilon}}(\mathbb{R}^d),$ then $[(-\Delta)^{\frac{\epsilon}{2}},M_f]\in\mathcal{L}_{\frac{d}{1-\epsilon},\infty}$ and
$$\Big\|[(-\Delta)^{\frac{\epsilon}{2}},M_f]\Big\|_{\frac{d}{1-\epsilon},\infty}\leq c_{d,\epsilon}\|f\|_{\dot{W}^{1}_{\frac{d}{1-\epsilon}}(\mathbb{R}^d)}.$$
\item\label{mtpb} If $f\in\dot{W}^{1}_{\frac{d}{1-\epsilon}}(\mathbb{R}^d)$, then
$$\lim_{t\to\infty}t^{\frac{1-\epsilon}{d}}\mu\Big(t,[(-\Delta)^{\frac{\epsilon}{2}},M_f]\Big)=\kappa_{d,\epsilon} \|f\|_{\dot{W}^{1}_{\frac{d}{1-\epsilon}}(\mathbb{R}^d)}$$
with
\begin{equation}
	\label{eq:constasymp}
	\kappa_{d,\epsilon} := |\epsilon| \left( (2\pi)^{-d} d^{-1} \int_{\mathbb S^{d-1}} |\omega_d|^\frac{d}{1-\epsilon}\,d\omega \right)^\frac{1-\epsilon}{d} \,.
\end{equation}
\item\label{mtpc} If $f\in L_{1,{\rm loc}}(\mathbb{R}^d)$ and if 
$[(-\Delta)^{\frac{\epsilon}{2}},M_f]\in\mathcal{L}_{\frac{d}{1-\epsilon},\infty},$ then $f\in\dot{W}^{1}_{\frac{d}{1-\epsilon}}(\mathbb{R}^d)$ and
$$\|f\|_{\dot{W}^{1}_{\frac{d}{1-\epsilon}}(\mathbb{R}^d)}\leq c'_{d,\epsilon}\Big\|[(-\Delta)^{\frac{\epsilon}{2}},M_f]\Big\|_{\frac{d}{1-\epsilon},\infty}.$$
\end{enumerate} 
\end{thm}

Our second result concerns the case $-\frac d2<\epsilon<0$.

\begin{thm}\label{main theorem negative} Let $d\geq 2$ and let $-\frac{d}{2}<\epsilon<0.$
\begin{enumerate}[{\rm (i)}]
\item\label{mtna} If $f\in\dot{W}^{1}_{\frac{d}{1-\epsilon}}(\mathbb{R}^d),$ then $[(-\Delta)^{\frac{\epsilon}{2}},M_f]\in\mathcal{L}_{\frac{d}{1-\epsilon},\infty}$ and
$$\Big\|[(-\Delta)^{\frac{\epsilon}{2}},M_f]\Big\|_{\frac{d}{1-\epsilon},\infty}\leq c_{d,\epsilon}\|f\|_{\dot{W}^{1}_{\frac{d}{1-\epsilon}}(\mathbb{R}^d)}.$$
\item\label{mtnb} If $f\in\dot{W}^{1}_{\frac{d}{1-\epsilon}}(\mathbb{R}^d)$, then
$$\lim_{t\to\infty}t^{\frac{1-\epsilon}{d}}\mu\Big(t,[(-\Delta)^{\frac{\epsilon}{2}},M_f]\Big)=\kappa_{d,\epsilon} \|f\|_{\dot{W}^{1}_{\frac{d}{1-\epsilon}}(\mathbb{R}^d)}$$
with $\kappa_{d,\epsilon}$ given by \eqref{eq:constasymp}.
\item\label{mtnc} If $f\in L_{1,{\rm loc}}(\mathbb R^d)$ and if 
$[(-\Delta)^{\frac{\epsilon}{2}},M_f]\in\mathcal{L}_{\frac{d}{1-\epsilon},\infty},$ then $f\in\dot{W}^{1}_{\frac{d}{1-\epsilon}}(\mathbb{R}^d)$ and
$$\|f\|_{\dot{W}^{1}_{\frac{d}{1-\epsilon}}(\mathbb{R}^d)}\leq c'_{d,\epsilon}\Big\|[(-\Delta)^{\frac{\epsilon}{2}},M_f]\Big\|_{\frac{d}{1-\epsilon},\infty}.$$
\end{enumerate} 
\end{thm}

Note that in this theorem we assume $\epsilon>-\frac d2$ and $d\geq 2$. For the case $\epsilon\leq - \frac d2$, see Lemma \ref{unbdd} below. Concerning the assumption $d\geq 2$, see the discussion in the following subsection.

Our third result concerns the case $\epsilon=1$.

\begin{thm}\label{epsilon1}
Let $d\geq 1.$
\begin{enumerate}[{\rm (i)}]
\item\label{mtabounded} If $f\in\dot{W}^{1}_{\infty}(\mathbb{R}^d),$ then $[(-\Delta)^{\frac{1}{2}},M_f]$ is bounded and
$$\Big\|[(-\Delta)^{\frac{1}{2}},M_f]\Big\|_{\infty} \leq c_{d} \, \|f\|_{\dot{W}^{1}_{\infty}(\mathbb{R}^d)}.$$
\item\label{mtcbounded} If $f\in L_{1,\rm loc}(\mathbb R^d)$ and if 
$[(-\Delta)^{\frac{1}{2}},M_f]$ is bounded, then $f\in\dot{W}^{1}_{\infty}(\mathbb{R}^d)$ and
$$\|f\|_{\dot{W}^{1}_{\infty}(\mathbb{R}^d)}\leq \Big\|[(-\Delta)^{\frac{1}{2}},M_f]\Big\|_\infty \,.$$
\item\label{mtdbounded} If $f\in\dot{W}^{1}_{\infty}(\mathbb{R}^d)$ and if 
$[(-\Delta)^{\frac{1}{2}},M_f]$ is compact, then $f\equiv{\rm const}$.
\end{enumerate} 
\end{thm}

Part \eqref{mtabounded} of this theorem is a celebrated result of Calder\'on \cite{Ca}, which we restate here only for the sake of completeness. Parts \eqref{mtcbounded} and \eqref{mtdbounded} for $d=1$ are due to Janson and Peetre \cite[Section 6, Example 7]{JaPe}. Part \eqref{mtcbounded} appeared recently in \cite[Theorem 1.5 with $\Omega\equiv 1$]{ChDiHo}, but we present an alternative, shorter proof for the case at hand. Part \eqref{mtdbounded} for $d\geq 2$ can be considered as new.


\subsection{Comparison with known results}

Let us formulate precisely the previous results on the operators $[(-\Delta)^\frac\epsilon 2,M_f]$. We will not define the relevant function spaces, since they will not play any role in the remainder of this paper. They are the fractional Sobolev spaces $\dot W^s_p(\mathbb R^d)$ with $0<s<1$, which are equal to a special case of Besov spaces, namely $\dot W^s_p(\mathbb R^d)=\dot B^s_{p,p}(\mathbb R^d)$. For these spaces see, for instance, \cite[Chapter 17]{Leoni-book}. Less standard spaces are the BMO-Sobolev spaces $(-\Delta)^\frac s2 BMO(\mathbb R^d)$, $0<s<1$, which are studied in \cite{St}. As discussed, for instance, in \cite{Yo}, these spaces coincides with a special case of Triebel--Lizorkin spaces, namely $(-\Delta)^\frac s2 BMO(\mathbb R^d) = \dot F^s_{\infty,2}(\mathbb R^d)$. For these spaces see \cite[Subsection 2.3.4]{Tr}.

\subsubsection*{Case $\epsilon\in(0,1)$} It is shown by Murray \cite{Mu} in dimension $d=1$ and by Janson and Peetre \cite{JaPe} in general dimension $d$ that
$$[(-\Delta)^{\frac{\epsilon}{2}},M_f]\mbox{ is bounded iff }
f\in (-\Delta)^{\frac{\epsilon}2} BMO(\mathbb{R}^d).$$
We are not aware of a published result characterizing compactness of $[(-\Delta)^{\frac{\epsilon}{2}},M_f]$, but it is natural to guess that the relevant space is the closure of $C^\infty_c(\mathbb R^d)$ in the space $(-\Delta)^{\frac{\epsilon}2}BMO(\mathbb{R}^d).$ Concerning trace ideals it is shown in \cite{JaPe} that for any $\frac d{1-\epsilon}<p<\infty$ one has
$$
[(-\Delta)^{\frac{\epsilon}{2}},M_f] \in \mathcal L_p
\qquad\iff\qquad
f\in \dot W^{\epsilon+\frac dp}_{p}(\mathbb R^d) \,,
$$
as well as
$$
[(-\Delta)^{\frac{\epsilon}{2}},M_f] \in \mathcal L_\frac d{1-\epsilon}
\qquad\iff\qquad
f \equiv {\rm const} \,.
$$
Our Theorem \ref{main theorem positive} completes these results, namely by characterizing membership of the commutator $[(-\Delta)^{\frac{\epsilon}{2}},M_f]$ to the weak Schatten class $\mathcal L_{\frac d{1-\epsilon},\infty}$ in terms of the condition $f\in\dot W^{1}_{\frac d{1-\epsilon}}(\mathbb R^d)$. Moreover, we compute the asymptotics of the singular values under the sole assumption of membership to this Sobolev space. 

\subsubsection*{Case $\epsilon\in(-\frac d2,0)$} Janson and Peetre \cite{JaPe} have shown that for any $\frac{d}{1-\epsilon}<p<\frac d{-\epsilon}$ one has
$$
[(-\Delta)^{\frac{\epsilon}{2}},M_f] \in \mathcal L_p
\qquad\iff\qquad
f\in \dot W^{\epsilon+\frac dp}_{p}(\mathbb R^d) \,.
$$
We are not aware of any results characterizing membership of $[(-\Delta)^{\frac{\epsilon}{2}},M_f]$ to $\mathcal L_p$ with $p\geq\frac d{-\epsilon}$ or its boundedness or compactness. Thus, the results for $\epsilon<0$ are far less complete than those for $\epsilon>0$. Our Theorem \ref{main theorem negative} completes the existing results at the endpoint $p=\frac{d}{1-\epsilon}$ in a similar way as in the case $\epsilon>0$. However, now we need the additional assumption $d\geq 2$.

The origin of this extra-assumption can be seen from the fact that the integrability exponent $\frac{d}{1-\epsilon}$ of the Sobolev space 
$\dot W^1_\frac d{1-\epsilon}$ is less than $1$ for $d=1$ and
 $\epsilon<0$. In this way the restriction to $d\geq 2$ arises from a technical point of view in our proofs. We think, however, that this restriction is not only technical but that the results are significantly different for 
 $d=1$. More specifically, while it is probably still true that for 
 $f\in C^\infty_c(\mathbb R)$, the operator $[(-\Delta)^{\frac{\epsilon}{2}},M_f]$ belongs to 
 $\mathcal L_{\frac 1{1-\epsilon}}$ and its singular values satisfy the asymptotics in \eqref{mtnb}, it is conceivable that there are 
 $f$ in $L_{1,{\rm loc}}(\mathbb R)$ with 
 $f'\in L_{\frac1{1-\epsilon}}(\mathbb R)$ for which $[(-\Delta)^{\frac{\epsilon}{2}},M_f]$ does not belong to $\mathcal L_{\frac 1{1-\epsilon}}$. This is a subject for further investigation. 

\subsubsection*{Case $\epsilon =1$}

As we have already mentioned, the boundedness of $[(-\Delta)^\frac12,M_f]$ under the assumption $f\in\dot W^1_\infty(\mathbb R^d)$ was proved in an influential paper of Calder\'on \cite{Ca}. For alternative proofs and extensions we refer, for instance, to \cite{CoMe,Ca2,CoMIMe}.

Janson and Peetre \cite{JaPe} prove that the condition $f\in\dot W^1_\infty(\mathbb R^d)$ is also necessary for boundedness in dimension $d=1$ and mention without proof that a referee of their paper has told them that this condition is also necessary for $d\geq 2$. We provide a proof of this claim. A different proof has appeared recently in \cite[Theorem 1.5 with $\Omega\equiv 1$]{ChDiHo}. The latter paper deals with a much larger class of operators than we do, but for the problem at hand our proof has the advantage of being much more direct.

Finally, Janson and Peetre \cite{JaPe} prove that in dimension $d=1$ the operator $[(-\Delta)^\frac12,M_f]$ is not compact unless $f$ is constant. Our Theorem \ref{epsilon1} shows that this remains valid in any dimension.

\subsection{Outline of the paper}

We end this introduction by giving a quick overview over this paper. For simplicity we restrict ourselves to the proofs of Theorems \ref{main theorem positive} and \ref{main theorem negative}; that of Theorem \ref{epsilon1} appears in Section \ref{sec:epsilon1} and uses related arguments.

When proving Theorems \ref{main theorem positive} and \ref{main theorem negative}, the methods that we are using are rather different for part (i) on the one hand and parts (ii) and (iii) on the other hand. 

Let us begin with discussing our proof of part (i) in Theorems \ref{main theorem positive} and \ref{main theorem negative}, which appears in Section \ref{sec:proofparti}. It is of operator-theoretic nature and based on the technique of Double Operator Integrals (DOIs). As two of us have demonstrated in \cite{KPSS,PS-crelle,LMSZ}, DOIs are an excellent tool for studying commutators. Earlier uses of DOIs in connection with spectral properties of pseudodifferential operators can be found, for instance, in the work of Birman and Solomyak \cite{BiSo}.

DOIs allow us to write the commutators $[(-\Delta)^\frac \epsilon 2,M_f]$ as a certain transformation applied to the simple commutator $[\nabla,M_f] = M_{\nabla f}$. Note that trace ideal properties of the operators $(-\Delta)^\frac{\epsilon-1}{4}[\nabla,M_f] (-\Delta)^\frac{\epsilon-1}{4}$ and generalizations thereof can be obtained from Cwikel's theorem \cite{Cwikel} and its generalizations.

Therefore, the proof of our results is reduced to studying the mapping properties of the transformation that allows us to write $[(-\Delta)^\frac \epsilon 2,M_f]$ in terms of $[\nabla,M_f]$. This transformation depends on the parameter $\epsilon$. The mapping properties in the case $\epsilon\in[-1,1]$ are relatively straightforward, given the previous results in \cite{PS-crelle} by one of us. To prove the mapping properties for $\epsilon<-1$ we use a certain `renormalization procedure', where we extract a finite number (depending on $\epsilon$) of extra terms, before being able to apply again the results in \cite{PS-crelle}.

We now turn to a discussion of the methods used to prove parts (ii) and (iii) in Theorems \ref{main theorem positive} and \ref{main theorem negative}. For part (ii) the analysis is divided into two steps. In a first step we derive the spectral asymptotics in the smooth case, that is, for $f\in C^\infty_c(\mathbb R^d)$ and with $-\Delta$ replaced by $1-\Delta$. In a second step we use the uniform a priori bounds from part (i) to extend the asymptotics to the maximal class of functions $f$.

The first step here uses an approximation result for commutators that is in the spirit of parametrix constructions in the theory of pseudodifferential operators, but we present it in an elementary fashion; see Section \ref{sec:approximate}. Once we have this approximation, we can apply the results about spectral properties that we developed in our previous paper \cite{FrSuZa}; see Section \ref{sec:asymptotics}. To be more precise, the methods from \cite{FrSuZa} are applicable, but the results are not, at least not in an obvious way. Instead of referring the reader to redo the arguments in \cite{FrSuZa}, in Section \ref{sec:spectralasymp} we present a method to deduce from \cite{FrSuZa} the results we need \emph{without} redoing the argument. We believe that this argument is interesting in its own right and illustrates, once again, the power of DOI techniques.

Part (iii) of Theorems \ref{main theorem positive} and \ref{main theorem negative} is proved in Section \ref{sec:proofpartiii}. The argument is related to that in the proof of part (ii). Namely, we show that if $[(-\Delta)^\frac\epsilon 2,M_f]$ has certain trace ideal properties, then the same is true for the corresponding operator with a regularized $f$. For the latter operator we can use (a localized version of) the spectral asymptotics in part (ii) to get a bound on the Sobolev norm of the regularized version of $f$, uniform in the regularization parameter. This allows us to conclude the $f$ itself has to be sufficiently regular.

\subsection{Acknowledgements}

\!\!Partial support through US National Science Foundation grant DMS-1954995 (RLF), the Deutsche Forschungsgemeinschaft (German Research Foundation) through Germany’s Excellence Strategy EXC-2111-390814868 (RLF), and through Australian Research Council, Laureate Fellowship FL170100052 (FS) and DP230100434 (DZ) is acknowledged.


\section{Preliminaries}\label{prelim section}

All considered functions are complex-valued, unless otherwise specified. 

Let $(\Omega, \Sigma,\nu)$ be a measure space with a $\sigma$-finite measure $\nu$,
defined on a $\sigma$-algebra $\Sigma$, and let $L(\Omega)$ be the algebra of all classes of equivalent measurable complex-valued functions on $(\Omega, \Sigma, \nu)$.  Denote by $L_0(\Omega)$ the subalgebra of $L(\Omega)$ consisting of all functions $f$ such that $\nu(\{|f|>s\})<\infty$ for some $s>0$.  For every $f\in L_0(\Omega)$, its non-increasing rearrangement is  defined by
$$ \mu(t, f):= \inf \{s> 0 :  \nu(\{|f|>s\}) \le t\}, ~t>0.  $$
For $0<p<\infty$ the space $L_p(\Omega, \Sigma,\nu)$ (resp.\  $L_{p, \infty}(\Omega, \Sigma,\nu)$) consists of all elements $f\in L_0(\Omega)$ for which
$$
\|f\|_p :=\left(\int _0^\infty  \mu(t, f)^p dt\right)^{1/p}<\infty, \quad \left( {\rm  resp.}\
\|f\|_{p,\infty} :=\sup_{t>0}t^{1/p} \mu(t, f)<\infty \right).
$$

As always, the measure spaces $\mathbb{Z}_+=\{0,1,2,\cdots\}$, $\mathbb{N}=\{1,2,3,\cdots\}$ are equipped with counting measures and the Euclidean space $\mathbb{R}^d$ with Lebesgue measure.  

For detailed information concerning classical function and sequence spaces such as $L_p(\mathbb R^d)$, $L_{p,\infty}(\mathbb R^d),$ 
$\ell_p=\ell_p(\mathbb{Z}_+)$ and $\ell_{p,\infty}= \ell_{p,\infty}(\mathbb{Z}_+)$, we refer the reader for instance to \cite{KPS, LSZ2012, LSZ-book-2}.

In this paper, $\nabla$ is the {\it self-adjoint} gradient operator on $L_2(\mathbb{R}^d),$ that is
$$\nabla=(\frac1i\frac{\partial}{\partial t_1},\cdots,\frac1i\frac{\partial}{\partial t_d}).$$

\begin{defi} The homogeneous Sobolev space $\dot W^{1}_{p}(\mathbb{R}^d),$ $1\leq p\leq\infty,$ consists of functions $f\in L_{1,{\rm loc}}(\mathbb{R}^d)$ whose distributional gradient belongs to $L_p(\mathbb R^d).$ This is a space of functions modulo constants. We set
	$$
	\| f \|_{\dot W^1_p(\mathbb R^d)} := \|\nabla f\|_p \,.
	$$
\end{defi}

The following result is available, see e.g. \cite[Theorem 12.9]{Leoni-book}.

\begin{thm}[Sobolev Embedding Theorem]\label{sobolev embedding theorem} Let $1\leq p<d$ and let $\frac1q=\frac1p-\frac1d.$ We have $\dot W^1_p(\mathbb{R}^d)\subset L_q(\mathbb{R}^d).$ This should be understood in the sense that in the class of functions modulo constants there is one that belongs to $L_q(\mathbb{R}^d).$
\end{thm}

Let $H$ be a complex separable Hilbert space, and let $\Bc(H)$ denote the set of bounded operators on $H$. The standard uniform operator norm on $\Bc(H)$ is denoted by $\|\cdot\|_\infty$. 

Let $\mathcal{K}(H)$ denote the ideal of compact operators on $H$. Given $T\in \mathcal{K}(H)$, the  singular value function $t\mapsto \mu(t, T)$ is defined by the formula 
    \begin{equation*}\label{def-singular}
\mu(t,T) := \inf\{\|T-R\|_{\infty}:\quad \mathrm{rank}(R)\leq t\},\quad t\ge 0.
\end{equation*}
We denote by $\mu(T)$ the sequence $\{\mu(n,T)\}_{n=0}^\infty.$
 Equivalently, $\mu(T)$ is the sequence of eigenvalues of $|T|$ arranged in nonincreasing order with multiplicities.

The following basic properties of singular values will be frequently used in the sequel:
\begin{equation}\label{singular-value-sum}
\mu(t+s , A+B)  \leq  \mu (t, A) + \mu (s,  B),\quad t, s\ge 0 \,,
\end{equation}
    and, for any $0<p<\infty$,
\begin{equation}\label{singular-value-square}
\mu(t , A)  =  \mu (t, A^*)  =   \mu (t,  |A|^p)^{\frac 1 p},\quad t\ge 0 \,;
\end{equation}
see e.g. \cite[Section~2.3]{LSZ2012}.

Let $p \in (0,\infty).$ The Schatten class $\mathcal{L}_p$ is the set of operators $T$ in $\mathcal{K}(H)$ such that $\mu(T)$ is $p$-summable, that is, belongs to the sequence space $\ell_p$. If $1\leq p <\infty$, then the $\mathcal{L}_p$ norm is defined by
    \begin{equation*}
        \|T\|_p := \|\mu(T)\|_{\ell_p} = \big(\sum_{n=0}^\infty \mu(n,T)^p\big)^{1/p}.
    \end{equation*}
    With this norm $\mathcal{L}_p$ for $1\leq p <\infty$ is a Banach space, and a Banach ideal of $\Bc(H)$.
    
Analogously, for $0<p< \infty$, the weak Schatten class 
    $\mathcal{L}_{p,\infty}$ is the set of operators $T\in \mathcal{K}(H)$ such that $\mu(T)$ is in the weak $L_p$-space $\ell_{p,\infty}$, with quasi-norm
    \begin{equation*}
        \|T\|_{p,\infty} := \sup_{n\geq 0}\, (n+1)^{1/p}\mu(n,T) < \infty.
    \end{equation*}
    As with the $\mathcal{L}_p$ spaces, $\mathcal{L}_{p,\infty}$ is an ideal of $\Bc(H)$.
    
For more details on ideals of compact operators and singular value sequences, we refer the reader to \cite{LSZ2012, LSZ-book-2, Simon-book}.

We will frequently use the following two trace ideal bounds for operators of the form $M_f g(\nabla)$ on $L^2(\mathbb R^d)$. To formulate them, we use the Birman--Solomyak spaces $\ell_p(L_q)(\mathbb R^d)$ and $\ell_{p,\infty}(L_q)(\mathbb R^d)$ (see \cite{BiSo}, \cite[Chapter 4]{Simon-book} or \cite[Subsection 5.6]{BKS}), as well as their analogue $\ell_{2,\log}(L_\infty)(\mathbb R^d)$ in \cite[Definition 5.5]{LeSZ}. The only fact about these spaces that we will be using is that they all contain $C_c(\mathbb R^d)$.

\begin{thm}\label{pre-cwikel estimate}
	Let $d\geq 1$.
	\begin{enumerate}
		\item[(i)] If $2\leq p\leq\infty$ and $f,g\in L_p(\mathbb R^d)$, then $M_fg(\nabla)\in \mathcal{L}_p(L_2(\mathbb{R}^d))$ and
		$$
		\|M_fg(\nabla)\|_{\mathcal{L}_p(L_2(\mathbb{R}^d))}\leq c_{p,d}\|f\|_{L_p(\mathbb{R}^d)}\|g\|_{L_p(\mathbb{R}^d)}.
		$$
		\item[(ii)] If $0<p<2$ and $f,g\in \ell_p(L_2)(\mathbb R^d)$, then $M_fg(\nabla)\in \mathcal{L}_p(L_2(\mathbb{R}^d))$ and
		$$
		\|M_fg(\nabla)\|_{\mathcal{L}_p(L_2(\mathbb{R}^d))}\leq c_{p,d}\|f\|_{\ell_p(L_2)(\mathbb{R}^d)}\|g\|_{\ell_p(L_2)(\mathbb{R}^d)}.
		$$
	\end{enumerate}
\end{thm}

Part (i) is known as the Kato--Seiler--Simon inequality; for a proof see, e.g., \cite[Theorem 4.1]{Simon-book}. Part (2) is due to Birman and Solomyak \cite[Theorem 11.1]{BiSo77}; for a proof for $p\geq 1$ see also \cite[Theorem 4.5]{Simon-book}. For a strengthening see \cite[Theorem~1.4]{LeSZ}.

\begin{thm}\label{cwikel estimate}
	Let $d\geq 1$.
	\begin{enumerate}
		\item[(i)] If $2<p<\infty$ and $f\in L_p(\mathbb R^d)$, $g\in L_{p,\infty}(\mathbb R^d)$, then $M_fg(\nabla)\in \mathcal{L}_{p,\infty}(L_2(\mathbb{R}^d))$ and
		$$\|M_fg(\nabla)\|_{\mathcal{L}_{p,\infty}(L_2(\mathbb{R}^d))}\leq c_{p,d}\|f\|_{L_p(\mathbb{R}^d)}\|g\|_{L_{p,\infty}(\mathbb{R}^d)}.$$
		\item[(ii)] If $p=2$ and $f\in \ell_{2,\log}(L_\infty)(\mathbb R^d)$, $g\in\ell_{2,\infty}(L_4)(\mathbb R^d)$, then $M_fg(\nabla)\in \mathcal{L}_{p,\infty}(L_2(\mathbb{R}^d))$ and
		$$\|M_fg(\nabla)\|_{\mathcal{L}_{2,\infty}(L_2(\mathbb{R}^d))}\leq c_{2,d}\|f\|_{\ell_{2,\log}(L_\infty)(\mathbb{R}^d)}\|g\|_{\ell_{2,\infty}(L_4)(\mathbb{R}^d)}.$$
		\item[(iii)] If $0<p<2$ and $f\in \ell_{p}(L_2)(\mathbb R^d)$, $g\in\ell_{p,\infty}(L_2)(\mathbb R^d)$, then $M_fg(\nabla)\in \mathcal{L}_{p,\infty}(L_2(\mathbb{R}^d))$ and
		$$\|M_fg(\nabla)\|_{\mathcal{L}_{p,\infty}(L_2(\mathbb{R}^d))}\leq c_{p,d}\|f\|_{\ell_p(L_2)(\mathbb{R}^d)}\|g\|_{\ell_{p,\infty}(L_2)(\mathbb{R}^d)}.$$
	\end{enumerate}
\end{thm}

Part (1) is known as the Cwikel inequality and due to \cite{Cwikel}; for proofs see also, for instance, \cite[Theorem 4.2]{Simon-book} and \cite{Fr,LeSZ}. Part (2) is due to two of us \cite[Theorem~1.3]{LeSZ}. Part (3) is due to Simon \cite[Theorem 4.6]{Simon-book} for $p>1$ and appears in \cite[Subsection 5.7]{BKS} in general. For a strengthening see \cite[Theorem~1.4]{LeSZ}.


\section{Domain and codomain of the commutator $[(-\Delta)^{\frac{\epsilon}{2}},M_f]$}\label{sec:preliminary}

We begin by clarifying in which sense the commutator $[(-\Delta)^\frac\epsilon2,M_f]$ is understood for $f\in L_{1,{\rm loc}}(\mathbb R^d)$ and $\epsilon\in(-d,1]$.

\begin{lem}\label{definition}
	Let $\epsilon\in(-d,1]$ and let $f\in L_{1,{\rm loc}}(\mathbb R^d)$. For any $\phi\in C_c^\infty(\mathbb R^d)$ there are unique distributions, denoted by
	$$
	(-\Delta)^\frac\epsilon2 M_f\phi \quad\text{and}\quad
	M_f(-\Delta)^\frac\epsilon2\phi \,,
	$$
	such that for all $\psi\in C_c^\infty(\mathbb R^d)$ one has
	$$
	\langle (-\Delta)^\frac\epsilon2 M_f\phi, \psi \rangle = \langle f\phi,(-\Delta)^\frac\epsilon2 \psi\rangle
	\qquad\text{and}\qquad
	\langle M_f(-\Delta)^\frac\epsilon2\phi, \psi \rangle = \langle (-\Delta)^\frac\epsilon2\phi, \overline f \psi \rangle \,.
	$$
	Here $\langle\cdot,\cdot\rangle$ on the left sides denotes the sesquilinear duality pairing between $(C_c^\infty)'$ and $C_c^\infty$, and on the right sides that between $L_1$ and $L_{\infty}$ ($L_{\infty}$ and $L_1$, respectively). The distribution $M_f(-\Delta)^\frac\epsilon2\phi$ is regular (that is, given by an $L_{1,{\rm loc}}$-function), and so is, for $\epsilon<0$, the distribution $(-\Delta)^\frac\epsilon2 M_f \phi$.\\
	Moreover, if $\{f_n\}_{n\geq0}\subset L_{1,{\rm loc}}(\mathbb{R}^d)$ is such that $f_n\to f$ in $L_{1,{\rm loc}}(\mathbb{R}^d)$, then for every $\phi\in C^{\infty}_c(\mathbb{R}^d),$ we have
	$$(-\Delta)^{\frac{\epsilon}{2}}M_{f_n}\phi\to (-\Delta)^{\frac{\epsilon}{2}}M_f\phi,\quad M_{f_n}(-\Delta)^{\frac{\epsilon}{2}}\phi\to M_f(-\Delta)^{\frac{\epsilon}{2}}\phi,\quad n\to\infty,$$
	in $(C^{\infty}_c(\mathbb{R}^d))'.$
\end{lem}

As a consequence of this lemma, the mappings 
$$
M_f(-\Delta)^{\frac{\epsilon}{2}}:C^{\infty}_c(\mathbb{R}^d)\to L_{1,{\rm loc}}(\mathbb{R}^d)
$$
and
$$
(-\Delta)^{\frac{\epsilon}{2}}M_f:C^{\infty}_c(\mathbb{R}^d)\to 
\begin{cases}
	L_{1,{\rm loc}}(\mathbb{R}^d) & \text{if}\ \epsilon\in(-d,0) \,, \\
	(C^{\infty}_c(\mathbb{R}^d))' & \text{if}\ \epsilon\in(0,1] \,.
\end{cases}
$$
are well defined.

\begin{proof}
	We begin with the case $\epsilon\in(-d,0)$. It is well known (see, e.g., \cite[Chapter V, Theorem 1]{Stein}) that $(-\Delta)^{\frac{\epsilon}{2}}:L_1(\mathbb{R}^d)\to L_{\frac{d}{d+\epsilon},\infty}(\mathbb{R}^d).$ If $\phi\in C^{\infty}_c(\mathbb{R}^d),$ then $f\phi\in L_1(\mathbb{R}^d)$ and, therefore,
	$$\big((-\Delta)^{\frac{\epsilon}{2}}M_f\big)(\phi)=(-\Delta)^{\frac{\epsilon}{2}}(f\phi)\in L_{\frac{d}{d+\epsilon},\infty}(\mathbb{R}^d)\subset L_{1,{\rm loc}}(\mathbb{R}^d).$$	
	Moreover, by duality, $(-\Delta)^{\frac{\epsilon}{2}}: L_{-\frac{d}{\epsilon},1}(\mathbb R^d) \to L_{\infty}(\mathbb{R}^d).$ Therefore,	
	$$
	\big(M_f(-\Delta)^{\frac{\epsilon}{2}}\big)(\phi)=f\cdot (-\Delta)^{\frac{\epsilon}{2}}\phi\in L_{1,{\rm loc}}(\mathbb{R}^d)\cdot L_{\infty}(\mathbb{R}^d)\subset L_{1,{\rm loc}}(\mathbb{R}^d).
	$$
	Thus, both $(-\Delta)^\frac\epsilon2 M_f\phi$ and $M_f(-\Delta)^\frac\epsilon2\phi$ are regular distributions.
	
	Now let $\epsilon\in(0,1]$. Then $(-\Delta)^{\frac{\epsilon}{2}}:C^{\infty}_c(\mathbb{R}^d)\to L_{\infty}(\mathbb{R}^d).$ Thus,
	$$
	(M_f(-\Delta)^{\frac{\epsilon}{2}})(\phi)=f\cdot (-\Delta)^{\frac{\epsilon}{2}}\phi\in L_{1,{\rm loc}}(\mathbb{R}^d)\cdot L_{\infty}(\mathbb{R}^d)\subset L_{1,{\rm loc}}(\mathbb{R}^d).
	$$
	Moreover, by duality, $(-\Delta)^{\frac{\epsilon}{2}}: (L_{\infty}(\mathbb{R}^d))' \to (C^{\infty}_c(\mathbb{R}^d))'$, where on $(L_{\infty}(\mathbb{R}^d))'$ we consider the weak star topology. Since $f\phi\in L_1(\mathbb{R}^d) \subset(L_{\infty}(\mathbb{R}^d))'$, we have
	$$
	((-\Delta)^{\frac{\epsilon}{2}}M_f)(\phi)=(-\Delta)^{\frac{\epsilon}{2}}(f\phi)\in(C^{\infty}_c(\mathbb{R}^d))'.
	$$	
	This proves the existence of the two distributions. The uniqueness is clear. It remains to prove the convergence statement. For the distribution $(-\Delta)^\frac\epsilon2 M_f\phi$ we use the fact that $f_n\phi\to f\phi$ in $L_1$ and therefore 
	$$
	\langle f_n \phi,(-\Delta)^\frac\epsilon2 \psi\rangle \to \langle f\phi,(-\Delta)^\frac\epsilon2 \psi\rangle \,.
	$$
	The proof for $M_f(-\Delta)^{\frac{\epsilon}{2}}\phi$ is similar, using $\overline{f_n}\psi\to \overline f\psi$ in $L_1$.
\end{proof}

\begin{lem}\label{tricky convergence} 
	Let $\epsilon\in(-\frac d2,1]$ and let $f\in L_\infty(\mathbb R^d)$ and $\phi\in C^\infty_c(\mathbb R^d)$. For $m\geq 1$ let $P_m=\chi_{(\frac1m,m)}(-\Delta)$. Then
	$$
	(P_m(-\Delta)^{\frac{\epsilon}{2}}) M_f P_m \phi \to (-\Delta)^{\frac{\epsilon}{2}} M_f \phi,
	\ \
	P_m M_f (P_m(-\Delta)^{\frac{\epsilon}{2}})\phi \to M_f (-\Delta)^{\frac{\epsilon}{2}}\phi,
	\ \
	m\to\infty,
	$$
	in $(C_c^\infty(\mathbb R^d))'$. Here $(-\Delta)^{\frac{\epsilon}{2}} M_f \phi$ and $M_f (-\Delta)^{\frac{\epsilon}{2}}\phi$ are defined as in Lemma \ref{definition}.
\end{lem}

\begin{proof}
	Let $\psi\in C^\infty_c(\mathbb R^d)$. Then, with $\langle\cdot,\cdot\rangle$ denoting the $L_2$-inner product,
	\begin{align*}
		\langle( (P_m(-\Delta)^{\frac{\epsilon}{2}}) \cdot M_f\cdot P_m)\phi,\psi\rangle & =\langle f\cdot P_m\phi,(P_m(-\Delta)^{\frac{\epsilon}{2}})\psi\rangle, \\
		\langle(P_m\cdot M_f\cdot (P_m(-\Delta)^{\frac{\epsilon}{2}})\phi,\psi\rangle & =\langle (P_m(-\Delta)^{\frac{\epsilon}{2}})\phi,\overline f\cdot P_m\psi\rangle.
	\end{align*}
	Let $\chi$ denote either $\phi$ or $\psi$. We claim that
	\begin{equation}
		\label{eq:trickyconv}
		\chi,(-\Delta)^\frac\epsilon2\chi \in L_2(\mathbb{R}^d).
	\end{equation}
	Indeed, for $\chi$ this is clear. When $\epsilon\in(0,1]$ it is also clear for $(-\Delta)^\frac\epsilon2\chi$. Thus let $\epsilon\in(-\frac d2,0)$. As observed in the proof of Lemma \ref{definition}, we have
$$(-\Delta)^{\frac{\epsilon}{2}}:L_1(\mathbb{R}^d)\to L_{\frac{d}{d+\epsilon},\infty}(\mathbb{R}^d),\quad (-\Delta)^{\frac{\epsilon}{2}}:L_{-\frac{d}{\epsilon},1}(\mathbb{R}^d)\to L_{\infty}(\mathbb{R}^d).$$
Therefore,
$$(-\Delta)^{\frac{\epsilon}{2}}:L_1(\mathbb{R}^d)\cap L_{-\frac{d}{\epsilon},1}(\mathbb{R}^d)\to L_{\frac{d}{d+\epsilon},\infty}(\mathbb{R}^d)\cap L_{\infty}(\mathbb{R}^d)\subset L_2(\mathbb{R}^d).$$
In particular, $(-\Delta)^\frac\epsilon2\chi \in L_2(\mathbb{R}^d).$ This proves \eqref{eq:trickyconv}.
	
	It follows from \eqref{eq:trickyconv} that $P_m\chi\to\chi$ and $P_m (-\Delta)^\frac\epsilon2\chi\to (-\Delta)^\frac\epsilon2\chi$ in $L_2(\mathbb{R}^d)$. This, together with $f\in L_{\infty}(\mathbb R^d)$, implies
	\begin{align*}
		\begin{split}
			\langle f\cdot P_m\phi,(P_m(-\Delta)^{\frac{\epsilon}{2}})\psi\rangle
			& \to \langle f\cdot \phi,(-\Delta)^\frac\epsilon2\psi \rangle \,,\\
			\langle (P_m(-\Delta)^{\frac{\epsilon}{2}})\phi,\overline f\cdot P_m\psi\rangle
			& \to \langle (-\Delta)^\frac\epsilon2\phi,\overline f\cdot\psi \rangle \,.
		\end{split}
	\end{align*}
	Here $\langle\cdot,\cdot\rangle$ denotes the $L_2$-inner product. Alternatively, it can be interpreted as the duality pairing between $L_1$ and $L_{\infty}$ ($L_{\infty}$ and $L_1$, respectively), and then, by Lemma \ref{definition}, the right sides coincide with $\langle (-\Delta)^\frac\epsilon2 M_f\phi,\psi\rangle$ and $\langle M_f (-\Delta)^\frac\epsilon2 \phi,\psi\rangle$, respectively, where now $\langle\cdot,\cdot\rangle$ denotes the sesquilinear duality pairing between $(C_c^\infty(\mathbb R^d))'$ and $C_c^\infty(\mathbb R^d)$. This proves the claimed convergence in $C_c^\infty(\mathbb R^d)$.
\end{proof}

Having defined the operator $[(-\Delta)^{\frac{\epsilon}{2}},M_f]:C^{\infty}_c(\mathbb{R}^d)\to (C^{\infty}_c)'(\mathbb{R}^d),$ we are now interested in when it extends to a bounded operator on $L_2(\mathbb{R}^d)$ and belongs to certain operator ideal. The following theorem provides a negative result in a certain range of negative $\epsilon$'s.

\begin{lem}\label{unbdd}
If $\epsilon\in(-d,-\frac{d}{2}]$ and if $0\neq f\in L_1(\mathbb{R}^d)$ is compactly supported, then the operator $[(-\Delta)^{\frac{\epsilon}{2}},M_f]$ does not map $L_2(\mathbb{R}^d)$ to itself.
\end{lem}

\begin{proof} The operator $(-\Delta)^{\frac{\epsilon}{2}}$ is an integral operator with an integral kernel $|t-s|^{-d-\epsilon}$ (up to a multiplicative constant). Hence, the operators $M_f(-\Delta)^{\frac{\epsilon}{2}},(-\Delta)^{\frac{\epsilon}{2}}M_f:C^{\infty}_c(\mathbb{R}^d)\to L_{1,{\rm loc}}(\mathbb{R}^d)$ (those operators are well defined, see the proof of Lemma \ref{definition}) are also integral ones. Thus, $[(-\Delta)^{\frac{\epsilon}{2}},M_f]$ is an integral operator with the integral kernel given, up to a multiplicative constant, by
$$(t,s)\mapsto \frac{f(t)-f(s)}{|t-s|^{d+\epsilon}},\quad t,s\in\mathbb{R}^d.$$

Suppose $f$ is supported in the ball $B(0,R).$ Let $\phi\in C^{\infty}_c(\mathbb{R}^d)$ be such that
$$\int_{\mathbb{R}^d}(f\phi)(s)ds\neq0.$$
We have
$$\Big(\big([(-\Delta)^{\frac{\epsilon}{2}},M_f]\big)\phi\Big)(t)=-\int_{|s|\leq R}\frac{(f\phi)(s)}{|t-s|^{d+\epsilon}}ds,\quad |t|>R.$$
It follows that
$$\Big(\big([(-\Delta)^{\frac{\epsilon}{2}},M_f]\big)\phi\Big)(t)=-|t|^{-d-\epsilon}\cdot\int_{|s|\leq R}(f\phi)(s)ds+O(|t|^{-d-\epsilon-1}),\quad |t|>2R.$$
The function on the right hand side does not belong to $L_2(\mathbb{R}^d\backslash B(0,2R)).$ Hence, we have $\big([(-\Delta)^{\frac{\epsilon}{2}},M_f]\big)\phi\notin L_2(\mathbb{R}^d).$
\end{proof}

Let now $-\frac{d}{2}<\epsilon<0.$ By the weak Young inequality (see, e.g., \cite[Chapter V, Theorem 1]{Stein}), H\"older's inequality and Sobolev's inequality (see Theorem \ref{sobolev embedding theorem}) we find
$$
\|M_f(-\Delta)^{\frac{\epsilon}{2}}\|_\infty ,\ \|(-\Delta)^{\frac{\epsilon}{2}}M_f\|_\infty \leq c_{d,\epsilon}^{(1)}\|f\|_{-\frac{d}{\epsilon}} \leq c_{d,\epsilon}^{(2)}\|f\|_{\dot{W}^{1}_{\frac{d}{1-\epsilon}}(\mathbb{R}^d)}.
$$
Hence, for $-\frac{d}{2}<\epsilon<0$ and $f\in\dot W^1_{\frac d{1-\epsilon}}$ the commutator $[(-\Delta)^{\frac{\epsilon}{2}},M_f]$ is a difference of two bounded operators and is, therefore, bounded. Moreover,
\begin{equation}\label{uniform norm negative equation}
\Big\|[(-\Delta)^{\frac{\epsilon}{2}},M_f]\Big\|_{\infty}\leq c_{d,\epsilon}\|f\|_{\dot{W}^{1}_{\frac{d}{1-\epsilon}}(\mathbb{R}^d)},\quad f\in \dot{W}^{1}_{\frac{d}{1-\epsilon}}(\mathbb{R}^d),\quad -\frac{d}{2}<\epsilon<0.
\end{equation}


\section{Proof of Theorem \ref{main theorem positive} \eqref{mtpa} and Theorem \ref{main theorem negative} \eqref{mtna}}\label{sec:proofparti}

If $\mathbf A=(A_1,\ldots,A_d)$ is a $d$-tuple of mutually commuting self-adjoint operator on a Hilbert space and $\phi$ is a sufficiently regular function on $\mathbb R^d\times\mathbb R^d$, then the symbol $T^{\bf A}_{\phi}$ denotes the corresponding Double Operator Integral (DOI). We refer the reader to \cite{KPSS,CSZ-AiF} for the notion of a DOI with respect to tuples of mutually commuting operators. The following lemma shows that for certain $\phi$ this DOI can be reduced to one for a single operator. It is the exact analogue of \cite[Lemma 8]{PS-crelle} and its proof is exactly the same as in \cite{PS-crelle}, so we omit it.

\begin{lem}\label{doi composition lemma} Let $\Psi:\mathbb{R}\times\mathbb{R}\to\mathbb{R}$ be a bounded Borel measurable function and let $h:\mathbb{R}^d\to\mathbb{R}$ be a Borel measurable function. Set
	$$\phi(\lambda,\mu)=\Psi(h(\lambda),h(\mu)),\quad \lambda,\mu\in\mathbb{R}^d.$$
	For every tuple ${\bf A}=(A_1,\cdots,A_d)$ of commuting self-adjoint operators we have
	$$T^{{\bf A}}_{\phi}=T^{h({\bf A})}_{\Psi}.$$
\end{lem}

For a fixed $\epsilon\in\mathbb{R},$ define the function $\phi_{\epsilon}$ on $\mathbb{R}^d\times\mathbb{R}^d$ by setting
\begin{equation}\label{phi def}
\phi_{\epsilon}(\lambda,\mu) :=
\begin{cases}
\frac{|\lambda|^{\epsilon}-|\mu|^{\epsilon}}{|\lambda|-|\mu|}\cdot  |\lambda|^{\frac{1-\epsilon}{2}}|\mu|^{\frac{1-\epsilon}{2}} & \text{if}\ |\lambda|\neq|\mu| \,, \\
\epsilon & \text{if}\ |\lambda|=|\mu| \,.
\end{cases}
\end{equation}

\begin{lem}\label{tboundedphi}
Let $\epsilon\in[-1,1]$ and let $\phi_{\epsilon}$ be as in \eqref{phi def}. Then for any $1<p<\infty$ the operator $T^{\nabla}_{\phi_{\epsilon}}$ is bounded on $\mathcal{L}_p$ and on $\mathcal{L}_{p,\infty}$ 
\end{lem}

\begin{proof} For $\epsilon=\pm1,$  $\phi_{\epsilon}$ is a constant and there is nothing to prove. Suppose $\epsilon\in(-1,1).$ By Lemma \ref{doi composition lemma} applied with $h(\lambda)=|\lambda|,$ we have
$$T^{\nabla}_{\phi_{\epsilon}}=T^{(-\Delta)^{\frac12}}_{\Psi_{\epsilon}}$$
with
$$\Psi_{\epsilon}(s,t):=\frac{s^{\epsilon}-t^{\epsilon}}{s-t}\cdot s^{\frac{1-\epsilon}{2}}t^{\frac{1-\epsilon}{2}},\quad s,t>0.$$
Thus,
$$\Psi_{\epsilon}(s,t)=h_{\epsilon}(\frac{s}{t}),\quad s,t>0,$$
with the function $h_{\epsilon}$ given by the formula
$$(h_{\epsilon}\circ\exp)(t)=\frac{e^{\epsilon t}-1}{e^t-1}\cdot e^{\frac{(1-\epsilon)t}{2}}=\frac{\sinh(\frac{\epsilon t}{2})}{\sinh(\frac{t}{2})},\quad t\in\mathbb{R}.$$
In particular, $h_{\epsilon}\circ\exp$ is a Schwartz function. By \cite[Lemma 9]{PS-crelle}, we have 
$$T^{(-\Delta)^{\frac12}}_{\Psi_{\epsilon}}:\mathcal{L}_1\to\mathcal{L}_1,\quad T^{(-\Delta)^{\frac12}}_{\Psi_{\epsilon}}:\mathcal{L}_{\infty}\to\mathcal{L}_{\infty}.$$
Therefore, by interpolation,
$$T^{(-\Delta)^{\frac12}}_{\Psi_{\epsilon}}:\mathcal{L}_p\to\mathcal{L}_p,\quad T^{(-\Delta)^{\frac12}}_{\Psi_{\epsilon}}:\mathcal{L}_{p,\infty}\to\mathcal{L}_{p,\infty},\quad 1<p<\infty,$$
proving the assertion.
\end{proof}

For $\epsilon\in\mathbb R$ and $n\in\mathbb N$, define the function $\theta_{n,\epsilon}$ on $\mathbb R^d\times\mathbb R^d$ by setting
$$
\theta_{n,\epsilon}(\lambda,\mu) : =\sum_{l=1}^n(\frac{|\lambda|}{|\mu|})^{\frac{\epsilon+2l-1}{2}}+\sum_{l=1}^n(\frac{|\mu|}{|\lambda|})^{\frac{\epsilon+2l-1}{2}}, ~\lambda, \mu\in \mathbb R^d, ~\lambda, \mu\neq 0.
$$

\begin{lem}\label{phi induction lemma} Let $n\in\mathbb{N}$ and let $\epsilon\in[-2n-1,-2n+1).$ We have
$$\phi_{\epsilon}(\lambda,\mu)=\phi_{\epsilon+2n}(\lambda,\mu)-\theta_{n,\epsilon}(\lambda,\mu),\quad \lambda,\mu\in\mathbb{R}^d,\quad \lambda,\mu\neq0.$$
\end{lem}

\begin{proof} For every $\epsilon\in\mathbb{R},$ we have
$$\phi_{\epsilon}(\lambda,\mu)-\phi_{\epsilon+2}(\lambda,\mu)=$$
$$=\frac{|\lambda|^{\epsilon}-|\mu|^{\epsilon}}{|\lambda|-|\mu|}\cdot |\lambda|^{\frac{1-\epsilon}{2}}|\mu|^{\frac{1-\epsilon}{2}}-\frac{|\lambda|^{\epsilon+2}-|\mu|^{\epsilon+2}}{|\lambda|-|\mu|}\cdot |\lambda|^{-\frac{1+\epsilon}{2}}|\mu|^{-\frac{1+\epsilon}{2}}=$$
$$=\frac{|\lambda|^{-\frac{1+\epsilon}{2}}|\mu|^{-\frac{1+\epsilon}{2}}}{|\lambda|-|\mu|}\cdot\Big(|\lambda||\mu|(|\lambda|^{\epsilon}-|\mu|^{\epsilon})-(|\lambda|^{\epsilon+2}-|\mu|^{\epsilon+2})\Big).$$
Note that
$$|\lambda||\mu|(|\lambda|^{\epsilon}-|\mu|^{\epsilon})-(|\lambda|^{\epsilon+2}-|\mu|^{\epsilon+2})=$$
$$=-|\lambda|\cdot|\mu|^{\epsilon+1}+|\mu|\cdot |\mu|^{\epsilon+1}-|\lambda|\cdot|\lambda|^{\epsilon+1}+|\mu|\cdot|\lambda|^{\epsilon+1}=$$
$$=-(|\lambda|-|\mu|)(|\lambda|^{\epsilon+1}+|\mu|^{\epsilon+1}).$$
Thus,
$$\phi_{\epsilon}(\lambda,\mu)-\phi_{\epsilon+2}(\lambda,\mu)=$$
$$=-|\lambda|^{-\frac{1+\epsilon}{2}}|\mu|^{-\frac{1+\epsilon}{2}}\cdot(|\lambda|^{\epsilon+1}+|\mu|^{\epsilon+1})=-(\frac{|\lambda|}{|\mu|})^{\frac{\epsilon+1}{2}}-(\frac{|\mu|}{|\lambda|})^{\frac{\epsilon+1}{2}}.$$
The assertion follows now by induction on $n.$
\end{proof}

Define the function $\psi_k,$ $1\leq k\leq d,$ on $\mathbb{R}^d\times\mathbb{R}^d$ by setting
\begin{equation}\label{eq:defpsik}
\psi_k(\lambda,\mu) :=
\begin{cases}
\frac{|\lambda|-|\mu|}{|\lambda-\mu|^2}\cdot(\lambda_k -\mu_k) & \text{if}\ \lambda\neq \mu \,,\\
0 & \text{if}\ \lambda=\mu \,.
\end{cases}
\end{equation}

\begin{lem}\label{tboundedpsi}
Let $1\leq k\leq d$ and let $\psi_k$ be as in \eqref{eq:defpsik}. For all $1<p<\infty$ the operator $T^{\nabla}_{\psi_k}$ is bounded on $\mathcal L_p$ and on $\mathcal L_{p,\infty}$. 
\end{lem}

\begin{proof}
Boundedness of $T^{\nabla}_{\psi_k}:\mathcal{L}_p\to\mathcal{L}_p,$ $1<p<\infty,$ is established in \cite[Corollary 5.2]{CSZ-AiF}; see also \cite[Theorem 5.1]{CMPS}. By interpolation, $T^{\nabla}_{\psi_k}:\mathcal{L}_{p,\infty}\to\mathcal{L}_{p,\infty},$ $1<p<\infty,$ is bounded as well.
\end{proof}

For a tuple $\mathbf A=(A_1,\ldots,A_d)$ we write 
$$
|{\bf A}| = (A_1^2+\ldots+A_d^2)^\frac12 \,.
$$
In the next two lemmas we derive representations for commutators $[|{\bf A}|^{\epsilon},B]$ as DOIs. 

\begin{lem}\label{equality for bounded operators} 
	Let $\epsilon\in\mathbb{R}$ and let $\phi_{\epsilon}$ and $\psi_k$ be as in \eqref{phi def} and \eqref{eq:defpsik}, respectively. If $B\in\mathcal{L}_2$ and if ${\bf A}$ is a tuple of bounded self-adjoint, mutually commuting operators such that $|{\bf A}|$ is bounded from below by a strictly positive constant, then
$$
[|{\bf A}|^{\epsilon},B] = \sum_{k=1}^dT^{{\bf A}}_{\phi_{\epsilon}\cdot\psi_k}(|{\bf A}|^{\frac{\epsilon-1}{2}}[A_k,B]|{\bf A}|^{\frac{\epsilon-1}{2}}).
$$
\end{lem}

\begin{proof} Let $g_k$ be the function on $\mathbb{R}^d\times\mathbb{R}^d$ given by the formula $$g_k(\lambda,\mu)=|\lambda|^{\frac{\epsilon-1}{2}}(\lambda_k-\mu_k) |\mu|^{\frac{\epsilon-1}{2}},\quad \lambda,\mu\in\mathbb{R}^d.$$
	Note that $g_k$ is bounded when $\mu,\lambda$ range over compact sets in $\mathbb R^d\setminus\{0\}$. Then we have
$$|{\bf A}|^{\frac{\epsilon-1}{2}}[A_k,B]|{\bf A}|^{\frac{\epsilon-1}{2}}=T^{{\bf A}}_{g_k}(B).$$
It follows that
$$\sum_{k=1}^dT^{{\bf A}}_{\phi_\epsilon\cdot\psi_k}(|{\bf A}|^{\frac{\epsilon-1}{2}}[A_k,B]|{\bf A}|^{\frac{\epsilon-1}{2}})=$$
$$=\sum_{k=1}^dT^{{\bf A}}_{\phi_\epsilon\cdot\psi_k}\Big(T^{{\bf A}}_{g_k}(B)\Big)=\sum_{k=1}^dT^{{\bf A}}_{\phi_\epsilon\cdot\psi_k\cdot g_k}(B)=T^{{\bf A}}_{\sum_{k=1}^d\phi_\epsilon\cdot\psi_k\cdot g_k}(B).$$
Since
$$(\sum_{k=1}^d\phi_\epsilon\cdot\psi_k\cdot g_k)(\lambda,\mu)=|\lambda|^{\epsilon}-|\mu|^{\epsilon},\quad \lambda,\mu\in\mathbb{R}^d,$$
we have
$$T^{{\bf A}}_{\sum_{k=1}^d \phi_\epsilon \cdot\psi_k\cdot g_k}(B)=|{\bf A}|^{\epsilon}B-B|{\bf A}|^{\epsilon},$$
which proves the assertion.
\end{proof}

\begin{lem}\label{improved equality for bounded operators} Let $n\in\mathbb{N}$, let $\epsilon\in[-2n-1,-2n+1)$ and let $\phi_{\epsilon}$ and $\psi_k$ be as in \eqref{phi def} and \eqref{eq:defpsik}, respectively. If $B\in\mathcal{L}_2$ and if ${\bf A}$ is a tuple of bounded self-adjoint, mutually commuting operators such that $|{\bf A}|$ is bounded from below by a strictly positive constant, then
\begin{align*}
	[|{\bf A}|^{\epsilon},B] & = \sum_{k=1}^dT^{{\bf A}}_{\phi_{\epsilon+2n}\cdot \psi_k}\Big(|{\bf A}|^{\frac{\epsilon-1}{2}}[A_k,B]|{\bf A}|^{\frac{\epsilon-1}{2}}\Big) \\
	& \quad -\sum_{k=1}^dT^{{\bf A}}_{\psi_k}\Big(\sum_{l=1}^n|{\bf A}|^{\epsilon+l-1}[A_k,B]|{\bf A}|^{-l}+\sum_{l=1}^n|{\bf A}|^{-l}[A_k,B]|{\bf A}|^{\epsilon+l-1}\Big)
\end{align*}
\end{lem}

\begin{proof} By Lemma \ref{equality for bounded operators} and Lemma \ref{phi induction lemma}, we have
$$[|{\bf A}|^{\epsilon},B]=\sum_{k=1}^dT^{{\bf A}}_{(\phi_{\epsilon+2n}-\theta_{n,\epsilon})\cdot\psi_k}(|{\bf A}|^{\frac{\epsilon-1}{2}}[A_k,B]|{\bf A}|^{\frac{\epsilon-1}{2}})=$$
$$=\sum_{k=1}^dT^{{\bf A}}_{\phi_{\epsilon+2n}\cdot \psi_k}\Big(|{\bf A}|^{\frac{\epsilon-1}{2}}[A_k,B]|{\bf A}|^{\frac{\epsilon-1}{2}}\Big)-\sum_{k=1}^dT^{{\bf A}}_{\psi_k}\Big(T^{{\bf A}}_{\theta_{n,\epsilon}}\Big(|{\bf A}|^{\frac{\epsilon-1}{2}}[A_k,B]|{\bf A}|^{\frac{\epsilon-1}{2}}\Big)\Big).$$	
It is immediate from the definition of $\theta_{n,\epsilon}$, given before Lemma \ref{phi induction lemma}, that
$$T^{{\bf A}}_{\theta_{n,\epsilon}}\Big(|{\bf A}|^{\frac{\epsilon-1}{2}}[A_k,B]|{\bf A}|^{\frac{\epsilon-1}{2}}\Big)=$$
$$=\sum_{l=1}^n|{\bf A}|^{\frac{\epsilon+2l-1}{2}}\cdot |{\bf A}|^{\frac{\epsilon-1}{2}}[A_k,B]|{\bf A}|^{\frac{\epsilon-1}{2}}\cdot|{\bf A}|^{-\frac{\epsilon+2l-1}{2}}+$$
$$+\sum_{l=1}^n|{\bf A}|^{-\frac{\epsilon+2l-1}{2}}\cdot |{\bf A}|^{\frac{\epsilon-1}{2}}[A_k,B]|{\bf A}|^{\frac{\epsilon-1}{2}}\cdot|{\bf A}|^{\frac{\epsilon+2l-1}{2}}=$$
$$=\sum_{l=1}^n|{\bf A}|^{\epsilon+l-1}[A_k,B]|{\bf A}|^{-l}+\sum_{l=1}^n|{\bf A}|^{-l}[A_k,B]|{\bf A}|^{\epsilon+l-1}.$$
This proves the claimed representation.
\end{proof}

Our goal is to apply the representation formulas in Lemmas \ref{equality for bounded operators} and \ref{improved equality for bounded operators} to (regularizations of) $A=\nabla$ and $B=M_f$. To that end, we now derive trace ideal properties of the operators that appear. It is exactly in the following lemma that we are using the assumption $\epsilon>-\frac{d}{2}.$

\begin{lem}\label{simple cwikel corollary} Let $\epsilon\in(-\frac{d}{2},1)$ with $\epsilon> 0$ when $d=1$. Let $f\in L_{\frac{d}{1-\epsilon}}(\mathbb{R}^d).$
\begin{enumerate}[{\rm (i)}]
\item\label{scca} We have
$$\Big\|(-\Delta)^{\frac{\epsilon-1}{4}}M_f(-\Delta)^{\frac{\epsilon-1}{4}}\Big\|_{\frac{d}{1-\epsilon},\infty}\leq c_{d,\epsilon}\|f\|_{\frac{d}{1-\epsilon}}.$$
\item\label{sccb} If $n\in\mathbb{N}$ is such that $\epsilon\in[-2n-1,-2n+1),$ then we have
$$\Big\|(-\Delta)^{\frac{\epsilon+l-1}{2}}M_f(-\Delta)^{-\frac{l}{2}}\Big\|_{\frac{d}{1-\epsilon},\infty}\leq c'_{d,\epsilon}\|f\|_{\frac{d}{1-\epsilon}},\quad 1\leq l\leq n.$$
\end{enumerate}
\end{lem}

\begin{proof} In the proof of \eqref{scca}, we are only using the assumption $\epsilon\in(1-d,1).$ We write
$$(-\Delta)^{\frac{\epsilon-1}{4}}M_f(-\Delta)^{\frac{\epsilon-1}{4}}=(-\Delta)^{\frac{\epsilon-1}{4}}M_{|f|^{\frac12}}\cdot M_{{\rm sgn}(f)}\cdot M_{|f|^{\frac12}}(-\Delta)^{\frac{\epsilon-1}{4}}.$$
By H\"older's inequality for weak $\mathcal{L}_p$ ideals (see \cite[Theorem 11.6.9]{BiSo-book} and, concerning sharp constants, \cite{SZ2}), we write
$$\Big\|(-\Delta)^{\frac{\epsilon-1}{4}}M_f(-\Delta)^{\frac{\epsilon-1}{4}}\Big\|_{\frac{d}{1-\epsilon},\infty}\leq$$
$$\leq c_{d,\epsilon}^{(1)}\Big\|(-\Delta)^{\frac{\epsilon-1}{4}}M_{|f|^{\frac12}}\Big\|_{\frac{2d}{1-\epsilon},\infty}\cdot \Big\|M_{|f|^{\frac12}}(-\Delta)^{\frac{\epsilon-1}{4}}\Big\|_{\frac{2d}{1-\epsilon},\infty}.$$
Since $\epsilon\in(1-d,1),$ it follows that $\frac{2d}{1-\epsilon}>2.$ 
Note that the function $t\mapsto|t|^{\frac{\epsilon-1}{2}},$ 
$t\in\mathbb{R}^d,$ falls into
 $L_{\frac{2d}{1-\epsilon},\infty}(\mathbb{R}^d).$ Hence, Theorem \ref{cwikel estimate} is applicable to the functions 
 $|f|^{1/2}$ and $g(t)=|t|^{\frac{\epsilon-1}{2}}$ and we can write
$$\Big\|(-\Delta)^{\frac{\epsilon-1}{4}}M_{|f|^{\frac12}}\Big\|_{\frac{2d}{1-\epsilon},\infty}, \Big\|M_{|f|^{\frac12}}(-\Delta)^{\frac{\epsilon-1}{4}}\Big\|_{\frac{2d}{1-\epsilon},\infty}\leq c_{d,\epsilon}^{(2)}\big\||f|^{\frac12}\big\|_{\frac{2d}{1-\epsilon}}=c_{d,\epsilon}^{(2)}\|f\|_{\frac{d}{1-\epsilon}}^{\frac12}.$$
Combining these estimates, we obtain \eqref{scca}.

In the proof of \eqref{sccb}, we note that $\epsilon$ is negative and that $1-\epsilon>n\geq l,$ $1\leq l\leq n.$ We write
$$(-\Delta)^{\frac{\epsilon+l-1}{2}}M_f(-\Delta)^{-\frac{l}{2}}=$$
$$=(-\Delta)^{\frac{\epsilon+l-1}{2}}M_{|f|^{1-\frac{l}{1-\epsilon}}}\cdot M_{{\rm sgn}(f)}\cdot M_{|f|^{\frac{l}{1-\epsilon}}}(-\Delta)^{-\frac{l}{2}}.$$
Using H\"older inequality for weak $\mathcal{L}_p$ ideals, we write
$$\Big\|(-\Delta)^{\frac{\epsilon+l-1}{2}}M_f(-\Delta)^{-\frac{l}{2}}\Big\|_{\frac{d}{1-\epsilon},\infty}\leq$$
$$\leq c_{l,d,\epsilon}\Big\|(-\Delta)^{\frac{\epsilon+l-1}{2}}M_{|f|^{1-\frac{l}{1-\epsilon}}}\Big\|_{\frac{d}{1-\epsilon-l},\infty}\Big\| M_{|f|^{\frac{l}{1-\epsilon}}}(-\Delta)^{-\frac{l}{2}}\Big\|_{\frac{d}{l},\infty}.$$
Note that
$$\frac{d}{l}\geq\frac{d}{n}>2,\quad \frac{d}{1-\epsilon-l}\geq\frac{d}{-\epsilon}>2.$$
Note that the function $t\mapsto |t|^{\epsilon+l-1},$ $t\in\mathbb{R}^d$ (respectively, the function $t\mapsto |t|^{-l},$ $t\in\mathbb{R}^d$) falls into $L_{\frac{d}{1-\epsilon-l},\infty}(\mathbb{R}^d)$ (respectively, into $L_{\frac{d}{l},\infty}(\mathbb{R}^d)$). Hence, Theorem \ref{cwikel estimate} is applicable and we can write
$$\Big\|(-\Delta)^{\frac{\epsilon+l-1}{2}}M_{|f|^{1-\frac{l}{1-\epsilon}}}\Big\|_{\frac{d}{1-\epsilon-l},\infty}\leq c_{d,\epsilon}^{(3)}\||f|^{1-\frac{l}{1-\epsilon}}\|_{\frac{d}{1-\epsilon-l}}=c_{d,\epsilon}^{(3)}\|f\|_{\frac{d}{1-\epsilon}}^{1-\frac{l}{1-\epsilon}},$$
$$\Big\| M_{|f|^{\frac{l}{1-\epsilon}}}(-\Delta)^{-\frac{l}{2}}\Big\|_{\frac{d}{l},\infty}\leq c_{d,\epsilon}^{(4)}\||f|^{\frac{l}{1-\epsilon}}\|_{\frac{d}{l}}=c_{d,l}^{(4)}\|f\|_{\frac{d}{1-\epsilon}}^{\frac{l}{1-\epsilon}}.$$
Combining these estimates, we obtain \eqref{sccb}.
\end{proof}

The following corollary follows immediately by applying Lemma \ref{simple cwikel corollary} with $f$ replaced by $D_k f$.

\begin{lem}\label{xkyk} Let $\epsilon\in(-\frac{d}{2},1)$ with $\epsilon>0$ when $d=1$, and let $n\in\mathbb{Z}_+$ with $\epsilon\in[-2n-1,-2n+1).$ For $f\in C^{\infty}_c(\mathbb{R}^d)$ set 
$$X_k=(-\Delta)^{\frac{\epsilon-1}{4}}M_{D_kf}(-\Delta)^{\frac{\epsilon-1}{4}},\quad 1\leq k\leq d,$$
$$Y_k=\sum_{l=1}^n \left( (-\Delta)^{\frac{\epsilon+l-1}{2}}M_{D_kf}(-\Delta)^{-\frac{l}{2}}+(-\Delta)^{-\frac{l}{2}}M_{D_kf}(-\Delta)^{\frac{\epsilon+l-1}{2}}\right) ,\quad 1\leq k\leq d.$$
Here, for $n=0,$ $Y_k$ is assumed to be zero. Then 
$$\|X_k\|_{\frac{d}{1-\epsilon},\infty},\|Y_k\|_{\frac{d}{1-\epsilon},\infty}\leq\|f\|_{\dot{W}^1_{\frac{d}{1-\epsilon}}}.
$$
\end{lem}

We are now in position to state and prove a useful representation formula for the commutators of interest.
	
\begin{prop}\label{equality for gradient} 
	Let $d\geq 2$ and let $\epsilon\in(-\frac{d}{2},1)$ (alternatively, let $d=1$ and $\epsilon\in(0,1)$). Choose $n\in\mathbb Z_+$ with $\epsilon\in[-2n-1,2n+1)$. For every $f\in C^{\infty}_c(\mathbb{R}^d),$ we have
$$
[(-\Delta)^{\frac{\epsilon}{2}},M_f]=\sum_{k=1}^dT^{\nabla}_{\psi_k}\Big(T^{\nabla}_{\phi_{\epsilon+2n}}(X_k)\Big)-\sum_{k=1}^dT^{\nabla}_{\psi_k}(Y_k).
$$
Here, $X_k$ and $Y_k$ are given in Lemma \ref{xkyk}, $\phi_{\epsilon}$ is given in \eqref{phi def} and $\psi_k$ is given in \eqref{eq:defpsik}.
\end{prop}

\begin{proof} Fix $m\in\mathbb{N}.$ Let 
$P_m=\chi_{(\frac1m,m)}(-\Delta).$ Consider the tuple $
{\bf A}=\nabla\cdot P_m$ and $B=P_mM_fP_m$ on the Hilbert space $P_m(L_2(\mathbb{R}^d)).$ Clearly, ${\bf A}$ is bounded and $|{\bf A}|$ is bounded from below. Also, $B\in\mathcal{L}_2$ (due to the fact that $M_fg(\nabla)$  belongs to $ \mathcal{L}_2$ whenever $f,g\in L_2(\mathbb{R}^d)$). We  now apply 
Lemma \ref{improved equality for bounded operators} 
(for $n\in\mathbb{N}$) or 
Lemma \ref{equality for bounded operators} (for $n=0$). Abbreviating $\delta:=\epsilon+2n$ we obtain
\begin{align*}
		T_m & :=(P_m(-\Delta)^{\frac{\epsilon}{2}})\cdot M_f\cdot P_m-P_m\cdot M_f\cdot(P_m(-\Delta)^{\frac{\epsilon}{2}})=[|{\bf A}|^{\epsilon},B] \\
		& =\sum_{k=1}^dT^{{\bf A}}_{\psi_k}\Big(T^{{\bf A}}_{\phi_{\delta}}\Big(|{\bf A}|^{\frac{\epsilon-1}{2}}[A_k,B]|{\bf A}|^{\frac{\epsilon-1}{2}}\Big)\Big) \\& \quad -\sum_{k=1}^dT^{{\bf A}}_{\psi_k}\Big(\sum_{l=1}^n|{\bf A}|^{\epsilon+l-1}[A_k,B]|{\bf A}|^{-l}+\sum_{l=1}^n|{\bf A}|^{-l}[A_k,B]|{\bf A}|^{\epsilon+l-1}\Big).
\end{align*}
Since
$$|{\bf A}|^{\frac{\epsilon-1}{2}}[A_k,B]|{\bf A}|^{\frac{\epsilon-1}{2}}=P_mX_kP_m,$$
$$\sum_{l=1}^n|{\bf A}|^{\epsilon+l-1}[A_k,B]|{\bf A}|^{-l}+\sum_{l=1}^n|{\bf A}|^{-l}[A_k,B]|{\bf A}|^{\epsilon+l-1}=P_mY_kP_m,$$
we see that 
$$T_m=\sum_{k=1}^dT^{\nabla}_{\psi_k}\Big(T^{\nabla}_{\phi_{\delta}}\Big(P_m\cdot X_k\cdot P_m\Big)\Big)-\sum_{k=1}^dT^{\nabla}_{\psi_k}\Big(P_m\cdot Y_k\cdot P_m\Big).$$

We now want to pass to the limit $m\to\infty$. We have, by Lemma \ref{xkyk},
$$X_k,Y_k\in\mathcal{L}_{\frac{d}{1-\epsilon},\infty}\subset\mathcal{L}_p,\quad \frac{d}{1-\epsilon}<p<\infty.$$
Since $P_m\uparrow 1$ strongly, it follows that
$$P_m\cdot X_k\cdot P_m\to X_k,\quad P_m\cdot Y_k\cdot P_m\to Y_k
\qquad\text{in}\ \mathcal L_p \,.
$$
The operators $T^{\nabla}_{\phi_{\delta}}$ and $T^{\nabla}_{\psi_k}$ are bounded in $\mathcal{L}_p$ by Lemma \ref{tboundedphi} and Lemma \ref{tboundedpsi}, respectively. It follows that
$$
T^{\nabla}_{\psi_k}\Big(T^{\nabla}_{\phi_{\delta}}\Big(P_m\cdot X_k\cdot P_m\Big)\Big)\to T^{\nabla}_{\psi_k}\Big(T^{\nabla}_{\phi_{\delta}}(X_k)\Big), \qquad
T^{\nabla}_{\psi_k}(P_m\cdot Y_k\cdot P_m)\to T^{\nabla}_{\psi_k}(Y_k)
$$
in $\mathcal{L}_p$. We therefore have
$$
\mathcal{L}_p\ni T_\infty := \lim_{m\to\infty} T_m = \sum_{k=1}^dT^{\nabla}_{\psi_k}\Big(T^{\nabla}_{\phi_{\delta}}( X_k )\Big)-\sum_{k=1}^dT^{\nabla}_{\psi_k}( Y_k ).
$$

To complete the proof, it remains to notice that $T_{\infty}:L_2(\mathbb{R}^d)\to L_2(\mathbb{R}^d)$ is the unique bounded extension of the operator $[(-\Delta)^{\frac{\epsilon}{2}},M_f]:C^{\infty}_c(\mathbb{R}^d)\to (C^\infty_c(\mathbb R^d))'$ defined via Lemma \ref{definition}. Indeed, by Lemma \ref{tricky convergence} we have for all $\phi\in C_c^\infty(\mathbb R^d)$
$$
T_m\phi \to (-\Delta)^\frac\epsilon2 M_f \phi - M_f (-\Delta)^\frac\epsilon2 \phi = [(-\Delta)^{\frac{\epsilon}{2}},M_f] \phi
$$
in $(C_c^\infty(\mathbb R^d))'$. This implies
$$
T_{\infty}\phi=[(-\Delta)^{\frac{\epsilon}{2}},M_f]\phi,\quad \phi\in C^{\infty}_c(\mathbb{R}^d),
$$
and therefore proves the claim.
\end{proof}

\begin{rem}
	In the proof above, we used the folklore result that in any separable Banach ideal $(\mathcal E, \|\cdot\|_{\mathcal E})$ and for any sequence of projections $P_m\uparrow 1$, we have 
	$\|x-P_mxP_m\|_{\mathcal E}\to 0$ as $m\to \infty$. For a proof of a similar but more general fact, we refer to \cite[Proposition 2.5]{CS}. In particular, this fact holds in any Schatten ideal $\mathcal{L}_p$, $1\leq p<\infty$, but fails in the non-separable ideals $\mathcal{L}_{p,\infty}.$
\end{rem}
	
We can now derive the trace ideal inequality in Theorem \ref{main theorem positive} \eqref{mtpa} and Theorem \ref{main theorem negative} \eqref{mtna} for smooth functions.
	
\begin{lem}\label{boundcinfty} 
Let $\epsilon\in(-\frac d2,1)$ with $\epsilon>0$ when $d=1$. For every $f\in C^{\infty}_c(\mathbb{R}^d),$ we have
$$\Big\|[(-\Delta)^{\frac{\epsilon}{2}},M_f]\Big\|_{\frac{d}{1-\epsilon},\infty}\leq c_{d,\epsilon}\|f\|_{\dot{W}^{1}_{\frac{d}{1-\epsilon}}(\mathbb{R}^d)}.$$
\end{lem}	

\begin{proof} 
Choosing $n$  as in Proposition \ref{equality for gradient} and applying its result, we find
$$
[(-\Delta)^{\frac{\epsilon}{2}},M_f]=\sum_{k=1}^dT^{\nabla}_{\psi_k}\Big(T^{\nabla}_{\phi_{\epsilon+2n}}(X_k)\Big)-\sum_{k=1}^dT^{\nabla}_{\psi_k}(Y_k).
$$
By the triangle inequality, we obtain
$$
\Big\|[(-\Delta)^{\frac{\epsilon}{2}},M_f]\Big\|_{\frac{d}{1-\epsilon},\infty}\leq$$
$$\leq c_{d,\epsilon}'\Big(\sum_{k=1}^d\Big\|T^{\nabla}_{\psi_k}\Big(T^{\nabla}_{\phi_{\epsilon+2n}}(X_k)\Big)\Big\|_{\frac{d}{1-\epsilon},\infty}+\sum_{k=1}^d\Big\|T^{\nabla}_{\psi_k}(Y_k)\Big\|_{\frac{d}{1-\epsilon},\infty}\Big)\leq
$$
$$\leq c_{d,\epsilon}'\Big\|T^{\nabla}_{\psi_k}\Big\|_{\mathcal{L}_{\frac{d}{1-\epsilon},\infty}\to\mathcal{L}_{\frac{d}{1-\epsilon},\infty}}\cdot\Big(1+\Big\|T^{\nabla}_{\phi_{\epsilon+2n}}\Big\|_{\mathcal{L}_{\frac{d}{1-\epsilon},\infty}\to\mathcal{L}_{\frac{d}{1-\epsilon},\infty}}\Big)\cdot$$
$$\cdot\Big(\sum_{k=1}^d\|X_k\|_{\frac{d}{1-\epsilon},\infty}+\sum_{k=1}^d\|Y_k\|_{\frac{d}{1-\epsilon},\infty}\Big).
$$
By Lemma \ref{tboundedpsi}, the first factor is finite (its value depends only on $d$ and $\epsilon$). By Lemma \ref{tboundedphi}, the second factor is finite (its value depends only on $d$ and $\epsilon$).	The assertion follows now from Lemma \ref{xkyk}.
\end{proof}

\begin{fact}\label{distl2conv}
	Let $(A_n)$ be a sequence of bounded operators with $\sup_n \|A_n\|_\infty<\infty$ and let $A:C^{\infty}_c(\mathbb{R}^d)\to (C^{\infty}_c(\mathbb{R}^d))'$ be such that $A_n\phi\to A\phi$ in the sense of distributions for every $\phi\in C^{\infty}_c(\mathbb{R}^d).$ Then $A$ extends to a bounded operator on $L_2(\mathbb R^d)$ and $A_n\to A$ strongly.
\end{fact}

\begin{proof}
	Let $M:=\sup_n \|A_n\|_\infty$. For $\phi,\psi\in C_c^\infty(\mathbb{R}^d)$ we have
	$$
	|\langle \psi,A\phi\rangle | = \lim_{n\to\infty} |\langle \psi,A_n\phi \rangle | \leq M \|\psi\|_2\|\phi\|_2 \,.
	$$
	By density and the Riesz representation theorem, this implies $A\phi\in L_2(\mathbb{R}^d)$ with $\|A\phi\|_2\leq M \|\phi\|_2$. Consequently, $A$ extends to a bounded operator on $L_2(\mathbb{R}^d)$ and one easily verifies that $A_n f\to Af$ for any $f\in L_2(\mathbb{R}^d)$.
\end{proof}

\begin{proof}[Proof of Theorem \ref{main theorem positive} \eqref{mtpa}] 
The assertion follows from Lemma \ref{boundcinfty} by a simple density argument. Let us give the details.
	
Let $f\in\dot W^1_{\frac{d}{1-\epsilon}}(\mathbb{R}^d).$ It is well known (see, e.g., \cite[Theorem 11.43]{Leoni-book}) that 
$C^{\infty}_c(\mathbb{R}^d)$ is dense in 
$\dot{W}^{1}_{\frac{d}{1-\epsilon}}(\mathbb{R}^d).$ (Observe that  \cite[Theorem 11.43]{Leoni-book} requires that either $d\geq 2$ or else $p>1.$ In our case, if  $d<2,$ then $d=1$ and $p=\frac1{1-\epsilon}.$ So, $p>1$ means $\epsilon\in(0,1),$ satisfying the assumption.) Choose a sequence 
$\{f_n\}_{n\geq0}\subset C^\infty_c(\mathbb{R}^d)$ such that 
$f_n\to f$ in $\dot{W}^{1}_{\frac{d}{1-\epsilon}}(\mathbb{R}^d).$  The proof of \cite[Theorem 11.43]{Leoni-book} also shows that there is a sequence 
$(c_n)\subset\mathbb C$ such that $f_n-c_n \to f$ in the space $L_{\frac{d}{1-\epsilon},{\rm loc}}(\mathbb R^d)$. (Indeed, $c_n$ can be chosen as the mean value of $f$ over $\{ x\in\mathbb R^d:\ n<|x|<2n\}$.)

Denote, for brevity,
$$A_n:=[(-\Delta)^{\frac{\epsilon}{2}},M_{f_n}] = [(-\Delta)^{\frac{\epsilon}{2}},M_{f_n-c_n}],\qquad A:=[(-\Delta)^{\frac{\epsilon}{2}},M_f].$$
Since $f_n-c_n\to f$ in $L_{1,{\rm loc}}(\mathbb{R}^d),$ it follows from Lemma \ref{definition} that $A_n\phi\to A\phi$ in $(C^{\infty}_c(\mathbb{R}^d))'$ for every $\phi\in C^{\infty}_c(\mathbb{R}^d).$ By Fact \ref{distl2conv} and the Fatou property of the ideal $\mathcal{L}_{\frac{d}{1-\epsilon},\infty}$ (see, e.g., \cite[Theorem 2.7 (d)]{Simon-book}), we have
$$
\|A\|_{\frac{d}{1-\epsilon},\infty}\leq \liminf_{n\to\infty} \|A_n\|_{\frac{d}{1-\epsilon},\infty}
$$
Since, by Lemma \ref{boundcinfty},
$$
\|A_n\|_{\frac{d}{1-\epsilon},\infty}\leq c_{d,\epsilon} \|f_n\|_{\dot{W}^{1}_{\frac{d}{1-\epsilon}}},\quad n\geq 0,
$$	
we obtain the assertion.
\end{proof}

\begin{proof}[Proof of Theorem \ref{main theorem negative} \eqref{mtna}] The assertion follows from Lemma \ref{boundcinfty} by a simple density argument. Let us give the details.
	
Let $f\in\dot W^1_{\frac{d}{1-\epsilon}}(\mathbb{R}^d).$ It is well known (see, e.g., \cite[Theorem 11.43]{Leoni-book}) that $C^{\infty}_c(\mathbb{R}^d)$ is dense in $\dot{W}^{1}_{\frac{d}{1-\epsilon}}(\mathbb{R}^d).$ Choose a sequence $(f_n)\subset C^\infty_c(\mathbb{R}^d)$ such that $f_n\to f$ in $\dot{W}^{1}_{\frac{d}{1-\epsilon}}(\mathbb{R}^d).$ 

Denote for brevity
$$A_n=[(-\Delta)^{\frac{\epsilon}{2}},M_{f_n}],\quad A=[(-\Delta)^{\frac{\epsilon}{2}},M_f].$$
It follows from \eqref{uniform norm negative equation} that
$$\|A_n-A\|_{\infty}\leq c_{d,\epsilon}^{(1)}\|f_n-f\|_{\dot{W}^{1}_{\frac{d}{1-\epsilon}}}.$$
	
Recall from Lemma \ref{boundcinfty} that
$$\|A_n\|_{\frac{d}{1-\epsilon},\infty}\leq c_{d,\epsilon} \|f_n\|_{\dot{W}^{1}_{\frac{d}{1-\epsilon}}},\quad n\geq 0.$$	
Using the Fatou property of the ideal $\mathcal{L}_{\frac{d}{1-\epsilon},\infty}$ (see, e.g., \cite[Theorem 2.7 (d)]{Simon-book}), we obtain
$$\|A\|_{\frac{d}{1-\epsilon},\infty} \leq\limsup_{n\to\infty}\|A_n\|_{\frac{d}{1-\epsilon},\infty}
\leq c_{d,\epsilon} \limsup_{n\to\infty}\|f_n\|_{\dot{W}^{1}_{\frac{d}{1-\epsilon}}} 
= c_{d,\epsilon} \|f\|_{\dot{W}^{1}_{\frac{d}{1-\epsilon}}},$$
as claimed.
\end{proof}


\section{Approximate expression for commutator}\label{sec:approximate}

In this section, we prove an approximation results, which provides the leading term of the commutator $[(1-\Delta)^{\frac{\epsilon}{2}},M_f]$. This approximation is needed in the proof of Theorem \ref{main theorem positive} \eqref{mtpb} and Theorem \ref{main theorem negative} \eqref{mtnb}.

\begin{thm}\label{approximation theorem} 
	Let $d\geq 2$ and $\epsilon\in(-\frac d2,1)$ (alternatively, let $d=1$ and $\epsilon\in(0,1)$). For every $f\in C_c^{\infty}(\mathbb{R}^d)$ we have
$$
[(1-\Delta)^{\frac{\epsilon}{2}},M_f] \in -\tfrac{\epsilon}{2}[\Delta,M_f](1-\Delta)^{\frac{\epsilon}{2}-1}+(\mathcal{L}_{\frac{d}{1-\epsilon},\infty})_0.
$$
\end{thm}

The rest of this section is devoted to the proof of this theorem.

Let $(\Omega,\nu)$ be a measure space and $H$ a complex, separable Hilbert space.  Recall that a function $f : (\Omega,\nu) \rightarrow \Bc(H)$ is called \emph{measurable in the weak operator topology} if for every pair of vectors $\xi , \eta \in H$ the function $s \mapsto \langle f(s) \xi  , \eta \rangle$ is measurable. A function $f : (\Omega,\nu) \rightarrow \Bc(H)$ is said to be integrable in the weak operator topology if it is measurable in the weak operator topology and $\int_{\Omega} \|f(s)\|_{\infty} ds <\infty  $; for details see, e.g., \cite[Subsection 2.7]{CSZ}. In this case, $\int_{\Omega}  f(s) ds $ determines uniquely an operator in $\Bc(H)$ called the weak integral. 
If the function $F : (0,\infty) \rightarrow \mathcal{L}_{p,\infty}$, $1<p<\infty$, is continuous in the weak operator topology,  and $\int_0^{\infty}\|F(\lambda)\|_{p,\infty}d\lambda<\infty$, then its  weak integral exists and belongs to  $ \mathcal{L}_{p,\infty}$. Furthermore, we record the following fact, established in \cite[Proposition 2.3.2]{SZ}.

\begin{fact}\label{general integral estimate} For every $1<p<\infty$ there is a constant $c_p$ such that, if $F:(0,\infty)\to\mathcal{L}_{p,\infty}$ is continuous in the weak operator topology, then
$$\Big\|\int_0^{\infty}F(\lambda)d\lambda\Big\|_{p,\infty}\leq c_p\int_0^{\infty}\|F(\lambda)\|_{p,\infty}d\lambda.$$
\end{fact}

\begin{fact}\label{commutator lemma} Let $f\in C^{\infty}_c(\mathbb{R}^d).$ For every $n\geq2,$ we have
$$[\frac1{1+\lambda-\Delta},M_f]=\sum_{k=1}^{n-1}A_{k,f}\frac1{(1+\lambda-\Delta)^{k+1}}+\frac1{1+\lambda-\Delta}A_{n,f}\frac1{(1+\lambda-\Delta)^n}.$$
Here, we use the inductive notation
$$A_{0,f}=M_f,\quad A_{k+1,f}=[\Delta,A_{k,f}],\quad k\geq0.$$
\end{fact}

\begin{lem}\label{first concrete integral lemma} 
Let $\epsilon\in(0,1).$ For $f\in C^{\infty}_c(\mathbb{R}^d)$ we have
$$\int_0^{\infty}\frac1{1+\lambda-\Delta}A_{2,f}\frac1{(1+\lambda-\Delta)^2}\cdot \lambda^{\frac{\epsilon}{2}}d\lambda\in\mathcal{L}_{d,\infty}.$$
\end{lem}

\begin{proof} 
We first assume $d\geq 2$ and bound
\begin{align*}
& \Big\|\frac1{1+\lambda-\Delta}A_{2,f}\frac1{(1+\lambda-\Delta)^2}\Big\|_{d,\infty} \\
& \leq \Big\|\frac1{1+\lambda-\Delta}\Big\|_{\infty} \Big\|A_{2,f}(1-\Delta)^{-\frac32}\Big\|_{d,\infty} \Big\|\frac{(1-\Delta)^{\frac32}}{(1+\lambda-\Delta)^2}\Big\|_{\infty} \,.
\end{align*}
One easily finds that
$$\Big\|\frac1{1+\lambda-\Delta}\Big\|_{\infty} \leq \frac{1}{1+\lambda},\quad
\Big\|\frac{(1-\Delta)^{\frac32}}{(1+\lambda-\Delta)^2}\Big\|_{\infty} \leq \frac{1}{(1+\lambda)^\frac12}.$$
Moreover, $A_{2,f}$ is a second order differential operator with bounded, compactly supported coefficients. Explicitly,
$$
A_{2,f} = 4 \sum_{j,k} M_{\partial_j\partial_k f} \partial_j \partial_k + 4 \sum_j M_{\partial_j \Delta f} \partial_j + M_{\Delta^2 f} \,.
$$
We claim that by Cwikel's estimate (Theorem \ref{cwikel estimate} (i) for $d\geq 3$ and Theorem \ref{cwikel estimate} (ii) for $d=2$) we have $A_{2,f}(1-\Delta)^{-\frac32}\in \mathcal L_{d,\infty}$. To see this, we write $\partial_j\partial_k (1-\Delta)^{-\frac32}$ as $(1-\Delta)^{-\frac12}$ times the bounded operator $\partial_j\partial_k (1-\Delta)^{-1}$. The operator $M_{\partial_j\partial_k f}(1-\Delta)^{-\frac12}$ belongs to the claimed trace ideal by Cwikel's bound. The other terms in the expression for $A_{2,f}$ can be handled similarly (and, in fact, enjoy better trace ideal properties than $\mathcal L_{d,\infty}$). Thus, we have shown that
$$\Big\|\frac1{1+\lambda-\Delta}A_{2,f}\frac1{(1+\lambda-\Delta)^2}\Big\|_{d,\infty} \leq\frac{c_{n,f}}{(1+\lambda)^{\frac32}}.$$
Hence, the integrand is absolutely integrable in $\mathcal{L}_{d,\infty}$ and the assertion of the lemma for $d\geq 2$ follows from Fact \ref{general integral estimate}.

The proof for $d=1$ is a variation of this argument. We bound
\begin{align*}
& \Big\|\frac1{1+\lambda-\Delta}A_{2,f}\frac1{(1+\lambda-\Delta)^2}\Big\|_1 \\
& \leq \Big\|\frac1{1+\lambda-\Delta}\Big\|_{\infty} \Big\|A_{2,f}(1-\Delta)^{\frac{\epsilon}{2}-2}\Big\|_1 \Big\|\frac{(1-\Delta)^{2-\frac{\epsilon}{2}}}{(1+\lambda-\Delta)^2}\Big\|_{\infty}.
\end{align*}
One easily finds that
$$\Big\|\frac1{1+\lambda-\Delta}\Big\|_{\infty} \leq \frac{1}{1+\lambda},\quad
\Big\|\frac{(1-\Delta)^{2-\frac{\epsilon}{2}}}{(1+\lambda-\Delta)^2}\Big\|_{\infty} \leq \frac{1}{(1+\lambda)^{\frac{\epsilon}{2}}}.$$
Furthermore, we have $A_{2,f}(1-\Delta)^{\frac{\epsilon}{2}-2}\in\mathcal{L}_1$ using Theorem \ref{pre-cwikel estimate}. We, therefore, established the inequality
$$\Big\|\frac1{1+\lambda-\Delta}A_{2,f}\frac1{(1+\lambda-\Delta)^2}\Big\|_1\leq\frac{c_{n,f}}{(1+\lambda)^{1+\frac{\epsilon}{2}}}.$$
The integrand is absolutely integrable in $\mathcal{L}_1.$ Using Fact \ref{general integral estimate} (with $\mathcal{L}_1$ instead of $\mathcal{L}_{d,\infty}$), we infer that the integral in the assertion belongs to $\mathcal{L}_1$ and, therefore, to $\mathcal{L}_{1,\infty}.$  
\end{proof}

\begin{proof}[Proof of Theorem \ref{approximation theorem} for $0<\epsilon<1$] Set $\delta = 2-\epsilon \in(1,2).$ Using the functional calculus, we write (see, e.g., \cite[Theorem 1, p.232]{BiSo-book})
$$(1-\Delta)^{\frac{\epsilon}{2}}=\frac{\sin(\frac{\pi\epsilon}{2})}{\pi}\int_0^{\infty}\frac{1-\Delta}{1+\lambda-\Delta}\lambda^{-\frac{\delta}{2}}d\lambda.$$
Therefore,
$$[(1-\Delta)^{\frac{\epsilon}{2}},M_f]=\frac{\sin(\frac{\pi\epsilon}{2})}{\pi}\int_0^{\infty}[\frac{1-\Delta}{1+\lambda-\Delta},M_f]\lambda^{-\frac{\delta}{2}}d\lambda.$$
It is immediate that
$$[\frac{1-\Delta}{1+\lambda-\Delta},M_f]=[1-\frac{\lambda}{1+\lambda-\Delta},M_f]=-[\frac1{1+\lambda-\Delta},M_f]\cdot\lambda.$$
Hence,
$$[(1-\Delta)^{\frac{\epsilon}{2}},M_f]=-\frac{\sin(\frac{\pi\epsilon}{2})}{\pi}\int_0^{\infty}[\frac1{1+\lambda-\Delta},M_f]\lambda^{\frac{\epsilon}{2}}d\lambda.$$
Using Fact \ref{commutator lemma} with $n=2,$ we write
$$[(1-\Delta)^{\frac{\epsilon}{2}},M_f]=-A_{1,f}\cdot \frac{\sin(\frac{\pi\epsilon}{2})}{\pi}\int_0^{\infty}\frac{\lambda^{\frac{\epsilon}{2}}d\lambda}{(1+\lambda-\Delta)^2}-$$
$$-\frac{\sin(\frac{\pi\epsilon}{2})}{\pi}\int_0^{\infty}\frac1{1+\lambda-\Delta}A_{2,f} \frac1{(1+\lambda-\Delta)^2}\lambda^{\frac{\epsilon}{2}}d\lambda.$$
It follows from Lemma \ref{first concrete integral lemma} that
$$\int_0^{\infty}\frac1{1+\lambda-\Delta}A_{2,f} \frac1{(1+\lambda-\Delta)^2}\lambda^{\frac{\epsilon}{2}}d\lambda\in\mathcal{L}_{d,\infty} \subset (\mathcal L_{\frac d{1-\epsilon},\infty})_0.$$
Again by the functional calculus, we have
$$\int_0^{\infty}\frac{\lambda^{\frac{\epsilon}{2}}d\lambda}{(1+\lambda-\Delta)^2}=(1-\Delta)^{\frac{\epsilon}{2}-1}B(1+\tfrac{\epsilon}{2},1-\tfrac{\epsilon}{2})=\frac{\frac{\pi\epsilon}{2}}{\sin(\frac{\pi\epsilon}{2})}(1-\Delta)^{\frac{\epsilon}{2}-1},$$
where $B(\cdot,\cdot)$ denotes the beta function. This completes the proof. 
\end{proof}

\begin{lem}\label{second concrete integral lemma} Let $\alpha,\beta>0$ and $\delta\in(0,2).$ Suppose $p>1$ is such that $p\geq\frac{d}{\alpha+\beta+2}.$ We have
$$(1-\Delta)^{-\frac{\alpha}{2}}\cdot\int_0^{\infty}\frac1{1+\lambda-\Delta}A_{2,f} \frac1{(1+\lambda-\Delta)^2}\lambda^{-\frac{\delta}{2}}d\lambda\cdot(1-\Delta)^{-\frac{\beta}{2}}\in\mathcal{L}_{p,\infty}.$$
\end{lem}

\begin{proof}
	We claim that
\begin{equation}
	\label{eq:traceidealbound}
	(1-\Delta)^{-\frac{\alpha}{2}}M_h(1-\Delta)^{-\frac{\beta}{2}-1}\in\mathcal{L}_{\frac{d}{\alpha+\beta+2},\infty},\quad h\in C_c(\mathbb{R}^d).
\end{equation}
By complex interpolation \cite[Theorem 2.9]{Simon-book}, it suffices to prove that	
$$M_h(1-\Delta)^{-\frac{\alpha+\beta+2}{2}}\in\mathcal{L}_{\frac{d}{\alpha+\beta+2},\infty},\quad h\in C_c(\mathbb{R}^d).$$		
The latter inclusion follows from Theorem \ref{cwikel estimate}, thus proving \eqref{eq:traceidealbound}.
	
We now note that
$$(1-\Delta)^{-\frac{\alpha}{2}}A_{2,f}(1-\Delta)^{-\frac{\beta}{2}-2}\in\mathcal{L}_{\frac{d}{\alpha+\beta+2},\infty}.$$	
Indeed, $A_{2,f}$ is a differential operator of second order with smooth, compactly supported coefficients. Hence, the inclusion follows from \eqref{eq:traceidealbound}.

We now argue similarly as in the proof of Lemma \ref{first concrete integral lemma}. We bound 
$$\Big\|(1-\Delta)^{-\frac{\alpha}{2}}\cdot\frac1{1+\lambda-\Delta}A_{2,f} \frac1{(1+\lambda-\Delta)^2}\cdot(1-\Delta)^{-\frac{\beta}{2}}\Big\|_{\frac{d}{\alpha+\beta+2},\infty}\leq$$
$$\leq\frac1{1+\lambda}\Big\|(1-\Delta)^{-\frac{\alpha}{2}}A_{2,f}(1-\Delta)^{-\frac{\beta}{2}-2}\Big\|_{\frac{d}{\alpha+\beta+2},\infty}\leq\frac{c_{\alpha,\beta,f}}{1+\lambda}.$$
Hence, the integrand is absolutely integrable in $\mathcal{L}_{\frac{d}{\alpha+\beta+2},\infty}.$ Consequently, it is absolutely integrable in $\mathcal{L}_{p,\infty}$ and the assertion follows from Fact \ref{general integral estimate}.
\end{proof}

\begin{proof}[Proof of Theorem \ref{approximation theorem} for $-\frac{d}{2}<\epsilon<0$, $d\geq 2$] 
	Choose $0<\delta<1$ and $m\in\mathbb{N}$ such that $\epsilon=-m\delta.$ Using Leibniz's rule, we write
$$[(1-\Delta)^{\frac{\epsilon}{2}},M_f]=[(1-\Delta)^{-\frac{m\delta}{2}},M_f]=$$
$$=\sum_{l=0}^{m-1}(1-\Delta)^{-\frac{l\delta}{2}}[(1-\Delta)^{-\frac{\delta}2},M_f](1-\Delta)^{\frac{(l+1-m)\delta}{2}}.$$
Using the functional calculus, we write
$$(
1-\Delta)^{-\frac{\delta}{2}}=\frac{\sin(\frac{\pi\delta}{2})}{\pi}\int_0^{\infty}\frac{\lambda^{-\frac{\delta}{2}}\,d\lambda}{1+\lambda-\Delta} \,,
$$
where the integrand is a norm-continuous  function of $\lambda$ and the integral converges in operator norm.

We have
$$[(1-\Delta)^{-\frac{\delta}{2}},M_f]=\frac{\sin(\frac{\pi\delta}{2})}{\pi}\int_0^{\infty}[\frac1{1+\lambda-\Delta},M_f]\cdot\lambda^{-\frac{\delta}{2}}d\lambda.$$
Using Fact \ref{commutator lemma} with $n=2,$ we write
\begin{align*}
	[(1-\Delta)^{-\frac{\delta}{2}},M_f] & =A_{1,f}\cdot \frac{\sin(\frac{\pi\delta}{2})}{\pi}\int_0^{\infty}\frac{\lambda^{-\frac{\delta}{2}}d\lambda}{(1+\lambda-\Delta)^2} \\
	& \quad +\frac{\sin(\frac{\pi\delta}{2})}{\pi}\int_0^{\infty}\frac1{1+\lambda-\Delta}A_{2,f} \frac1{(1+\lambda-\Delta)^2}\lambda^{-\frac{\delta}{2}}d\lambda.
\end{align*}
Thus,
$$[(1-\Delta)^{\frac{\epsilon}{2}},M_f]=$$
$$=\sum_{l=0}^{m-1}(1-\Delta)^{-\frac{l\delta}{2}}A_{1,f}\cdot \frac{\sin(\frac{\pi\delta}{2})}{\pi}\int_0^{\infty}\frac{\lambda^{-\frac{\delta}{2}}d\lambda}{(1+\lambda-\Delta)^2}\cdot (1-\Delta)^{\frac{(l+1-m)\delta}{2}}+$$
$$+\sum_{l=0}^{m-1}\frac{\sin(\frac{\pi\delta}{2})}{\pi}(1-\Delta)^{-\frac{l\delta}{2}}\cdot\int_0^{\infty}\frac1{1+\lambda-\Delta}A_{2,f} \frac1{(1+\lambda-\Delta)^2}\lambda^{-\frac{\delta}{2}}d\lambda\cdot (1-\Delta)^{\frac{(l+1-m)\delta}{2}}.$$

Choose $p>1$ such that $\frac{d}{2-\epsilon-\delta}\leq p<\frac{d}{1-\epsilon}.$ (Note that $\frac{d}{1-\epsilon}>1$ since $\epsilon>-\frac d2$ and $d\geq 2$.) It follows from Lemma \ref{second concrete integral lemma} that
$$(1-\Delta)^{-\frac{l\delta}{2}}\cdot\int_0^{\infty}\frac1{1+\lambda-\Delta}A_{2,f} \frac1{(1+\lambda-\Delta)^2}\lambda^{-\frac{\delta}{2}}d\lambda\cdot (1-\Delta)^{\frac{(l+1-m)\delta}{2}}\in\mathcal{L}_{p,\infty}.$$
Thus,
$$[(1-\Delta)^{\frac{\epsilon}{2}},M_f]\in$$
$$\in\sum_{l=0}^{m-1}(1-\Delta)^{-\frac{l\delta}{2}}A_{1,f}\cdot \frac{\sin(\frac{\pi\delta}{2})}{\pi}\int_0^{\infty}\frac{\lambda^{-\frac{\delta}{2}}d\lambda}{(1+\lambda-\Delta)^2}\cdot (1-\Delta)^{\frac{(l+1-m)\delta}{2}}+(\mathcal{L}_{\frac{d}{1-\epsilon},\infty})_0.$$
Again by the functional calculus, we have
$$\int_0^{\infty}\frac{\lambda^{-\frac{\delta}{2}}d\lambda}{(1+\lambda-\Delta)^2}=(1-\Delta)^{-\frac{\delta}{2}-1}\cdot B(\tfrac{\delta}{2}+1,1-\tfrac{\delta}{2})=\frac{\frac{\pi\delta}{2}}{\sin(\frac{\pi\delta}{2})}(1-\Delta)^{-\frac{\delta}{2}-1}.$$
Thus,
$$[(1-\Delta)^{\frac{\epsilon}{2}},M_f]\in\tfrac{\delta}{2}\sum_{l=0}^{m-1}(1-\Delta)^{-\frac{l\delta}{2}}A_{1,f}(1-\Delta)^{\frac{(l-m)\delta}{2}-1}+(\mathcal{L}_{\frac{d}{1-\epsilon},\infty})_0.$$
With the help of Cwikel estimates (Theorem \ref{cwikel estimate}) it is easy to see that
$$[(1-\Delta)^{-\frac{l\delta}{2}},A_{1,f}](1-\Delta)^{\frac{(l-m)\delta}{2}-1}\in (\mathcal{L}_{\frac{d}{1-\epsilon},\infty})_0,\quad 0\leq l<m.$$
Thus,
$$[(1-\Delta)^{\frac{\epsilon}{2}},M_f]\in \tfrac{\delta}{2}\sum_{l=0}^{m-1} A_{1,f}(1-\Delta)^{-\frac{l\delta}{2}} (1-\Delta)^{\frac{(l-m)\delta}{2}-1} + (\mathcal{L}_{\frac{d}{1-\epsilon},\infty})_0.$$
Since all summands on the right hand side are equal, the assertion follows.
\end{proof}


\section{Spectral asymptotics for pseudo-differential operators}\label{sec:spectralasymp}

Our goal in this section is to extend the main result in \cite{FrSuZa}.

Let $\Pi$ be the $C^{\ast}$-subalgebra of $\Bc(L_2(\mathbb R^d))$ generated by the algebras
$$\{M_f:\ f\in\mathbb C+C_0(\mathbb R^d)\}
\qquad\text{and}\qquad
\{g(\nabla(-\Delta)^{-\frac12}):\ g\in C(\mathbb S^{d-1})\}.$$
According to \cite{SZ-DAO1} (where a much stronger result is given in Theorem 1.2) or \cite{MSZ-DAO2} (where a very general result is given in Theorem 3.3 and examplified on p.~284), there is a $\ast$-homomorphism
$$
{\rm sym} : \Pi \to C(\mathbb S^{d-1},\mathbb C + C_0(\mathbb R^d))
$$
such that, for all $f\in \mathbb C+C_0(\mathbb R^d)$ and $g\in C(\mathbb S^{d-1})$,
$$
{\rm sym} (M_f) = f\otimes 1
\qquad\text{and}\qquad
{\rm sym} (g(\nabla(-\Delta)^{-\frac12})) = 1 \otimes g \,.
$$

We say that $T \in \Bc(L_2(\mathbb{R}^d))$ is compactly supported from the right if there is a $\phi\in C_c^\infty(\mathbb{R}^d)$ such that $T = TM_{\phi}.$ We say that $T \in \Bc(L_2(\mathbb{R}^d))$ is compactly supported if there is a $\phi\in C_c^\infty(\mathbb{R}^d)$ such that $T =M_{\phi} TM_{\phi}.$

\begin{thm}\label{FSZ generalisation thm} Let $d\geq 2$ and $p>0.$ If $T\in\Pi$ is compactly supported from the right, then
$$\lim_{t\to\infty}t^{\frac1p}\mu\Big(t,T(1-\Delta)^{-\frac{d}{2p}}\Big)=d^{-\frac1p} (2\pi)^{-\frac dp} \|{\rm sym}(T)\|_{L_p(\mathbb R^d\times\mathbb S^{d-1})}.$$
\end{thm}

This theorem with $p=d$ appears in \cite{FrSuZa}. The above more general assertion can be obtained either by following the same steps as in that paper or, as we shall show here, as a consequence of the results proved there. The corresponding result for $d=1$ is essentially the well-known Weyl asymptotic and will be discussed separately in Subsection \ref{sec:asymp1d}.


\subsection{An abstract result on spectral asymptotics}

Our goal in this subsection is to prove the following result, which allows us to reduce spectral asymptotics for the product of powers to the power of a product. The parameter $d$ in this subsection is an arbitrary real number, not necessarily an integer (although it will be in the application to the proof of Theorem \ref{FSZ generalisation thm}).

\begin{prop}\label{abstract separable lemma} 
Let $d>p>1.$ Let $0\leq A,B\in\Bc(H).$ Suppose $BA^{\frac12}\in\mathcal{L}_{d,\infty},$ $B^pA^p\in\mathcal{L}_{\frac{d}{p},\infty}$ and $[B^p,A^{\frac12}]\in(\mathcal{L}_{\frac{d}{p},\infty})_0.$ Then
$$B^pA^p-(A^{\frac12}BA^{\frac12})^p\in(\mathcal{L}_{\frac{d}{p},\infty})_0.$$
\end{prop}

The first step in the proof of this proposition is an integral representation for the difference on the left side of the proposition. This representation formula is a special case of \cite[Theorem 5.2.1]{SZ}, with a predecessor in \cite {CSZ}.

\begin{lem}\label{csz key lemma} 
Let $0\leq A,B\in\Bc(H)$ and  let $Y=A^{\frac12}BA^{\frac12}.$ For $p>1,$ define the mapping $T_p:\mathbb{R}\to \Bc(H)$ by,
\begin{align*}
T_p(s) :=\begin{cases}
B^{p-1}[BA^{\frac{1}{2}},A^{p-\frac{1}{2}}]+[BA^{\frac{1}{2}},A^{\frac{1}{2}}]Y^{p-1}
& \text{if}\ s=0 \,, \\
B^{p-1+is}[BA^{\frac{1}{2}},A^{p-\frac{1}{2}+is}]Y^{-is}+B^{is}[BA^{\frac{1}{2}},A^{\frac{1}{2}+is}]Y^{p-1-is} & \text{if}\ s \neq 0.
\end{cases}
\end{align*}
We also define the function $g_p\in\mathcal{S}(\mathbb{R})$ by setting
\begin{align*}
g_p(t) :=
\begin{cases}
1-\frac{p}{2} & \text{if}\ t=0 \,,\\
1-\frac{e^{\frac{p}{2}t}-e^{-\frac{p}{2}t}}{(e^{\frac{t}{2}}-e^{-\frac{t}{2}})(e^{\left(\frac{p-1}{2}\right)t}+e^{-\left(\frac{p-1}{2}\right)t})} & \text{if}\ t\neq 0 \,.
\end{cases}
\end{align*}
Then the mapping $T_p:\mathbb{R}\to \Bc(H)$ is continuous in the weak operator topology and we have
$$B^pA^p-(A^{\frac{1}{2}}BA^{\frac{1}{2}})^p = T_p(0)-\int_{\mathbb{R}} T_p(s)\widehat{g}_p(s)ds.$$
\end{lem}

\begin{lem}\label{tp estimate lemma} 
Let $d>0$ and $p>1.$ Let $0\leq A,B\in\mathcal B(H)$ and let $T_p$ be as in Lemma \ref{csz key lemma}. Suppose $B^pA^p\in\mathcal{L}_{\frac{d}{p},\infty}.$ We have
$$\sup_{s\in\mathbb{R}}\|T_p(s)\|_{\frac{d}{p},\infty}<\infty.$$
\end{lem}

\begin{proof} 
We use the notion of logarithmic submajorization discussed, for instance in \cite[Definition 2.3.10]{LSZ-book-2}. For $x,y\in\mathcal B(H)$ one write $x\prec\prec_{{\rm log}} y$ if $\prod_{k=1}^{n} \mu(k,x) \leq \prod_{k=1}^{n}\mu(k,y)$ for all $n\geq 0$. In this notation the Araki--Lieb--Thirring inequality (see \cite{Ar}) states that
$$X^rY^r\prec\prec_{{\rm log}}|XY|^r,\quad 0<r\leq 1.$$
Using this inequality with $X=B^p,$ $Y=A^p$ and taking into account that every quasi-Banach ideal is closed with respect to the logarithmic submajorization (see \cite[Proposition 2.4.18]{LSZ-book-2}), we obtain from the assumption $B^pA^p\in\mathcal L_{\frac dp,\infty}$ the inclusions
$$B^{p-1}A^{p-1}\in\mathcal{L}_{\frac{d}{p-1},\infty},\quad BA\in\mathcal{L}_{d,\infty},\quad B^{\frac12}A^{\frac12}\in\mathcal{L}_{2d,\infty}.$$

It follows from \eqref{singular-value-sum} that
$$
\|x+y\|_{r,\infty}\leq 2^{\frac1r}(\|x\|_{r,\infty}+\|y\|_{r,\infty})
\qquad\text{for all}\ 0<r<\infty.
$$
Applying  this quasi-triangle inequality,  we have with $Y:=A^\frac12 B A^\frac12$
$$\|T_p(s)\|_{\frac{d}{p},\infty}\leq 2^{\frac{p}{d}}\|B^{p-1}[BA^{\frac{1}{2}},A^{p-\frac{1}{2}+is}]\|_{\frac{d}{p},\infty}+2^{\frac{p}{d}}\|[BA^{\frac{1}{2}},A^{\frac{1}{2}+is}]Y^{p-1}\|_{\frac{d}{p},\infty}.$$
Again using the quasi-triangle inequality, we obtain
$$\|B^{p-1}[BA^{\frac{1}{2}},A^{p-\frac{1}{2}+is}]\|_{\frac{d}{p},\infty}\leq 2^{\frac{p}{d}}\|B^pA^p\|_{\frac{d}{p},\infty}+2^{\frac{p}{d}}\|B^{p-1}A^{p-\frac{1}{2}+is}BA^{\frac{1}{2}}\|_{\frac{d}{p},\infty},$$
$$\|[BA^{\frac{1}{2}},A^{\frac{1}{2}+is}]Y^{p-1}\|_{\frac{d}{p},\infty}\leq2^{\frac{p}{d}}\|BA^{1+is}Y^{p-1}\|_{\frac{d}{p},\infty}+2^{\frac{p}{d}}\|A^{\frac{1}{2}}BA^{\frac{1}{2}}Y^{p-1}\|_{\frac{d}{p},\infty}.$$
Using H\"older's inequality, we obtain
$$\|B^{p-1}A^{p-\frac{1}{2}+is}BA^{\frac{1}{2}}\|_{\frac{d}{p},\infty}\leq 2^{\frac{p}{d}}\|B^{p-1}A^{p-1}\|_{\frac{d}{p-1},\infty}\|A^{\frac12}BA^{\frac12}\|_{d,\infty},$$
$$\|BA^{1+is}Y^{p-1}\|_{\frac{d}{p},\infty}\leq 2^{\frac{p}{d}}\|BA\|_{d,\infty}\|Y\|_{d,\infty}^{p-1}.$$
Clearly,
$$\|A^{\frac{1}{2}}BA^{\frac{1}{2}}Y^{p-1}\|_{\frac{d}{p},\infty}=\|Y\|_{d,\infty}^p=\|B^{\frac12}A^{\frac12}\|_{2d,\infty}^{2p}.$$
Combining these estimates, we complete the proof.
\end{proof}

\begin{lem}\label{power separable part lemma} Let $1<r<\infty.$ Let $0\leq X,Y\in\Bc(H).$ If $[X,Y]\in(\mathcal{L}_{r,\infty})_0,$ then also $[X,Y^z]\in(\mathcal{L}_{r,\infty})_0$ for every $z\in\mathbb{C}$ with $\Re(z)\geq 1.$
\end{lem}
\begin{proof} Without loss of generality, $\|Y\|_{\infty}=1.$ Let $f_z$ be a Lipschitz function on $\mathbb{R}$ such that $f_z(t)=t^z$ for $t\in[0,1].$ We have
$$[X,Y^z]=[X,f_z(Y)]=T^Y_{f_z^{[1]}}([X,Y]).$$
The main result of \cite{PS-acta} yields that $T^Y_{f_z^{[1]}}:(\mathcal{L}_{r,\infty})_0\to (\mathcal{L}_{r,\infty})_0$, and this is enough to complete the proof.
\end{proof}

\begin{lem}\label{tp separable lemma} 
Let $d>p>1.$ Let $0\leq A,B\in\mathcal B(H)$ and let $T_p$ be as in Lemma \ref{csz key lemma}. Suppose $BA^{\frac12}\in\mathcal{L}_{d,\infty}$ and $[B^p,A^{\frac12}]\in(\mathcal{L}_{\frac{d}{p},\infty})_0.$ We have
$$T_p(s)\in(\mathcal{L}_{\frac{d}{p},\infty})_0,\quad s\in\mathbb{R}.$$
\end{lem}
\begin{proof} Since $[B^p,A^{\frac12}]\in(\mathcal{L}_{\frac{d}{p},\infty})_0,$ it follows from \cite[Corollary 7.1]{HSZ-holder} that
$$[B^{p-1},A^{\frac12}]\in(\mathcal{L}_{\frac{d}{p-1},\infty})_0,\quad [B,A^{\frac12}]\in(\mathcal{L}_{d,\infty})_0.$$

By the Leibniz rule, we have
$$B^{p-1}[BA^{\frac{1}{2}},A^{p-\frac{1}{2}+is}]=X_1X_2-X_3X_4,$$
$$X_1=[B^p,A^{p-\frac12+is}],\quad X_2= A^{\frac12},\quad X_3=[B^{p-1},A^{p-\frac12+is}],\quad X_4=BA^{\frac12}.$$
Applying Lemma \ref{power separable part lemma} to $X=B^p$ (respectively, $X=B^{p-1}$), $Y=A^{\frac12},$ $r=\frac{d}{p}$ (respectively, $r=\frac{d}{p-1}$) and $z=2p-1+2is,$ we obtain 
$$X_1\in(\mathcal{L}_{\frac{d}{p},\infty})_0 \,, \quad X_3\in(\mathcal{L}_{\frac{d}{p-1},\infty})_0 \,.$$
Since also $X_2$ is bounded and $X_4\in\mathcal{L}_{d,\infty},$ it follows from the inclusions above and from H\"older's inequality that
$$B^{p-1}[BA^{\frac{1}{2}},A^{p-\frac{1}{2}+is}]\in(\mathcal{L}_{\frac{d}{p},\infty})_0.$$
Thus,
\begin{equation}\label{tpsl eq0}
B^{p-1+is}[BA^{\frac{1}{2}},A^{p-\frac{1}{2}+is}]Y^{-is}\in(\mathcal{L}_{\frac{d}{p},\infty})_0.
\end{equation}
	
We have
$$B^{is}[BA^{\frac{1}{2}},A^{\frac{1}{2}+is}]Y^{p-1-is}=B^{is}\cdot [B,A^{\frac{1}{2}+is}]\cdot A^{\frac12}Y^{-is}\cdot Y^{p-1}.$$
Recall that $[B,A^{\frac12}]\in(\mathcal{L}_{d,\infty})_0.$ Applying Lemma \ref{power separable part lemma} with $X=B,$ $Y=A^{\frac12},$ $r=d$ and $z=1+2is,$ we obtain
$$[B,A^{\frac{1}{2}+is}]\in(\mathcal{L}_{d,\infty})_0.$$
By assumption, $BA^{\frac12}\in\mathcal{L}_{d,\infty}.$ Thus, $Y\in\mathcal{L}_{d,\infty}$ and $Y^{p-1}\in\mathcal{L}_{\frac{d}{p-1},\infty}.$ Hence, 
\begin{equation}\label{tpsl eq1}
B^{is}[BA^{\frac{1}{2}},A^{\frac{1}{2}+is}]Y^{p-1-is}\in(\mathcal{L}_{\frac{d}{p},\infty})_0.
\end{equation}
The assertion follows now by combining \eqref{tpsl eq0} and \eqref{tpsl eq1}.
\end{proof}

We are finally in position to prove the main result of this subsection.

\begin{proof}[Proof of Proposition \ref{abstract separable lemma}] 
By Lemmas \ref{tp estimate lemma} and \ref{tp separable lemma}, the mapping
$$s\to T_p(s)\widehat{g_p}(s),\quad s\in\mathbb{R},$$
is absolutely integrable in $(\mathcal{L}_{\frac{d}{p},\infty})_0.$ Since $d>p,$ it follows that the latter mapping is Bochner integrable in $(\mathcal{L}_{\frac{d}{p},\infty})_0.$ In particular, its integral belongs to $(\mathcal{L}_{\frac{d}{p},\infty})_0.$
\end{proof}


\subsection{Proof of Theorem \ref{FSZ generalisation thm}}

We now combine the result from \cite{FrSuZa} with Proposition \ref{abstract separable lemma} and obtain the following preliminary version of Theorem \ref{FSZ generalisation thm}.

\begin{lem}\label{final concrete separable lemma} Let $1<p<d.$ If $0\leq T\in\Pi$ is compactly supported, then one has
$$\lim_{t\to\infty}t^{\frac{p}{d}}\mu(t,T^p(1-\Delta)^{-\frac{p}{2}})=\Big(d^{-\frac1d}(2\pi)^{-1}\|{\rm sym}(T)\|_{L_d(\mathbb{R}^d\times\mathbb{S}^{d-1})}\Big)^p.$$
\end{lem}

\begin{proof} Set $A=T$ and $B=(1-\Delta)^{-\frac12}.$ We want to apply Proposition \ref{abstract separable lemma}.

Let us show that the conditions of the Proposition \ref{abstract separable lemma} are satisfied. Indeed, the inclusion $[B^p,A^{\frac12}]\in(\mathcal{L}_{\frac{d}{p},\infty})_0$ is established in Theorem \ref{t delta commutator lemma} for every $p>0.$ Recall that $T$ is compactly supported and choose $\phi\in C_c(\mathbb{R}^d)$ such that $T=M_{\phi}T.$ Then, by Theorem \ref{cwikel estimate},
$$BA^{\frac12}=(1-\Delta)^{-\frac12}M_{\phi}\cdot T^{\frac12}\in\mathcal{L}_{d,\infty},\quad B^pA^p=(1-\Delta)^{-\frac{p}{2}}M_{\phi}\cdot T^{\frac{p}{2}}\in\mathcal{L}_{\frac{d}{p},\infty}.$$
By Proposition \ref{abstract separable lemma}, we have
$$B^pA^p-\big(A^{\frac12}BA^{\frac12}\big)^p
\in(\mathcal{L}_{\frac{d}{p},\infty})_0.$$

By \cite[Theorem 1.5]{FrSuZa}, we have
$$\lim_{t\to\infty}t^{\frac1d}\mu(t,AB)=c_T,\quad c_T:=d^{-\frac1d}(2\pi)^{-1}\|{\rm sym}(T)\|_{L_d(\mathbb{R}^d\times\mathbb{S}^{d-1})}.$$
By Theorem \ref{t delta commutator lemma}, we have $[B,A^{\frac12}]\in(\mathcal{L}_{d,\infty})_0.$ By a standard result on spectral asymptotics (see, e.g., \cite[Lemma 3.1]{FrSuZa}) this implies that
$$\lim_{t\to\infty}t^{\frac1d}\mu(t,A^{\frac12}BA^{\frac12})=c_T.$$
In other words, we have
$$\lim_{t\to\infty}t^{\frac{p}{d}}\mu\Big(t,\big(A^{\frac12}BA^{\frac12}\big)^p\Big)=c_T^p.$$
Again using \cite[Lemma 3.1]{FrSuZa}, we conclude that
$$\lim_{t\to\infty}t^{\frac{p}{d}}\mu(t,B^pA^p)=c_T^p.$$
This is exactly our assertion.
\end{proof}

\begin{lem}\label{post-FSZ induction lemma} 
Let $d\geq 2$ and $p>0.$ The following conditions are equivalent:
\begin{enumerate}[{\rm (i)}]
\item\label{pfila} for every nonnegative and compactly supported $T\in\Pi$ one has
$$\lim_{t\to\infty}t^{\frac1p}\mu(t,T(1-\Delta)^{-\frac{d}{2p}})=d^{-\frac1p}(2\pi)^{-\frac{d}{p}}\|{\rm sym}(T)\|_{L_p(\mathbb{R}^d\times\mathbb{S}^{d-1})}$$
\item\label{pfilb} for every nonnegative and compactly supported $T\in\Pi$ one has
$$\lim_{t\to\infty}t^{\frac2p}\mu(t,T(1-\Delta)^{-\frac{d}{p}})=d^{-\frac2p}(2\pi)^{-\frac{2d}{p}}\|{\rm sym}(T)\|_{L_{\frac{p}{2}}(\mathbb{R}^d\times\mathbb{S}^{d-1})}.$$
\end{enumerate}
\end{lem}

\begin{proof} We use the fact that the set of nonnegative, compactly supported elements of $\Pi$ is invariant under taking powers.

Renaming $T$ into $T^2,$ we rewrite \eqref{pfilb} as follows: for every nonnegative and compactly supported $T\in\Pi$ one has
$$
\lim_{t\to\infty}t^{\frac2p}\mu(t,T^2(1-\Delta)^{-\frac{d}{p}})=d^{-\frac2p}(2\pi)^{-\frac{2d}{p}}\|{\rm sym}(T)\|_{L_p(\mathbb{R}^d\times\mathbb{S}^{d-1})}^2.
$$
It follows from Theorem \ref{t delta commutator lemma} (which we proved for $d\geq 2$) and simple arithmetic that
$$T^2(1-\Delta)^{-\frac{d}{p}}-(1-\Delta)^{-\frac{d}{2p}}T^2(1-\Delta)^{-\frac{d}{2p}}\in(\mathcal{L}_{\frac{p}{2},\infty})_0.$$
Using a standard result about spectral asymptotics (see, e.g., \cite[Lemma 3.1]{FrSuZa}), we rewrite \eqref{pfilb} as follows:	
$$\lim_{t\to\infty}t^{\frac2p}\mu^2(t,T(1-\Delta)^{-\frac{d}{2p}})=d^{-\frac2p}(2\pi)^{-\frac{2d}{p}}\|{\rm sym}(T)\|_{L_p(\mathbb{R}^d\times\mathbb{S}^{d-1})}^2.$$
This is clearly equivalent to \eqref{pfila}.
\end{proof}

\begin{proof}[Proof of Theorem \ref{FSZ generalisation thm}] Let us first prove the assertion for $1<p\leq d.$ For $p=d,$ the assertion is given by Theorem 1.5 in \cite{FrSuZa}. Hence, we may assume that $1<p<d.$ Applying Lemma \ref{final concrete separable lemma} to $T^{\frac1p}$ (which belongs to $\Pi$ and is compactly supported), we conclude that
\begin{align*}
\lim_{t\to\infty}t^{\frac{p}{d}}\mu(t,T(1-\Delta)^{-\frac{p}{2}})
& =d^{-\frac{p}{d}}(2\pi)^{-p}\|{\rm sym}(T^\frac 1p)\|_{L_{d}(\mathbb{R}^d\times\mathbb{S}^{d-1})}^p \\
& = d^{-\frac{p}{d}}(2\pi)^{-p}\|{\rm sym}(T)\|_{L_{\frac{d}{p}}(\mathbb{R}^d\times\mathbb{S}^{d-1})} 
\end{align*}
Renaming $p$ into $\frac{d}{p}$ and noting that as $p$ runs through $(1,d)$, $\frac{d}{p}$ runs through the same interval, we conclude the assertion for $1<p<d$ (and, hence, for $1<p\leq d$).

Let us now prove the assertion in full generality. Fix $p>0.$ Choose $n\in\mathbb{Z}$ such that $q=2^np\in (1,d]$. By Lemma \ref{final concrete separable lemma}, for every nonnegative and compactly supported $T\in\Pi$ we have
$$\lim_{t\to\infty}t^{\frac1q}\mu(t,T(1-\Delta)^{-\frac{d}{2q}})=d^{-\frac1q}(2\pi)^{-\frac{d}{q}}\|{\rm sym}(T)\|_{L_q(\mathbb{R}^d\times\mathbb{S}^{d-1})}.$$
By repeated use of Lemma \ref{post-FSZ induction lemma}, we deduce that for every nonnegative and compactly supported $T\in\Pi$ we have
$$
\lim_{t\to\infty}t^{\frac1p}\mu(t,T(1-\Delta)^{-\frac{d}{2p}})=d^{-\frac1p}(2\pi)^{-\frac{d}{p}}\|{\rm sym}(T)\|_{L_p(\mathbb{R}^d\times\mathbb{S}^{d-1})}.
$$
This proves the assertion for every nonnegative and compactly supported $T\in\Pi.$
	
Consider now the general case. Let $T\in\Pi$ be compactly supported from the right. Clearly, $|T|\in\Pi$ is nonnegative and compactly supported. By the preceding paragraph, we have
$$\lim_{t\to\infty}t^{\frac1p}\mu(t,|T|(1-\Delta)^{-\frac{d}{2p}})=d^{-\frac1p}(2\pi)^{-\frac{d}{p}}\|{\rm sym}(|T|)\|_{L_p(\mathbb{R}^d\times\mathbb{S}^{d-1})}.$$
Observing that for any operators $A, B\in \Bc(H)$ we have  $\mu(AB)=\mu(|AB|)$ and $|AB|=\big||A|B\big|, $ we write
$$\mu\Big(T(1-\Delta)^{-\frac{d}{2p}}\Big)=\mu\Big(|T|(1-\Delta)^{-\frac{d}{2p}}\Big),\quad {\rm sym}(|T|)=|{\rm sym}(T)|.$$
This proves the assertion in general case.
\end{proof}


\subsection{Asymptotics for $d=1$}\label{sec:asymp1d}

For $d=1,$ the sphere $\mathbb{S}^{d-1}$ is just a two-point set. Hence, the algebra $\Pi$ consists of two `copies' of the first algebra in the definition of $\Pi$. So, the result analogous to Theorem \ref{FSZ generalisation thm} should be stated only for multiplication operators. This is the well-known Weyl asymptotic formula. For the convenience of the reader we explain how this can be deduced from the results of Birman--Solomyak in~\cite{BiSo70}.

\begin{thm}\label{weyl asymptotic} Let $d=1$ and $p>0.$ If $f\in C_c(\mathbb{R}),$ then
$$\lim_{t\to\infty}t^{\frac1p}\mu(t,M_f(1-\Delta)^{-\frac{1}{2p}})=\pi^{-\frac 1p} \|f\|_p.$$
\end{thm}

The ingredient from~\cite{BiSo70} that we use is the following.

\begin{lem}\label{from bs70 lemma} Let $d=1$ and let $p>1.$ If $f\in C_c(\mathbb{R}),$ then
$$\lim_{t\to\infty}t^{\frac1p}\mu(t,M_{\bar{f}}(-\Delta)^{-\frac1{2p}}M_f)= \pi^{-\frac1p} \|f\|_{2p}^2.$$
\end{lem}

\begin{proof} 
This is a very special case of \cite[Theorem 1]{BiSo70}. Indeed, take $m=1$ and $\alpha=\frac1p-1\in(-1,0).$ Set $\theta(t)=|t|^{\alpha},$ $\Phi(t,t')=1$ and $a(t)=f(t)$ for every $t\in\mathbb{R}.$ Set $E={\rm supp}(f).$ 

Note that in \cite{BiSo70} a non-standard notion of Fourier transform is used. However, for a homogeneous function $\theta$ of degree $\alpha\in(-1,0),$ this notion coincides with the usual notion of Fourier transform. In particular, the operator $T$ featuring in \cite[(2.4)]{BiSo70} is given by a constant times $M_{\bar{f}}(-\Delta)^{-\frac1{2p}}M_f$.
\end{proof}

\begin{lem}\label{d=1 induction lemma} 
Let $d=1$ and $p>0.$ The following conditions are equivalent:
\begin{enumerate}[{\rm (i)}]
\item\label{dila} for every $0\leq f\in C_c(\mathbb{R})$ one has
$$\lim_{t\to\infty}t^{\frac1p}\mu(t,M_f(1-\Delta)^{-\frac{1}{2p}})=\pi^{-\frac1p}\|f\|_p$$
\item\label{dilb} for every  $0\leq f\in C_c(\mathbb{R})$ one has
$$\lim_{t\to\infty}t^{\frac2p}\mu(t,M_f(1-\Delta)^{-\frac{1}{p}})=\pi^{-\frac2p}\|f\|_{\frac{p}{2}}.$$
\end{enumerate}
\end{lem}
\begin{proof} The proof follows that of Lemma \ref{post-FSZ induction lemma} {\it mutatis mutandi}. Instead of Theorem \ref{t delta commutator lemma}, which we stated under the assumption $d\geq 2$, we apply Lemma \ref{second appendix commutator lemma}, whose proof remains valid for $d=1$.
\end{proof}

\begin{proof}[Proof of Theorem \ref{weyl asymptotic}] First, take $p>2.$ Applying Lemma \ref{from bs70 lemma} to $\frac{p}{2}$ and taking into account that
$$\mu(M_{\bar{f}}(-\Delta)^{-\frac1p}M_f)=\mu^2(M_f(-\Delta)^{-\frac1{2p}}),$$ we arrive at 
$$\lim_{t\to\infty}t^{\frac2p}\mu^2(t,M_f(-\Delta)^{-\frac1{2p}})=\pi^{-\frac 2p} \|f\|_p^2.$$
Noting that, by Theorem \ref{pre-cwikel estimate},
$$M_f(-\Delta)^{-\frac1{2p}}-M_f(1-\Delta)^{-\frac1{2p}}\in(\mathcal{L}_{p,\infty})_0,$$
and using a standard result about spectral asymptotics (see, e.g., \cite[Lemma 3.1]{FrSuZa}), we infer the assertion for $p>2.$

The assertion for $p\leq 2$ follows by induction in the same way as in the proof of Theorem \ref{FSZ generalisation thm}, using Lemma \ref{d=1 induction lemma} instead of Lemma \ref{post-FSZ induction lemma}.
\end{proof}


\section{Proof of Theorem \ref{main theorem positive} \eqref{mtpb} and Theorem \ref{main theorem negative} \eqref{mtnb}}\label{sec:asymptotics}

Parts (ii) of Theorems \ref{main theorem positive} and \ref{main theorem negative} state spectral asymptotics for $[(-\Delta)^{\frac{\epsilon}{2}},M_f]$ under the sole assumption that the leading term is finite. We begin by proving these asymptotics in the smooth case.

\begin{lem}\label{mtb key corollary} 
Let $d\geq 2$ and $\epsilon\in(-\frac{d}{2},1)$ (alternatively, let $d=1$ and $\epsilon\in(0,1)$). If $f\in C^{\infty}_c(\mathbb{R}^d),$ then
$$\lim_{t\to\infty}t^{\frac{1-\epsilon}{d}}\mu\Big(t,[(-\Delta)^{\frac{\epsilon}{2}},M_f]\Big) = \kappa_{d,\epsilon} \|f\|_{\dot{W}^{1}_{\frac{d}{1-\epsilon}}(\mathbb{R}^d)}$$
with $\kappa_{d,\epsilon}$ given by \eqref{eq:constasymp}.
\end{lem}
\begin{proof} Set
$$T:=\epsilon \sum_{k=1}^d \frac{D_k}{(-\Delta)^{\frac12}}M_{D_kf}$$
and note that $T\in\Pi$ and that $T$ is compactly supported from the right. 

We claim that
\begin{equation}\label{eq:mtbkeycor}
[(-\Delta)^{\frac{\epsilon}{2}},M_f]-T(1-\Delta)^{-\frac{1-\epsilon}{2}}\in(\mathcal{L}_{\frac{d}{1-\epsilon},\infty})_0.
\end{equation}
To prove this, we note that by Theorem \ref{pre-cwikel estimate}
$$[(1-\Delta)^{\frac{\epsilon}{2}},M_f]-[(-\Delta)^{\frac{\epsilon}{2}},M_f]=$$
$$=\big((1-\Delta)^{\frac{\epsilon}{2}}-(-\Delta)^{\frac{\epsilon}{2}}\big)M_f-M_f\big((1-\Delta)^{\frac{\epsilon}{2}}-(-\Delta)^{\frac{\epsilon}{2}}\big)\in\mathcal{L}_{\frac{d}{1-\epsilon}}\subset (\mathcal{L}_{\frac{d}{1-\epsilon},\infty})_0.$$
Here we use the assumption $f\in C^\infty_c(\mathbb R^d)$ together with the fact that the function $\xi\mapsto (1+|\xi|^2)^\frac\epsilon2 - |\xi|^\epsilon$ belongs to $L_\frac{d}{1-\epsilon}(\mathbb R^d)$ if $\frac{d}{1-\epsilon}\geq 2$ and to $\ell_\frac{d}{1-\epsilon}(L_2)(\mathbb R^d)$ if $\frac{d}{1-\epsilon}<2$.
Next, by Theorem \ref{approximation theorem},
$$
[(1-\Delta)^{\frac{\epsilon}{2}},M_f] + \tfrac\epsilon2 [\Delta,M_f] (1-\Delta)^{\frac\epsilon2 -1} \in (\mathcal{L}_{\frac{d}{1-\epsilon},\infty})_0.
$$
Finally, we write
$$
-[\Delta,M_f] = 2 \sum_{k=1}^d M_{D_k f} D_k - M_{\Delta f} \,.
$$
For the second term we use Theorem \ref{pre-cwikel estimate} similarly as before and find
$$
M_{\Delta f} (1-\Delta)^{\frac\epsilon2 -1} \in\mathcal{L}_{\frac{d}{1-\epsilon}}\subset(\mathcal{L}_{\frac{d}{1-\epsilon},\infty})_0.
$$
For the first term we again use Theorem \ref{pre-cwikel estimate} and find
$$
\epsilon \sum_{k=1}^d M_{D_k f} D_k (1-\Delta)^{\frac\epsilon2 -1} - T (1-\Delta)^{-\frac{1-\epsilon}{2}} \in\mathcal{L}_{\frac{d}{1-\epsilon}}\subset(\mathcal{L}_{\frac{d}{1-\epsilon},\infty})_0.
$$
This completes the proof of \eqref{eq:mtbkeycor}.

Consider the case $d\geq 2.$ The inclusion \eqref{eq:mtbkeycor}, together with simple limiting arguments (see, e.g., \cite[Lemma 3.1]{FrSuZa}) and Theorem \ref{FSZ generalisation thm}, applied with $p=\frac{d}{1-\epsilon}$, implies
\begin{align*}
\lim_{t\to\infty}t^{\frac{1-\epsilon}{d}}\mu\Big(t,[(-\Delta)^{\frac{\epsilon}{2}},M_f]\Big) 
& = \lim_{t\to\infty}t^{\frac{1-\epsilon}{d}}\mu\Big(t, T(1-\Delta)^{-\frac{1-\epsilon}{2}} \Big) \\
& = d^{-\frac{1-\epsilon}d} (2\pi)^{-1+\epsilon} \| {\rm sym}(T) \|_{L_{\frac{d}{1-\epsilon}(\mathbb R^d\times\mathbb S^{d-1})}}.
\end{align*}
Finally, we compute, similarly as in \cite[Lemma 8.4]{FrSuZa},
\begin{align*}
\| {\rm sym}(T) \|_{L_{\frac{d}{1-\epsilon}(\mathbb R^d\times\mathbb S^{d-1})}}^\frac d{1-\epsilon} & = |\epsilon|^\frac{d}{1-\epsilon} \int_{\mathbb R^d} \int_{\mathbb S^{d-1}} |\omega\cdot\nabla f(x)|^\frac{d}{1-\epsilon}\,d\omega\,dx \\
& = |\epsilon|^\frac{d}{1-\epsilon} \int_{\mathbb R^d} |\nabla f(x)|^\frac{d}{1-\epsilon}\,dx \int_{\mathbb S^{d-1}} |\omega_d|^\frac{d}{1-\epsilon}\,d\omega \,.
\end{align*}
This completes the proof for $d\geq 2.$

Consider the case $d=1.$ The inclusion \eqref{eq:mtbkeycor} reads as
$$[(-\Delta)^{\frac{\epsilon}{2}},M_f]\in -\epsilon M_{Df}(1-\Delta)^{\frac{\epsilon-1}{2}}\cdot{\rm sgn}(D)+(\mathcal{L}_{\frac1{1-\epsilon},\infty})_0.$$
By a simple limiting argument (see, e.g., \cite[Lemma 3.1]{FrSuZa}), this reduces the assertion for $d=1$ to the corresponding assertion about the operator $\epsilon M_{Df}(1-\Delta)^{\frac{\epsilon-1}{2}}\cdot{\rm sgn}(D)$. Since ${\rm sgn}(D)$ is unitary, we have
$$
\mu(M_{Df}(1-\Delta)^{\frac{\epsilon-1}{2}}\cdot{\rm sgn}(D)) = \mu(M_{Df}(1-\Delta)^{\frac{\epsilon-1}{2}}) \,.
$$
The assertion for $d=1$ follows now from Theorem \ref{weyl asymptotic} (applied with $p=\frac1{1-\epsilon}$).
\end{proof}

\begin{proof}[Proof of Theorem \ref{main theorem positive} \eqref{mtpb} and Theorem \ref{main theorem negative} \eqref{mtnb}] The assertion follows from Lemma \ref{mtb key corollary} and the universal bounds in parts (i) by means a simple approximation argument (see, e.g., \cite[Lemma 3.2]{FrSuZa}).
\end{proof}

In the proof of part \eqref{mtnc} we need a lower bound with localization functions.

\begin{lem}\label{mtb key corollary2} 
Let $d\geq2$ and $\epsilon\in(-\frac{d}{2},1)$ (alternatively, let $d=1$ and $\epsilon\in(0,1)$). If $f\in C^{\infty}(\mathbb{R}^d)$ and $\chi\in C_c(\mathbb{R}^d),$ then
$$\lim_{t\to\infty}t^{\frac{1-\epsilon}{d}}\mu\Big(t,M_{\bar\chi} [(-\Delta)^{\frac{\epsilon}{2}},M_f] M_\chi \Big)
= \kappa_{d,\epsilon} \left( \int_{\mathbb R^d} |\chi|^\frac{2d}{1-\epsilon} |\nabla f|^\frac{d}{1-\epsilon}\right)^\frac{1-\epsilon}{d}$$
with $\kappa_{d,\epsilon}$ given by \eqref{eq:constasymp}.
\end{lem}

\begin{proof} Let $\phi\in C^{\infty}_c(\mathbb{R}^d)$ be real-valued with $\chi\phi=\chi.$ Then
$$M_{\bar\chi} [(-\Delta)^{\frac{\epsilon}{2}},M_f] M_\chi=M_{\bar\chi} [(-\Delta)^{\frac{\epsilon}{2}},M_{f\cdot\phi}] M_\chi.$$
Hence, we may assume without loss of generality that $f\in C^{\infty}_c(\mathbb{R}^d).$ The rest of the argument follows that in Corollary \ref{mtb key corollary} {\it mutatis mutandi}. 
\end{proof}

The following bounds is an immediate consequence of Lemma \ref{mtb key corollary2}.

\begin{lem}\label{mtb key corollary3}
Let $d\geq 2$ and let $\epsilon\in(-\frac{d}{2},1)$ (alternatively, let $d=1$ and $\epsilon\in(0,1)$). There is a constant $c_{d,\epsilon}>0$ such that for all $f\in C^{\infty}(\mathbb{R}^d)$ and $\chi\in C_c(\mathbb{R}^d),$
$$\|M_{\bar\chi} [(-\Delta)^{\frac{\epsilon}{2}},M_f] M_\chi \|_{\frac{d}{1-\epsilon},\infty}
\geq c_{d,\epsilon} \left( \int_{\mathbb R^d} |\chi|^\frac{2d}{1-\epsilon} |\nabla f|^\frac{d}{1-\epsilon}\right)^\frac{1-\epsilon}{d}.$$
\end{lem}

\section{Proof of Theorems \ref{main theorem positive} and \ref{main theorem negative} \eqref{mtnc}}\label{sec:proofpartiii}

In this section, $T_t,$ $t\in\mathbb{R}^d,$ denotes the translation operator,
$$
(T_t f)(s) = f(s+t) \,,\ s\in\mathbb R^d \,.
$$

\begin{lem}\label{convolution fubini lemma} Let $\epsilon\in(-\frac{d}{2},0)$ and let $f\in L_{1,{\rm loc}}(\mathbb{R}^d).$ If $h,\Phi\in C_c(\mathbb{R}^d),$ then
$$\int_{\mathbb{R}^d}\Phi(t)\langle T_{-t}f,h\rangle dt=\langle f\ast\Phi,h\rangle.$$
\end{lem}

\begin{proof} Let $h$ be supported in $B(0,r_1)$ and let $\Phi$ be supported in $B(r_2).$ We have
$$\int_{\mathbb{R}^d\times\mathbb{R}^d}|h(s)||\Phi(t)||f(s-t)|\,dsdt\leq\|h\|_{\infty}\|\Phi\|_{\infty}\int_{\substack{|s|\leq r_1\\ |t|\leq r_2}}|f(s-t)|\,dsdt=$$
$$=\|h\|_{\infty}\|\Phi\|_{\infty}\int_{\substack{|u|\leq r_1\\ |u+v|\leq r_2}}|f(v)|\,dudv\leq \|h\|_{\infty}\|\Phi\|_{\infty}\int_{\substack{|u|\leq r_1\\ |v|\leq r_1+r_2}}|f(v)|\,dudv=$$
$$=c_dr_1^2\|h\|_{\infty}\|\Phi\|_{\infty}\int_{|v|\leq r_1+r_2}|f(v)|\,dv<\infty.$$
By Fubini theorem, we have
$$\int_{\mathbb{R}^d}\Phi(t)\Big(\int_{\mathbb{R}^d}f(s-t)\overline{h(s)}\,ds\Big)dt=\int_{\mathbb{R}^d}\overline{h(s)} \Big(\int_{\mathbb{R}^d}f(s-t)\Phi(t)dt\Big)ds.$$
This completes the proof.
\end{proof}	

In the following lemma we use the notion of submajorization, which is discussed in detail for instance in \cite[Section 2.3]{LSZ-book-2}.

\begin{lem}\label{main convolution lemma} Let $\epsilon\in(-\frac{d}{2},1]$ and let $f\in L_{1,{\rm loc}}(\mathbb{R}^d)$ be such that $[(-\Delta)^{\frac{\epsilon}{2}},M_f]$ extends to a bounded operator on $L_2(\mathbb{R}^d).$ If $\Phi\in C_c(\mathbb{R}^d)$ is nonnegative with $\|\Phi\|_1=1,$ then $[(-\Delta)^{\frac{\epsilon}{2}},M_{f\ast\Phi}]$ also extends to a bounded operator on $L_2(\mathbb{R}^d)$ and
$$[(-\Delta)^{\frac{\epsilon}{2}},M_{f\ast\Phi}]\prec\prec [(-\Delta)^{\frac{\epsilon}{2}},M_f].$$
Moreover, if $[(-\Delta)^{\frac{\epsilon}{2}},M_f]$ is compact, then so is $[(-\Delta)^{\frac{\epsilon}{2}},M_{f\ast\Phi}]$.
\end{lem}
\begin{proof} Let $A:L_2(\mathbb{R}^d)\to L_2(\mathbb{R}^d)$ be the extension of $[(-\Delta)^{\frac{\epsilon}{2}},M_f].$ We consider the weak integral
$$B=\int_{\mathbb{R}^d}\Phi(t)T_{-t}AT_tdt.$$	
Clearly, $B$ is a bounded operator and we have $B\prec\prec A$; see, for instance, \cite[Lemma 18 and its proof]{LMSZ}.

We now claim that
\begin{equation}
	\label{eq:majorproof}
	\langle B\phi,\psi\rangle=\langle [(-\Delta)^{\frac{\epsilon}{2}},M_{f\ast\Phi}]\phi,\psi\rangle,\quad \phi,\psi\in C^{\infty}_c(\mathbb{R}^d).
\end{equation}
By definition,
$$LHS=\int_{\mathbb{R}^d}\Phi(t)\langle T_{-t}AT_t\phi,\psi\rangle dt=\int_{\mathbb{R}^d}\Phi(t)\langle [(-\Delta)^{\frac{\epsilon}{2}},M_{T_{-t}f}]\phi,\psi\rangle dt.$$
By Lemma \ref{definition}, we have
$$\langle [(-\Delta)^{\frac{\epsilon}{2}},M_{T_{-t}f}]\phi,\psi\rangle
=\langle T_{-t}f\cdot\phi,(-\Delta)^{\frac{\epsilon}{2}}\psi\rangle
-\langle (-\Delta)^{\frac{\epsilon}{2}}\phi,\overline{T_{-t}f}\cdot\psi\rangle=$$
$$=\langle T_{-t}f,\bar{\phi}\cdot(-\Delta)^{\frac{\epsilon}{2}}\psi\rangle-\langle \bar{\psi}\cdot(-\Delta)^{\frac{\epsilon}{2}}\phi,\overline{T_{-t}f}\rangle
=\langle T_{-t} f,h\rangle,$$
where
$$h=\bar{\phi}\cdot (-\Delta)^{\frac{\epsilon}{2}}\psi-\psi\cdot\overline{(-\Delta)^{\frac{\epsilon}{2}}\phi}\in C_c(\mathbb{R}^d).$$
Meanwhile, it follows from Lemma \ref{definition} that
$$RHS=\langle(f\ast\Phi)\cdot\phi,(-\Delta)^{\frac{\epsilon}{2}}\psi\rangle-\langle (-\Delta)^{\frac{\epsilon}{2}}\phi,\overline{(f\ast\Phi)}\cdot\psi\rangle=\langle f\ast\Phi,h\rangle.$$
The claim \eqref{eq:majorproof} follows now from Lemma \ref{convolution fubini lemma}.	

By \eqref{eq:majorproof}, we have
$$B\phi=[(-\Delta)^{\frac{\epsilon}{2}},M_{f\ast\Phi}]\phi,\quad \phi\in C^{\infty}_c(\mathbb{R}^d).$$
Therefore, $[(-\Delta)^{\frac{\epsilon}{2}},M_{f\ast\Phi}]$ extends to a bounded operator on $L_2(\mathbb{R}^d)$ and its extension is exactly $B.$

The last statement concerning compactness follows from the general fact that any bounded operator that is submajorized by a compact one is compact. The latter fact follows immediately from the definition of submajorization via the singular value function, together with the simple fact that a bounded operator $A$ is compact if and only if $N^{-1} \sum_{n=1}^N \mu(n,A)\to 0$ as $N\to\infty.$
\end{proof}

\begin{lem}\label{standard convolution lemma} Let $1<p\leq \infty$ and let $f\in L_{1,{\rm loc}}(\mathbb{R}^d)$ be such that
$$\|\nabla(f\ast\Phi)\|_{L_p(\mathbb{R}^d,\mathbb{C}^d)}\leq 1$$
for every nonnegative $\Phi\in C^{\infty}_c(\mathbb{R}^d)$ with $\|\Phi\|_1=1.$ Then $f\in\dot{W}^{1,p}(\mathbb{R}^d)$ and
$$\|f\|_{\dot{W}^{1}_{p}(\mathbb{R}^d)}\leq 1.$$
\end{lem}
\begin{proof} Let $0\leq\Phi\in C^{\infty}_c(\mathbb{R}^d)$ with $\|\Phi\|_1=1$ and set
$$\Phi_n(t):=n^d\Phi(nt),\quad t\in\mathbb{R}^d,\quad n\in\mathbb{N}.$$
By assumption, we have
$$\|\nabla(f\ast\Phi_n)\|_{L_p(\mathbb{R}^d,\mathbb{C}^d)}\leq 1.$$

Since $1<p\leq\infty$, $L_p(\mathbb R^d,\mathbb C^d)$ is the Banach dual of a separable space, namely of $L_{p'}(\mathbb R^d,\mathbb C^d)$. Therefore (see, e.g., \cite[Proposition 4.49]{vN}) there is a subsequence $(\nabla(f\ast\Phi_{n_k}))$ and a $G\in L_p(\mathbb R^d,\mathbb C^d)$ such that $\nabla(f\ast\Phi_{n_k})\to G$ in weak$^{\ast}$-topology on $L_p$. In particular, $\nabla(f\ast\Phi_{n_k})\to G$ in the sense of distributions. However, $\nabla(f\ast\Phi_{n_k})\to \nabla f$ in the sense of distributions. By uniqueness of the limit, we have $G=\nabla f$ and, therefore, $\nabla f\in L_p(\mathbb{R}^d,\mathbb{C}^d)$ and
$$\|\nabla f\|_{L_p(\mathbb{R}^d,\mathbb{C}^d)}\leq 1,$$
as claimed.
\end{proof}

\begin{proof}[Proof of Theorems \ref{main theorem positive} and \ref{main theorem negative} \eqref{mtnc}]
	Let $f\in L_{1,{\rm loc}}(\mathbb R^d)$ and assume $[(-\Delta)^\frac\epsilon2,M_f]$ extends to a bounded operator on $L_2(\mathbb R^d)$ that belongs to $\mathcal L_{\frac{d}{1-\epsilon},\infty}$.
	
Let $\Phi\in C^{\infty}_c(\mathbb{R}^d)$ be nonnegative with $\|\Phi\|_1=1.$ By Lemma \ref{main convolution lemma}, we have
$$
[(-\Delta)^{\frac{\epsilon}{2}},M_{f\ast\Phi}]\prec\prec [(-\Delta)^{\frac{\epsilon}{2}},M_f].
$$
Therefore, if $\chi\in C^{\infty}_c(\mathbb{R}^d)$ satisfies $\|\chi\|_{\infty}=1$, we have
$$\Big\|M_{\overline\chi}[(-\Delta)^{\frac{\epsilon}{2}},M_{f\ast\Phi}]M_{\chi}\Big\|_{\frac{d}{1-\epsilon},\infty}\leq c_{d,\epsilon}^{(1)}\Big\| [(-\Delta)^{\frac{\epsilon}{2}},M_f]\Big\|_{\frac{d}{1-\epsilon},\infty}.$$
Meanwhile, since $\chi\in C^{\infty}_c(\mathbb{R}^d)$ and $f\ast\Phi\in C^{\infty}(\mathbb{R}^d)$, Lemma \ref{mtb key corollary3} implies that
$$\||\chi|^2\cdot|\nabla(f\ast\Phi)|\|_{L_{\frac{d}{1-\epsilon}}(\mathbb{R}^d)}\leq c^{(2)}_{d,\epsilon}\Big\|M_{\bar{\chi}}[(-\Delta)^{\frac{\epsilon}{2}},M_{f\ast\Phi}]M_{\chi}\Big\|_{\frac{d}{1-\epsilon},\infty}.$$
Combining these bounds, we find
$$\||\chi|^2\cdot|\nabla(f\ast\Phi)|\|_{L_{\frac{d}{1-\epsilon}}(\mathbb{R}^d)}\leq c^{(1)}_{d,\epsilon}c^{(2)}_{d,\epsilon}\Big\|[(-\Delta)^{\frac{\epsilon}{2}},M_f]\Big\|_{\frac{d}{1-\epsilon},\infty}.$$
Taking the supremum over $\chi\in C^{\infty}_c(\mathbb{R}^d)$ such that $\|\chi\|_{\infty}=1,$ we obtain
$$\|\nabla(f\ast\Phi)\|_{L_{\frac{d}{1-\epsilon}}(\mathbb{R}^d,\mathbb{C}^d)}\leq c^{(1)}_{d,\epsilon}c^{(2)}_{d,\epsilon}\Big\|[(-\Delta)^{\frac{\epsilon}{2}},M_f]\Big\|_{\frac{d}{1-\epsilon},\infty}.$$
The assertion follows from Lemma \ref{standard convolution lemma}.
\end{proof}

\section{Proof of Theorem \ref{epsilon1}}\label{sec:epsilon1}

\begin{proof}[Proof of Theorem \ref{epsilon1} \eqref{mtabounded}]
The first part follows from \cite[Theorem 2]{Ca}. Indeed, note that for sufficiently nice functions $\phi,$
$$(-\Delta)^\frac12\phi(t) = c_d \int_{\mathbb{R}^d} \frac{\phi(s)-\phi(t)}{|t-s|^{d+1}}ds,$$
where the integral is understood in the principal value sense. Therefore $[(-\Delta)^\frac12,M_f]$ has integral kernel
$$c_d\,\frac{f(t)-f(s)}{|t-s|^{d+1}},$$
and this is precisely the setting of \cite[Theorem 2]{Ca}.
\end{proof}

For the proofs of \eqref{mtcbounded} and \eqref{mtdbounded}, we use the following notation
$$e_{\xi}(t):=e^{i\xi\cdot t},\quad t,\xi,\in\mathbb{R}^d.$$
and the following auxiliary results.

\begin{lem}\label{inner product compute lemma} Let $\phi,\psi,f\in \mathcal{S}(\mathbb{R}^d).$ We have
$$\langle \psi, f\phi \rangle
=(2\pi)^{-\frac{d}{2}}\int_{\mathbb{R}^d\times\mathbb{R}^d} \widehat f(\eta_1-\eta_2) \overline{\widehat\psi(\eta_1)} \widehat\phi(\eta_2)d\eta_1\,d\eta_2,$$
$$\langle \psi, [(-\Delta)^\frac12,M_f] \phi \rangle
=(2\pi)^{-\frac{d}{2}}\int_{\mathbb{R}^d\times\mathbb{R}^d} (|\eta_1|-|\eta_2|) \widehat{f}(\eta_1-\eta_2) \overline{\widehat{\psi}(\eta_1)} \widehat{\phi}(\eta_2)d\eta_1\,d\eta_2.$$
Here $\widehat{f}(x) = (2\pi)^{-\frac{d}{2}} \int_{\mathbb{R}^d} e^{-i\xi\cdot x} f(x)dx$ denotes the Fourier transform.
\end{lem}
\begin{proof} The first assertion is standard and the proof is omitted.
	
To prove the second assertion, suppose first that $\hat{\phi}$ and $\hat{\psi}$ vanish near $0.$ It follows that $(-\Delta)^{\frac12}\phi$ and $(-\Delta)^{\frac12}\psi$ are Schwartz functions. Note that
$$\widehat{(-\Delta)^{\frac12}\psi}(\eta_1)=|\eta_1|\hat{\psi}(\eta_1),\quad \widehat{(-\Delta)^{\frac12}\phi}(\eta_2)=|\eta_2|\hat{\phi}(\eta_2),\quad \eta_1,\eta_2\in\mathbb{R}^d.$$
We have
$$\langle \psi, [(-\Delta)^\frac12,M_f] \phi \rangle=\langle (-\Delta)^\frac12\psi, f\phi\rangle-\langle \psi, f((-\Delta)^\frac12\phi) \rangle.$$
Using the first assertion, we write
$$\langle (-\Delta)^\frac12\psi, f\phi\rangle=(2\pi)^{-\frac{d}{2}}\int_{\mathbb{R}^d\times\mathbb{R}^d} |\eta_1|\hat{f}(\eta_1-\eta_2) \overline{\hat{\psi}(\eta_1)} \hat{\phi}(\eta_2)d\eta_1\,d\eta_2,$$
$$\langle \psi, f((-\Delta)^\frac12\phi) \rangle=(2\pi)^{-\frac{d}{2}}\int_{\mathbb{R}^d\times\mathbb{R}^d}|\eta_2| \hat{f}(\eta_1-\eta_2) \overline{\hat{\psi}(\eta_1)} \hat{\phi}(\eta_2)d\eta_1\,d\eta_2.$$
Combining these equalities, we obtain the second assertion for the case when $\hat{\phi}$ and $\hat{\psi}$ vanish near $0.$ 

However,
$$\Big|\int_{\mathbb{R}^d\times\mathbb{R}^d} (|\eta_1|-|\eta_2|) \hat{f}(\eta_1-\eta_2) \overline{\hat{\psi}(\eta_1)} \hat{\phi}(\eta_2)d\eta_1\,d\eta_2\Big|\leq\|\hat{f'}\|_{\infty}\|\hat{\phi}\|_1\|\hat{\psi}\|_1.$$
So, the right hand side in the second assertion is a continuous functional of $\hat{\psi}$ in $L_1$-norm. One can also see that so is the left hand side. Thus, we can remove the restriction on $\phi$ and $\psi.$ 
\end{proof}

\begin{lem}\label{epsilon1lem}
Let $f\in C^{\infty}_c(\mathbb{R}^d)$ and let $\omega\in\mathbb{S}^{d-1}.$ We have
$$M_{e_{-n\omega}}[(-\Delta)^\frac12,M_f] M_{e_{n\omega}}\to M_{\omega\cdot\nabla f},\quad n\to\infty,$$ 
in the weak operator topology.
\end{lem}

\begin{proof} Let $\phi,\psi\in C^{\infty}_c(\mathbb{R}^d).$ Note that
$$\widehat{e_{n\omega}\phi}(\eta_1)=\hat{\phi}(\eta_1-n\omega),\quad \widehat{e_{n\omega}\psi}(\eta_2)=\hat{\psi}(\eta_2-n\omega).$$	
Applying the second assertion of Lemma \ref{inner product compute lemma} to the functions $e_{n\omega}\phi,e_{n\omega\psi},f\in C^{\infty}_c(\mathbb{R}^d),$ we arrive at
$$\langle \psi,\big(M_{e_{-n\omega}}[(-\Delta)^\frac12,M_f] M_{e_{n\omega}}\big)\phi \rangle=\langle e_{n\omega}\psi, [(-\Delta)^\frac12,M_f] (e_{n\omega}\phi) \rangle=$$
$$=(2\pi)^{-\frac{d}{2}} \int_{\mathbb{R}^d\times\mathbb{R}^d} (|\eta_1|-|\eta_2|) \widehat f(\eta_1-\eta_2) \overline{\widehat\psi(\eta_1-n\omega)} \widehat\phi(\eta_2-n\omega) d\eta_1\,d\eta_2.$$
Replacing $\eta_1$ with $\eta_1+n\omega$ and $\eta_2$ with $\eta_2+n\omega,$ we obtain
$$\langle \psi,\big(M_{e_{-n\omega}}[(-\Delta)^\frac12,M_f] M_{e_{n\omega}}\big)\phi \rangle=$$
$$=(2\pi)^{-\frac{d}{2}} \int_{\mathbb{R}^d\times\mathbb{R}^d} (|\eta_1+n\omega|-|\eta_2+n\omega|) \widehat f(\eta_1-\eta_2) \overline{\widehat\psi(\eta_1)} \widehat\phi(\eta_2)d\eta_1\,d\eta_2.$$

Set
\begin{align*}
	F_n(\eta_1,\eta_2) & :=(|\eta_1+n\omega|-|\eta_2+n\omega|) \widehat f(\eta_1-\eta_2) \overline{\widehat\psi(\eta_1)} \widehat\phi(\eta_2),\quad\eta_1,\eta_2\in\mathbb{R}^d, \\
	F(\eta_1,\eta_2) & :=(|\eta_1-\eta_2|) \widehat f(\eta_1-\eta_2) \overline{\widehat\psi(\eta_1)} \widehat\phi(\eta_2),\quad\eta_1,\eta_2\in\mathbb{R}^d, \\
	F_{\infty}(\eta_1,\eta_2) &:=(\omega\cdot(\eta_1-\eta_2)) \widehat f(\eta_1-\eta_2) \overline{\widehat\psi(\eta_1)} \widehat\phi(\eta_2),\quad\eta_1,\eta_2\in\mathbb{R}^d.
\end{align*}
It follows from the triangle inequality that $|F_n|\leq F$ on $\mathbb{R}^d\times\mathbb{R}^d.$ We have
$$\lim_{n\to\infty}|\eta_1+n\omega|-|\eta_2+n\omega| \to \omega\cdot(\eta_1- \eta_2),\quad \eta_1,\eta_2\in\mathbb{R}^d.$$
Thus, $F_n\to F_{\infty}$ pointwise. Moreover, since
$$|F(\eta_1,\eta_2)|\leq \|\widehat{(-\Delta)^{\frac12}f}\|_{\infty}|\hat{\psi}|(\eta_1)|\hat{\phi}|(\eta_2),\quad \eta_1,\eta_2\in\mathbb{R},$$
it follows that $F$ is integrable on $\mathbb{R}^d\times\mathbb{R}^d.$ By the Dominated Convergence Theorem, we have
$$\int_{\mathbb{R}^d\times\mathbb{R}^d}F_n(\eta_1,\eta_2)d\eta_1d\eta_2\to \int_{\mathbb{R}^d\times\mathbb{R}^d}F(\eta_1,\eta_2)d\eta_1d\eta_2,\quad n\to\infty.$$
Thus,
$$\lim_{n\to\infty}\langle \psi,\big(M_{e_{-n\omega}}[(-\Delta)^\frac12,M_f] M_{e_{n\omega}}\big)\phi \rangle=$$
$$=(2\pi)^{-\frac{d}{2}}\int_{\mathbb{R}^d\times\mathbb{R}^d}(\omega\cdot(\eta_1-\eta_2)) \widehat f(\eta_1-\eta_2) \overline{\widehat\psi(\eta_1)} \widehat\phi(\eta_2)d\eta_1d\eta_2=$$
$$=(2\pi)^{-\frac{d}{2}}\int_{\mathbb{R}^d\times\mathbb{R}^d}\widehat{\omega\cdot\nabla f}(\eta_1-\eta_2) \overline{\widehat\psi(\eta_1)} \widehat\phi(\eta_2)d\eta_1d\eta_2=\langle \psi,(\omega\cdot\nabla f)\phi\rangle.$$
Since our sequence of operators is bounded in the uniform norm, the assertion follows.
\end{proof}

\begin{lem}\label{rf lipschitz lemma} Let $f\in C^{\infty}(\mathbb{R}^d)$ be such that $[(-\Delta)^{\frac{1}{2}},M_f]$ is bounded. We have $f\in\dot{W}^{1}_{\infty}(\mathbb{R}^d)$ and
$$\|f\|_{\dot{W}^{1}_{\infty}(\mathbb{R}^d)}\leq\|[(-\Delta)^{\frac{1}{2}},M_f]\|_{\infty}.$$
\end{lem}
\begin{proof} 
	Let $\chi\in C^{\infty}_c(\mathbb{R}^d)$ with $\|\chi\|_{\infty}=1$ and choose $\theta\in C^{\infty}_c(\mathbb{R}^d)$ with $\theta\chi=\chi$. It is immediate that $f\theta\in C^{\infty}_c(\mathbb{R}^d)$ and
$$M_{\bar{\chi}} [(-\Delta)^{\frac{1}{2}},M_{f\theta}]M_{\chi}=M_{\bar{\chi}} [(-\Delta)^{\frac{1}{2}},M_{f}]M_{\chi}.$$
Therefore,
$$\|M_{\bar{\chi}} [(-\Delta)^{\frac{1}{2}},M_{f\theta}]M_{\chi}\|_{\infty}\leq\| [(-\Delta)^{\frac{1}{2}},M_f]\|_{\infty}$$
and, for every $n\in\mathbb{N}$ and for every $\omega\in\mathbb{S}^{d-1}$,
$$\|M_{\bar{\chi}}\cdot M_{e_{-n\omega}} [(-\Delta)^{\frac{1}{2}},M_{f\theta}]M_{e_{n\omega}}\cdot M_{\chi}\|_{\infty}\leq\| [(-\Delta)^{\frac{1}{2}},M_f]\|_{\infty}.$$
Letting $n\to\infty,$ we deduce from Lemma \ref{epsilon1lem} that for every $\omega\in\mathbb{S}^{d-1}$
$$\|M_{\bar{\chi}}\cdot M_{\omega\cdot\nabla(f\theta)}\cdot M_{\chi}\|_{\infty}\leq\| [(-\Delta)^{\frac{1}{2}},M_f]\|_{\infty}.$$
In other words, for every $\omega\in\mathbb{S}^{d-1}$ we have 
$$\||\chi|^2(\omega\cdot \nabla(f\theta))\|_{L_{\infty}(\mathbb{R}^d)}\leq\| [(-\Delta)^{\frac{1}{2}},M_f]\|_{\infty}.$$

We have $|\chi|^2\nabla\theta=0$ and therefore $|\chi|^2(\omega\cdot\nabla(f\theta))=|\chi|^2(\omega\nabla f)$.
Thus, we have shown that for every $\omega\in\mathbb{S}^{d-1}$ we have 
$$\||\chi|^2(\omega\cdot\nabla f)\|_{L_{\infty}(\mathbb R^d}\leq\| [(-\Delta)^{\frac{1}{2}},M_f]\|_{\infty}$$
Taking the supremum over all $\chi\in C^\infty_c(\mathbb R^d)$ with $\|\chi\|_\infty=1$, we obtain
$$\|\omega\cdot \nabla f\|_{L_{\infty}(\mathbb R^d)}\leq\| [(-\Delta)^{\frac{1}{2}},M_f]\|_{\infty}$$
for every $\omega\in\mathbb{S}^{d-1}.$ Since $\nabla f$ is continuous, we deduce that
$$|\omega\cdot (\nabla f)( x)|\leq\| [(-\Delta)^{\frac{1}{2}},M_f]\|_{\infty},\quad 
\qquad\text{for all}\ \omega\in\mathbb S^{d-1}\ \text{and all}\ x\in\mathbb R^d.$$
For given $x\in\mathbb R^d$ with $\nabla f(x)\neq 0$ we pick $\omega=\nabla f(x)/|\nabla f(x)|$ and obtain
$$
|(\nabla f)( x)| \leq\| [(-\Delta)^{\frac{1}{2}},M_f]\|_{\infty},\quad 
\qquad x\in\mathbb R^d,
$$
which is the claimed inequality.
\end{proof}

\begin{proof}[Proof of Theorem \ref{epsilon1} \eqref{mtcbounded}] Let $f\in L_{1,{\rm loc}}(\mathbb R^d)$ and assume the operator $[(-\Delta)^{\frac{1}{2}},M_f]$ is boun\-ded. Let $\Phi\in C^{\infty}_c(\mathbb{R}^d)$ be nonnegative with $\|\Phi\|_1=1.$ By Lemma \ref{main convolution lemma}, we have
$$\|[(-\Delta)^{\frac12},M_{f\ast\Phi}]\|_{\infty}\leq\|[(-\Delta)^{\frac12},M_f]\|_{\infty}.$$
Since $f\ast\Phi\in C^{\infty}(\mathbb{R}^d),$ it follows from Lemma \ref{rf lipschitz lemma} that
$$\|f\ast\Phi\|_{\dot{W}^{1}_{\infty}(\mathbb{R}^d)}\leq \|[(-\Delta)^{\frac12},M_f]\|_{\infty}.$$
The assertion follows now from Lemma \ref{standard convolution lemma}.	
\end{proof}

\begin{lem}\label{last rf lemma} Let $f,\chi\in C^{\infty}_c(\mathbb{R}^d)$. If  $M_{\bar{\chi}}[(-\Delta)^{\frac{1}{2}},M_f]M_{\chi}$ extends to a compact operator, then $\chi\nabla f=0.$
\end{lem}
\begin{proof} For every $\psi,\phi\in C_c^{\infty}(\mathbb{R}^d)$ and for every $\omega\in\mathbb{S}^{d-1},$ we have $e_{n\omega}\psi,e_{n\omega}\phi\to0$ weakly to zero in $L_2(\mathbb{R}^d)$ as $n\to\infty.$ Therefore, the compactness of the operator implies
$$\langle e_{n\omega}\chi\psi, [(-\Delta)^{\frac{1}{2}},M_f] e_{n\omega}\chi \phi \rangle \to 0,\quad n\to\infty.$$
In other words,
$$\langle \chi\psi, (M_{e_{-n\omega}}[(-\Delta)^{\frac{1}{2}},M_f] M_{e_{n\omega}})(\chi \phi) \rangle \to 0,\quad n\to\infty.$$
By Lemma \ref{epsilon1lem}, we deduce that for every $\omega\in\mathbb S^{d-1}$
$$\langle \chi\psi, M_{\omega\cdot\nabla f}(\chi \phi) \rangle=0.$$
In particular,
$$\int_{\mathbb{R}^d}|\chi|^2\bar{\psi}\phi\partial_kf=0,\quad 1\leq k\leq d.$$
Taking $\bar{\psi}\phi$ to be (approximately) the sign of $\partial_kf,$ we conclude that
$$\int_{\mathbb{R}^d}|\chi|^2|\partial_kf|=0,\quad 1\leq k\leq d.$$
This completes the proof.
\end{proof}

\begin{proof}[Proof of Theorem \ref{epsilon1} \eqref{mtdbounded}]
	Let $f\in L_{1,{\rm loc}}(\mathbb R^d)$ and assume that $[(-\Delta)^\frac12,M_f]$ extends to a compact operator.
	 	
Let $\Phi\in C^{\infty}_c(\mathbb{R}^d)$ be nonnegative with $\|\Phi\|_1=1.$ Let $\chi\in C^{\infty}_c(\mathbb{R}^d)$ and choose $\theta\in C^{\infty}_c(\mathbb{R}^d)$ with $\theta\chi=\chi$.

By Lemma \ref{main convolution lemma}, $[(-\Delta)^{\frac{1}{2}},M_{\Phi\ast f}]$ is compact. Therefore,
$$M_{\bar{\chi}}[(-\Delta)^{\frac{1}{2}},M_{(\Phi\ast f)\theta}] M_\chi
= M_{\bar{\chi}}[(-\Delta)^{\frac{1}{2}},M_{\Phi\ast f}] M_\chi$$
is also compact. Clearly, $(\Phi\ast f)\theta\in C^{\infty}_c(\mathbb{R}^d).$ It follows from Lemma \ref{last rf lemma} that
$$\chi\cdot\nabla((f\ast\Phi)\theta)=0.$$
We have $\chi\nabla\theta=0$ and therefore
$$\chi\cdot\nabla(f\ast\Phi)=0.$$
Since $\chi\in C_c^\infty(\mathbb R^d)$ is arbitrary, it follows that
$$(\nabla f)\ast\Phi=\nabla(f\ast\Phi)=0.$$
Since $\Phi$ is arbitrary, it follows that $\nabla f=0.$ This completes the proof.
\end{proof}


\appendix

\section{Commutator estimate}

Throughout this appendix we assume $d\geq 2.$ Our goal is to prove the following commutator bound. We recall that the algebra $\Pi$ is defined at the beginning of Section \ref{sec:spectralasymp}.

\begin{thm}\label{t delta commutator lemma}  If $T\in\Pi$ is compactly supported, then
$$[T,(1-\Delta)^{-\frac{p}{2}}]\in (\mathcal{L}_{\frac{d}{p},\infty})_0,\quad p>0.$$
\end{thm}

The following lemma is a standard result, but for the convenience of the reader we provide a proof.

\begin{lem}\label{first appendix commutator lemma} If $f\in C^{\infty}_c(\mathbb{R}^d),$ then 
$$[M_f,(1-\Delta)^{-\frac{p}{2}}]\in \mathcal{L}_{\frac{d}{p+1},\infty},\quad p>0.$$
\end{lem}
\begin{proof} We write $f=f_1f_2,$ $f_1,f_2\in C^{\infty}_c(\mathbb{R}^d).$ By the Leibniz rule, we have
$$[M_f,(1-\Delta)^{-\frac{p}{2}}]=[M_{f_1},(1-\Delta)^{-\frac{p}{2}}]\cdot M_{f_2}+M_{f_1}\cdot [M_{f_2},(1-\Delta)^{-\frac{p}{2}}]=X_1Y_1+Y_2X_2,$$
where we denote
$$X_1=[M_{f_1},(1-\Delta)^{-\frac{p}{2}}](1-\Delta)^{\frac{p+1}{2}},\quad X_2=(1-\Delta)^{\frac{p+1}{2}}[M_{f_2},(1-\Delta)^{-\frac{p}{2}}],$$
$$Y_1=(1-\Delta)^{-\frac{p+1}{2}}M_{f_2},\quad Y_2=M_{f_1}(1-\Delta)^{-\frac{p+1}{2}}.$$
Clearly, $X_1$ and $X_2$ are pseudodifferential operators of order $0.$ Hence, they are bounded. By Theorem \ref{cwikel estimate} we have $Y_1,Y_2\in\mathcal{L}_{\frac{d}{p+1},\infty}.$ This completes the proof.
\end{proof}

Using standard arguments we can now weaken the regularity on $f$ required in Lemma \ref{first appendix commutator lemma}.

\begin{lem}\label{second appendix commutator lemma} If $f\in C_c(\mathbb{R}^d),$ then 
$$[M_f,(1-\Delta)^{-\frac{p}{2}}]\in(\mathcal{L}_{\frac{d}{p},\infty})_0,\quad p>0.$$
\end{lem}
\begin{proof} Assume for definiteness that $f$ is supported on $[-1,1]^d.$ Choose a se\-quence $\{f_n\}_{n\geq0}$ $\subset C^{\infty}_c(\mathbb{R}^d)$ supported in $[-1,1]^d$ such that $f_n\to f$ in the uniform norm. Using quasi-triangle and H\"older inequalities, we write
$$\|[M_{f_n-f},(1-\Delta)^{-\frac{p}{2}}]\|_{\frac{d}{p},\infty}\leq 2^{1+\frac{p}{d}}\|f_n-f\|_{\infty}\|M_{\chi_{[-1,1]^d}}(1-\Delta)^{-\frac{p}{2}}]\|_{\frac{d}{p},\infty}.$$
The second factor on the right hand side is finite by Theorem \ref{cwikel estimate}. Therefore,
$$[M_{f_n},(1-\Delta)^{-\frac{p}{2}}]\to [M_f,(1-\Delta)^{-\frac{p}{2}}]$$
in $\mathcal{L}_{\frac{d}{p},\infty}.$ Since the sequence on the left hand side belongs to $(\mathcal{L}_{\frac{d}{p},\infty})_0$ by Lemma \ref{first appendix commutator lemma}, the assertion follows.
\end{proof}

The following result is a special case of \cite[Theorem 1.6]{MSX} (with $\alpha=1$ and $\beta=p$).

\begin{lem}\label{MSX theorem} If $f\in C^{\infty}_c(\mathbb{R}^d),$ then
$$
[(1-\Delta)^{\frac12},M_f] (1-\Delta)^{-\frac{p}{2}}\in\mathcal{L}_{\frac{d}{p},\infty},\quad p>0.$$
\end{lem}

The following lemma is a technical precursor to Lemma \ref{appendix final lemma}.

\begin{lem}\label{rk appendix lemma} If $f\in C^{\infty}_c(\mathbb{R}^d),$  then 
$$[D_k(-\Delta)^{-\frac12},M_f(1-\Delta)^{-\frac{p}{2}}]\in\mathcal{L}_{\frac{d}{p+1},\infty},\quad p>0,\quad k\in\{1,\cdots,d\}.$$
\end{lem}
\begin{proof} Set 
$$h_k(s)=\frac{s_k}{|s|}-\frac{s_k}{(1+|s|^2)^{\frac12}},\quad s\in\mathbb{R}^d.$$
It is immediate that
$$[D_k(-\Delta)^{-\frac12},M_f(1-\Delta)^{-\frac{p}{2}}]-[D_k(1-\Delta)^{-\frac12},M_f(1-\Delta)^{-\frac{p}{2}}]=$$
$$=h_k(\nabla)M_f(1-\Delta)^{-\frac{p}{2}}-M_f(1-\Delta)^{-\frac{p}{2}}h_k(\nabla)=$$
$$=h_k(\nabla)(1-\Delta) \cdot (1-\Delta)^{-1}M_f(1-\Delta)^{-\frac{p}{2}}-M_f(1-\Delta)^{-\frac{p+2}{2}}\cdot (1-\Delta) h_k(\nabla).$$
Using Theorem \ref{cwikel estimate} as in the proof of Lemma \ref{first appendix commutator lemma}, we obtain
$$(1-\Delta)^{-1}M_f(1-\Delta)^{-\frac{p}{2}},\quad M_f(1-\Delta)^{-\frac{p+2}{2}}\in\mathcal{L}_{\frac{d}{p+2},\infty}.$$
Since $h_k(\nabla)\cdot (1-\Delta)$ is bounded, it follows from \eqref{eq:traceidealbound} that
\begin{equation}\label{ral eq0}
[D_k(-\Delta)^{-\frac12},M_f(1-\Delta)^{-\frac{p}{2}}]-[D_k(1-\Delta)^{-\frac12},M_f(1-\Delta)^{-\frac{p}{2}}]\in\mathcal{L}_{\frac{d}{p+2},\infty}.
\end{equation}

Next,
$$[D_k(1-\Delta)^{-\frac12},M_f(1-\Delta)^{-\frac{p}{2}}]=[D_k(1-\Delta)^{-\frac12},M_f]\cdot (1-\Delta)^{-\frac{p}{2}}=$$
$$=M_{D_kf}(1-\Delta)^{-\frac{p+1}{2}}-D_k(1-\Delta)^{-\frac12}\cdot [(1-\Delta)^{\frac12},M_f](1-\Delta)^{-\frac{p+1}{2}}.$$
Applying Theorem \ref{cwikel estimate} to the first term and Theorem \ref{MSX theorem} (applied with $p+1$ instead of $p$) to the second term, we deduce that
\begin{equation}\label{ral eq1}
[D_k(-\Delta)^{-\frac12},M_f(1-\Delta)^{-\frac{p}{2}}]\in\mathcal{L}_{\frac{d}{p+1},\infty}.
\end{equation}
Combining \eqref{ral eq0} and \eqref{ral eq1}, we complete the proof.
\end{proof}

The next lemma is the crucial step in proving Theorem \ref{t delta commutator lemma}.

\begin{lem}\label{appendix final lemma} If $f\in C_c(\mathbb{R}^d)$ and if $g\in C(\mathbb{S}^{d-1}),$ then 
$$[g(\nabla(-\Delta)^{-\frac12}),M_f(1-\Delta)^{-\frac{p}{2}}]\in(\mathcal{L}_{\frac{d}{p},\infty})_0,\quad p>0.$$
\end{lem}
\begin{proof} If $g\in C(\mathbb{S}^{d-1})$ is a monomial, then Lemma \ref{rk appendix lemma} and the Leibniz rule yield
$$[g(\nabla(-\Delta)^{-\frac12}),M_f(1-\Delta)^{-\frac{p}{2}}]\in\mathcal{L}_{\frac{d}{p+1},\infty}.$$
If $g\in C(\mathbb{S}^{d-1})$ is a polynomial, then we obtain by linearity
$$[g(\nabla(-\Delta)^{-\frac12}),M_f(1-\Delta)^{-\frac{p}{2}}]\in\mathcal{L}_{\frac{d}{p+1},\infty}\subset(\mathcal{L}_{\frac{d}{p},\infty})_0.$$

Consider now the general case. Fix $g\in C(\mathbb{S}^{d-1})$ and choose a sequence of polynomials $\{g_n\}_{n\geq0}\in C(\mathbb{S}^{d-1})$ such that $g_n\to g$ in the uniform norm. Using quasi-triangle and H\"older inequalities, we write
$$\|[(g_n-g)(\nabla(-\Delta)^{-\frac12}),M_f(1-\Delta)^{-\frac{p}{2}}]\|_{\frac{d}{p},\infty}\leq 2^{1+\frac{p}{d}}\|g_n-g\|_{\infty}\|M_f(1-\Delta)^{-\frac{p}{2}}\|_{\frac{d}{p},\infty}.$$
The second factor on the right hand side is finite by Theorem \ref{cwikel estimate}. Therefore,
$$[g_n(\nabla(-\Delta)^{-\frac12}),M_f(1-\Delta)^{-\frac{p}{2}}]\to [g(\nabla(-\Delta)^{-\frac12}),M_f(1-\Delta)^{-\frac{p}{2}}]$$
in $\mathcal{L}_{\frac{d}{p},\infty}.$ Since the sequence on the left is in $(\mathcal{L}_{\frac{d}{p},\infty})_0,$ the assertion follows.
\end{proof}

Next, we derive an approximation result for operators $T\in\Pi$.

\begin{lem}\label{t standard approximation lemma} Let $T\in\Pi.$ There are $\{f_{n,l}\}_{n,l\geq1}\subset (C_0+\mathbb{C})(\mathbb{R}^d),$ $\{g_{n,l}\}_{n,l\geq1}\subset C(\mathbb{S}^{d-1})$ and $\{S_n\}_{n\geq0}\subset\mathcal{K}(L_2(\mathbb{R}^d))$ such that, with the limit in the uniform norm,
$$T=\lim_{n\to\infty}T_n,\quad T_n :=\sum_{l=1}^{l_n}M_{f_{n,l}}g_{n,l}(\nabla(-\Delta)^{-\frac12})+S_n.$$
\end{lem}

\begin{proof} If $T\in\Pi,$ then we can find a sequence $\{T_n\}_{n\geq0}\subset\Pi$ such that $T_n\to T$ in the uniform norm and such that 
$$T_n=\sum_{l=1}^{l_n}\prod_{k=1}^{k_n}M_{f_{n,k,l}}g_{n,k,l}(\nabla(-\Delta)^{-\frac12}).$$
Here, $f_{n,k,l}\in (C_0+\mathbb{C})(\mathbb{R}^d)$ and $g_{n,k,l}\in C(\mathbb{S}^{d-1})$ for every $n,$ every $k$ and every $l.$

By \cite[Lemma 5.8]{DAO3}, the operator
$$\prod_{k=1}^{k_n}M_{f_{n,k,l}}g_{n,k,l}(\nabla(-\Delta)^{-\frac12})-M_{\prod_{k=1}^{k_n}f_{n,k,l}}(\prod_{k=1}^{k_n}g_{n,k,l})(\nabla(-\Delta)^{-\frac12})$$
is compact. Setting
$$f_{n,l}=\prod_{k=1}^{k_n}f_{n,k,l},\quad g_{n,l}=\prod_{k=1}^{k_n}g_{n,k,l},\quad 1\leq l\leq l_n$$
the assertion follows.
\end{proof}

We are now in position to prove the main result of this appendix.

\begin{proof}[Proof of Theorem \ref{t delta commutator lemma}] Let $T\in\Pi$ and let $T_n\in\Pi$ be as in Lemma \ref{t standard approximation lemma}. Let $\phi\in C^{\infty}_c(\mathbb{R}^d).$ By the Leibniz rule, we have
\begin{multline}\label{a1 main eq}
[T_n,M_{\phi}(1-\Delta)^{-\frac{p}{2}}]=\sum_{l=1}^{l_n} [M_{f_{n,l}},M_{\phi}(1-\Delta)^{-\frac{p}{2}}]\cdot g_{n,l}(\nabla(-\Delta)^{-\frac12})+\\
+\sum_{l=1}^{l_n}M_{f_{n,l}}\cdot [g_{n,l}(\nabla(-\Delta)^{-\frac12}),M_{\phi}(1-\Delta)^{-\frac{p}{2}}]+\\
+[S_n,M_{\phi}(1-\Delta)^{-\frac{p}{2}}].
\end{multline}
The commutators on the right hand side of \eqref{a1 main eq} belong to $(\mathcal{L}_{\frac{d}{p},\infty})_0.$ Indeed,
$$[M_{f_{n,l}},M_{\phi}(1-\Delta)^{-\frac{p}{2}}]=[M_{f_{n,l}\phi},(1-\Delta)^{-\frac{p}{2}}]-[M_{\phi},(1-\Delta)^{-\frac{p}{2}}]\cdot M_{f_{n,l}}.$$
Hence, the first commutator on the right hand side of \eqref{a1 main eq} belongs to $(\mathcal{L}_{\frac{d}{p},\infty})_0$ by Lemma \ref{second appendix commutator lemma}. The second commutator on the right hand side  of \eqref{a1 main eq} belongs to $(\mathcal{L}_{\frac{d}{p},\infty})_0$ by Lemma \ref{appendix final lemma}. The third commutator on the right hand side of \eqref{a1 main eq} belongs to $(\mathcal{L}_{\frac{d}{p},\infty})_0$ since $M_{\phi}(1-\Delta)^{-\frac{p}{2}}\in\mathcal{L}_{\frac{d}{p},\infty}$ and since $S_n$ is compact. Therefore,
$$[T_n,M_{\phi}(1-\Delta)^{-\frac{p}{2}}]\in (\mathcal{L}_{\frac{d}{p},\infty})_0.$$
Since $T_n\to T$ in the uniform norm and since $M_\phi(1-\Delta)^{-\frac{p}{2}}\in\mathcal{L}_{\frac{d}{p},\infty}$ by Theorem~\ref{cwikel estimate}, it follows that
$$[T_n,M_{\phi}(1-\Delta)^{-\frac{p}{2}}]\to [T,M_{\phi}(1-\Delta)^{-\frac{p}{2}}]\,,$$
where the limit is taken in $\mathcal{L}_{\frac{d}{p},\infty}.$ Since the sequence on the left hand side belongs to $(\mathcal{L}_{\frac{d}{p},\infty})_0,$ it follows that
$$[T,M_{\phi}(1-\Delta)^{-\frac{p}{2}}]\in (\mathcal{L}_{\frac{d}{p},\infty})_0.$$

Suppose now that $T\in\Pi$ is compactly supported, that is, $T=M_{\phi}T=TM_{\phi}$ for some $\phi\in C^{\infty}_c(\mathbb{R}^d).$ We have
$$[T,(1-\Delta)^{-\frac{p}{2}}]=T(1-\Delta)^{-\frac{p}{2}}-(1-\Delta)^{-\frac{p}{2}}T=$$
$$=TM_{\phi}(1-\Delta)^{-\frac{p}{2}}-(1-\Delta)^{-\frac{p}{2}}M_{\phi}T=[T,M_{\phi}(1-\Delta)^{-\frac{p}{2}}]+[M_{\phi},(1-\Delta)^{-\frac{p}{2}}]\cdot T.$$
The first term belongs to $(\mathcal{L}_{\frac{d}{p},\infty})_0$ as we have just shown above, and the second one does by Lemma \ref{first appendix commutator lemma}. This completes the proof.
\end{proof}



\begin{thebibliography}{99}
\bibitem{Ar} Araki H. {\it On an inequality of Lieb and Thirring}. Lett. Math. Phys. \textbf{19} (1990), no. 2, 167--170.	
\bibitem{ArFiPe} Arazy J., Fisher S. D., Peetre J. \textit{Hankel operators on weighted Bergman spaces}. Amer. J. Math. \textbf{110} (1988), no. 6, 989--1053.	
\bibitem{ArFiJaPe} Arazy J., Fisher S. D., Janson S., Peetre J. \textit{Membership of Hankel operators on the ball in unitary ideals}. J. London Math. Soc. (2) \textbf{43} (1991), no. 3, 485--508.
\bibitem{BKS} Birman M. Sh., Karadzhov G. E., Solomyak M. Z. \textit{Boundedness conditions and spectrum estimates for the operators $b(X)a(D)$ and their analogs}. In: Estimates and asymptotics for discrete spectra of integral and differential equations (Leningrad, 1989–90), 85--106, Adv. Soviet Math., 7, Amer. Math. Soc., Providence, RI, 1991.
\bibitem{BiSo70} Birman M. Sh., Solomyak M. {\it Asymptotics of the spectrum of weakly polar integral operators.} (Russian) Izv. Akad. Nauk SSSR Ser. Mat. \textbf{34} (1970), no. 5, 1142--1158. English translation in: Math. USSR Izv. \textbf{4} (1970), no. 5, 1151--1168.
\bibitem{BiSo77} Birman M. Sh., Solomjak M. Z. \textit{Estimates for the singular numbers of integral operators}. Uspekhi Mat. Nauk \textbf{32} (1977), no. 1 (193), 17--84. English translation in: Russ. Math. Surv. \textbf{32} (1977), no. 1, 15--89.
\bibitem{BiSo} Birman M. Sh., Solomyak M. Z. \textit{Asymptotic behavior of the spectrum of pseudodifferential operators with anisotropically homogeneous symbols I,II}. (Russian) Vestnik Leningrad. Univ. \textbf{13} (1977), no. 3, 13--21; \textbf{13} (1979), no. 3, 5--10.
English translation in: Vestn. Leningr. Univ., Math. \textbf{10} (1982), 237--247; \textbf{12} (1980), 155--161.
\bibitem{BiSo-book} Birman M. S., Solomyak M. {\it Spectral theory of selfadjoint operators in Hilbert space.} Translated from the 1980 Russian original. Mathematics and its Applications (Soviet Series). D. Reidel Publishing Co., Dordrecht, 1987.
\bibitem{Ca} Calder\'on A. {\it Commutators of singular integral operators}. Proc. Nat. Acad. Sci. U.S.A. {\bf 53} (1965), 1092--1099. 
\bibitem{Ca2} Calder\'on A. \textit{Commutators, singular integrals on Lipschitz curves and applications}. Proceedings of the International Congress of Mathematicians (Helsinki, 1978), pp. 85--96, Acad. Sci. Fennica, Helsinki, 1980.
\bibitem{CMPS} Caspers M., Montgomery-Smith S., Potapov D., Sukochev F. {\it The best constants for operator Lipschitz functions on Schatten classes.} J. Funct. Anal. {\bf 267} (2014), no. 10, 3557--3579. 
\bibitem{CSZ-AiF} Caspers M., Sukochev F., Zanin D. {\it Weak type operator Lipschitz and commutator estimates for commuting tuples.} Ann. Inst. Fourier (Grenoble) {\bf 68} (2018), no. 4, 1643--1669.
\bibitem{ChDiHo} Chen Y., Ding Y., Hong G. {\it Commutators with fractional differentiation and new characterizations of BMO-Sobolev spaces}. Anal. PDE {\bf 9} (2016), no. 6, 1497--1522.
\bibitem{CS} Chilin V.,  Sukochev F. ,
{\it Weak convergence in non-commutative symmetric spaces,}
J. Operator Theory \textbf{31} (1994), 35--65.
\bibitem{CoMIMe} Coifman R. R., McIntosh A., Meyer Y. {\it L'int\'egrale de Cauchy d\'efinit un op\'erateur born\'e sur $L^2$ pour les courbes lipschitziennes}. Ann. of Math. (2) {\bf 116} (1982), no. 2, 361--387.
\bibitem{CoMe} Coifman R. R., Meyer Y. \textit{On commutators of singular integrals and bilinear singular integrals}. Trans. Amer. Math. Soc. \textbf{212} (1975), 315--331. 
\bibitem{CoRoWe} Coifman R. R., Rochberg R., Weiss G. \textit{Factorization theorems for Hardy spaces in several variables}. Ann. of Math. (2) \textbf{103} (1976), no. 3, 611--635.
\bibitem{CSZ} Connes A., Sukochev F., Zanin D. {\it Trace theorem for quasi-Fuchsian groups.} Mat. Sb. {\bf  208} (2017), 59--90.
\bibitem{CoSuTe} Connes A., Sullivan D., Teleman N. \textit{Quasiconformal mappings, operators on Hilbert space, and local formulae for characteristic classes}. Topology \textbf{33} (1994), no. 4, 663--681.
\bibitem{Cwikel} Cwikel M. {\it Weak type estimates for singular values and the number of bound states of Schr\"odinger operators.} Ann. of Math. (2) {\bf 106} (1977), no. 1, 93--100.
\bibitem{EnRo} Engli\v{s} M., Rochberg R. \textit{The Dixmier trace of Hankel operators on the Bergman space}. J. Funct. Anal. \textbf{257} (2009), no. 5, 1445--1479.
\bibitem{FeRo} Feldman M., Rochberg R. \textit{Singular value estimates for commutators and Hankel operators on the unit ball and the Heisenberg group}. Analysis and partial differential equations, 121--159, Lecture Notes in Pure and Appl. Math., 122, Dekker, New York, 1990.
\bibitem{Fr} Frank R. L. \textit{Cwikel's theorem and the CLR inequality}. J. Spectr. Theory \textbf{4} (2014), no.\ 1, 1--21.
\bibitem{FrSuZa} Frank R. L., Sukochev F., Zanin D. {\it Asymptotics of singular values for quantum derivatives}. Trans. Amer. Math. Soc.,  {\bf 376}  (2023),  no. 3, 2047--2088.
\bibitem{HSZ-holder} Huang J., Sukochev F., Zanin D. {\it Operator $\theta$-H\"older functions with respect to $\|\cdot\|_p,$ $0<p\leq\infty.$} J. Lond. Math. Soc. (2) {\bf 105} (2022), no. 4, 2436--2477.
\bibitem{Ja} Janson S. {\it Mean oscillation and commutators of singular integral operators}.
Ark. Mat. \textbf{16} (1978), no. 2, 263--270.
\bibitem{JaPe} Janson S., Peetre J. \textit{Paracommutators--boundedness and Schatten-von Neumann properties}. Trans. Amer. Math. Soc. \textbf{305} (1988), no. 2, 467--504.
\bibitem{JaRo} Janson S., Rochberg R. \textit{Intermediate Hankel operators on the Bergman space}. J. Operator Theory \textbf{29} (1993), no. 1, 137--155. 
\bibitem{JaWo} Janson S., Wolff T. H. \textit{Schatten classes and commutators of singular integral operators}. Ark. Mat. \textbf{20} (1982), no. 2, 301--310.
\bibitem{KPSS} Kissin E., Potapov D., Shulman V., Sukochev F. {\it Operator smoothness in Schatten norms for functions of several variables: Lipschitz conditions, differentiability and unbounded derivations.} Proc. Lond. Math. Soc. (3) {\bf 105} (2012), no. 4, 661--702.
\bibitem{DAO3} Kordyukov Y., Sukochev F., Zanin D. {\it $C^{\ast}$-algebraic approach to the principal symbol. III.} Journal of Noncommutative Geometry, to appear.
\bibitem{KPS} Krein S., Petunin Y., Semenov E. {\it Interpolation of linear operators.} Trans. Math. Mon.,  54, AMS, Providence, 1982.
\bibitem{Leoni-book} Leoni G. {\it A first course in Sobolev spaces.} Second edition. Graduate Studies in Mathematics, {\it 181}. American Mathematical Society, Providence, RI, 2017.
\bibitem{LeSZ} Levitina G., Sukochev F., Zanin D. {\it Cwikel estimates revisited.} Proc. Lond. Math. Soc. (3) {\bf 120} (2020), no. 2, 265--304.
\bibitem{LMSZ} Lord S., McDonald E., Sukochev F., Zanin D. \textit{Quantum differentiability of essentially bounded functions on Euclidean space}. J. Funct. Anal. \textbf{273} (2017), 2353--2387.
\bibitem{LSZ2012}  Lord S., Sukochev F., Zanin D. {\it  Singular traces: theory and applications}.  vol. {\bf 46},Walter de Gruyter,  2012.
\bibitem{LSZ-book-2} Lord S., Sukochev F., Zanin D. {\it Singular traces. Vol. 1--Theory. 2nd edition,} vol. {\bf 46/1}, Walter de Gruyter, Berlin (2021).
\bibitem{MSX} McDonald E., Sukochev F., Xiong X. {\it Quantum differentiability on noncommutative Euclidean spaces.} Comm. Math. Phys. {\bf 379} (2020), no. 2, 491--542.
\bibitem{MSZ-DAO2} McDonald E., Sukochev F., Zanin D. {\it  A $C^{\ast}$-algebraic approach to the principal symbol II.} Math. Ann. {\bf 374} (2019), no. 1--2, 273--322.
\bibitem{Mu} Murray M. A. M. {\it Commutators with fractional differentiation and BMO Sobolev spaces}. Indiana Univ. Math. J. {\bf 34} (1985), no. 1, 205--215.
\bibitem{PeRoWu} Peng, L. Z., Rochberg R., Wu Z. J. \textit{Orthogonal polynomials and middle Hankel operators on Bergman spaces}. Studia Math. \textbf{102} (1992), no. 1, 57--75. 
\bibitem{PS-crelle} Potapov D., Sukochev F. {\it Unbounded Fredholm modules and double operator integrals.} J. Reine Angew. Math. {\bf 626} (2009), 159--185.
\bibitem{PS-acta} Potapov D., Sukochev F. {\it  Operator-Lipschitz functions in Schatten-von Neumann classes.} Acta Math. {\bf 207} (2011), no. 2, 375--389.
\bibitem{RoSe2} Rochberg R., Semmes S. \textit{Nearly weakly orthonormal sequences, singular value estimates, and Calderon--Zygmund operators}. J. Funct. Anal. \textbf{86} (1989), no. 2, 237--306.
\bibitem{Simon-book}  Simon B. {\it Trace ideals and their applications.} Second edition. Mathematical Surveys and Monographs, {\bf 120}. American Mathematical Society, Providence, RI, 2005.
\bibitem{Stein} Stein E. {\it Singular integrals and differentiability properties of functions.} Princeton Mathematical Series, No. 30 Princeton University Press, Princeton, N.J.
\bibitem{St} Strichartz, R. S. \textit{Bounded mean oscillation and Sobolev spaces}. Indiana Univ. Math. J. \textbf{29} (1980), no. 4, 539--558.
\bibitem{SZ-DAO1} Sukochev F., Zanin D. {\it A $C^{\ast}$-algebraic approach to the principal symbol. I.} J. Operator Theory {\bf 80} (2018), no. 2, 481--522.
\bibitem{SZ} Sukochev F., Zanin D. \textit{ The Connes character formula for locally compact spectral triples}, Ast\'erisque \textbf{445} (2023); 150 pp.
\bibitem{SZ2} Sukochev F., Zanin D. \textit{Optimal constants in non-commutative H\"older inequality for quasi-norms}. Proc. Amer. Math. Soc. \textbf{149} (2021), no. 9, 3813--3817.
\bibitem{Tr} Triebel H. \textit{Theory of function spaces}. Monographs in Mathematics, 78. Birkh\"auser Verlag, Basel, 1983.
\bibitem{Uc} Uchiyama A. \textit{On the compactness of operators of Hankel type}. Tohoku Math. J. (2) \textbf{30} (1978), no. 1, 163--171.
\bibitem{vN} van Neerven J. \textit{Functional Analysis}. Cambridge Studies in Advanced Mathematics, 201. Cambridge University Press, 2023.
\bibitem{Yo} Youssfi A. \textit{Regularity properties of commutators and $BMO$-Triebel--Lizorkin spaces}. Ann. Inst. Fourier (Grenoble) \textbf{45} (1995), no. 3, 795--807.
\end{thebibliography}
\end{document}